\documentclass[a4paper]{amsart}

\usepackage[frenchstyle, oldstyle]{kpfonts}
\usepackage[utf8]{inputenc}
\usepackage[T1]{fontenc}
\usepackage[english]{babel}
\usepackage{amsfonts}
\usepackage{amsthm}
\usepackage{amssymb}
\usepackage{mathcomp}
\usepackage{mathrsfs}
\usepackage[all]{xy}
\usepackage{graphicx}
\usepackage{epstopdf}
\usepackage[centertags]{amsmath}
\usepackage{vmargin}
\usepackage{tikz}
\usepackage{dsfont}
\usepackage{cases}

\usepackage[colorlinks,final,hyperindex]{hyperref}
\hypersetup{citecolor=black, linkcolor=black}

\theoremstyle{plain}
\newtheorem{thm}{Theorem}
\newtheorem*{thmappendix*}{Theorem 8}
\newtheorem*{thmintro}{Theorem 10}
\newtheorem*{thm*}{Theorem}
\newtheorem{lemma}{Lemma}
\newtheorem{prop}{Proposition}
\newtheorem{cor}{Corollary}

\theoremstyle{definition}
\newtheorem{defin}{Definition}
\newtheorem*{nota}{\sc Notations}

\theoremstyle{remark}
\newtheorem{rmk}{\sc Remark}
\newtheorem{eg}{\sc Example}

\newcommand{\ac}{{\scriptstyle \text{\rm !`}}}

\setpapersize{A4}
             
\newcommand{\mbs}{\mathbb{S}}
\newcommand{\mbk}{\mathbb{K}}
\newcommand{\mbn}{\mathbb{N}}

\newcommand{\PP}{\mathcal{P}}

\newcommand{\HOM}{\mathrm{HOM}}

\newcommand{\ov}{\overline}

\newcommand{\Aa}{\mathcal{A}}
\newcommand{\BB}{\mathcal{B}}
\newcommand{\CC}{\mathcal{C}}
\newcommand{\DD}{\mathcal{D}}
\newcommand{\EE}{\mathcal{E}}
\newcommand{\FF}{\mathcal{F}}
\newcommand{\II}{\mathcal{I}}
\newcommand{\RR}{\mathcal{R}}
\newcommand{\UU}{\mathcal{U}}
\newcommand{\VV}{\mathcal{V}}
\newcommand{\WW}{\mathcal{W}}
\newcommand{\AAA}{\mathscr{A}}
\newcommand{\BBB}{\mathscr{B}}
\newcommand{\CCC}{\mathscr{C}}
\newcommand{\DDD}{\mathscr{D}}
\newcommand{\EEE}{\mathscr{E}}
\newcommand{\FFF}{\mathscr{F}}

\newcommand{\PPP}{\mathscr{P}}
\newcommand{\QQQ}{\mathscr{Q}}

\newcommand{\ra}{\rightarrow}

\newcommand{\uCog}{\mathsf{uCog}}
\newcommand{\cCog}{\mathsf{cCog}}

\newcommand{\cLieCog}{\mathsf{cLieCog}}

\newcommand{\uCocom}{\mathsf{uCocom}}

\newcommand{\uNilCocom}{\mathsf{uNilCocom}}
\newcommand{\Liealg}{\mathsf{Lie}-\mathsf{alg}}

\newcommand{\Tfree}{\mathbb{T}}

\newcommand{\Hinich}{\mathsf{Hinich} -\mathsf{cog}}

\newcommand{\Ccomod}{\mathscr C -\mathsf{comod}}
\newcommand{\Pmod}{\mathscr P -\mathsf{mod}}

\newcommand{\Operad}{\mathsf{Operad}}
\newcommand{\cCoop}{\mathsf{cCoop}}

\newcommand{\sSet}{\mathsf{sSet}}
\newcommand{\dgMod}{\mathsf{dgMod}}
\newcommand{\gMod}{\mathsf{gMod}}
\newcommand{\catC}{\mathsf{C}}
\newcommand{\catD}{\mathsf{D}}
\newcommand{\catE}{\mathsf{E}}

\newcommand{\Palg}{\mathscr{P}-\mathsf{alg}}
\newcommand{\Qalg}{\mathscr{Q}-\mathsf{alg}}
\newcommand{\Ccog}{\mathscr{C}-\mathsf{cog}}
\newcommand{\Dcog}{\mathscr{D}-\mathsf{cog}}

\newcommand{\Artinalg}{\mathsf{Artin}-\mathsf{alg}}

\newcommand{\colim}{\mathrm{colim}}
\newcommand{\Map}{\mathrm{Map}}
\newcommand{\Def}{\mathrm{Def}}

\newcommand{\Tw}{ Tw}
\newcommand{\End}{\mathrm{End}}
\newcommand{\Lie}{\mathscr{L}ie}

\newcommand{\uAs}{u\mathscr{A}s}

\newcommand{\Com}{\mathscr{C}om}
\newcommand{\uCom}{u\mathscr{C}om}
\newcommand{\itemt}{\item[$\triangleright$]}

\title{Homotopy theory of unital algebras}
\author{Brice Le Grignou}
\address{University of Utrecht}
\email{b.n.legrignou@uu.nl}

\date{\today}



\thanks{The author was supported by the ANR SAT grant.}

\begin{document}

\maketitle

\begin{abstract}
This paper provides an extensive study of the homotopy theory of types of algebras with units, like unital associative algebras or unital commutative algebras for instance. To this purpose, we endow the Koszul dual category of curved coalgebras, where the notion of quasi-isomorphism barely makes sense, with a model category structure Quillen equivalent to that of unital algebras. To prove such a result, we use recent methods based on presentable categories. This allows us to describe the homotopy properties of unital algebras in a simpler and richer way. Moreover, we endow the various model categories with several enrichments which induce suitable models for the mapping spaces and describe the formal deformations of morphisms of algebras. \end{abstract}

\setcounter{tocdepth}{1}
\tableofcontents

\section*{Introduction}

Among the various types of algebras, some of them include units like the ubiquitous unital associative algebras and unital commutative algebras or the unital Batalin--Vilkovisky algebras which arose in Mathematical Physics. When working with a chain complex carrying such an algebraic structure, like the de Rham algebra of differential manifolds for instance, one would like to understand the homotopical properties that this algebraic data satisfies with the underlying differential map. The purpose of the present paper is develop a framework which allows one to prove the homotopical properties carried by types of algebras with units; that is, their property up to quasi-isomorphisms.\\
 
In order to work with types of algebras in a general way, one needs a precise notion which encodes these ones. This is achieved by the concept of an operad. Operads are generalizations of associative algebras which encode some types of algebras (associative, commutative, Lie, Batalin-Vilkovisky, \ldots) in a way that a representation of an operad $\PPP$ is a chain complex together with a structure of algebra of the type encoded by $\PPP$.\\

Besides, one of the most common and powerful tool to study homotopical algebra, that is to study categories with a notion of equivalences, is the model category structure introduced by Daniel Quillen which makes the manipulation of weak equivalences easier by means of other maps called respectively cofibrations and fibrations. Hinich proved in the article \cite{Hinich97}, that the category of algebras over an operad carries a model structure whose weak equivalences are quasi-isomorphisms and whose fibrations are surjections. In a purely theoretical perspective, this model structure describes all the homotopical data of this category. However, the cofibrant objects are not easy to handle; they are the retracts of free algebras whose generators carries a particular filtration.\\

Hinich (\cite{Hinich01}) embedded the category of differential graded (dg) Lie algebras into the category of dg cocommutative coalgebras. From the model structure of the category of dg Lie algebras he obtained a model structure on the category of dg cocommutative coalgebras which is Quillen equivalent to the first one. In this new model category, any object is cofibrant. Moreover, this context allows one to build structures and morphisms using obstruction methods. So this new context of dg cocommutative coalgebras is more suitable to study the homotopy theory of dg Lie algebras than the category of dg Lie algebras itself. In a similar perspective, Lefevre-Hasegawa embedded the category of nonunital dg associative algebras into the category dg coassociative coalgebras shown to be Quillen equivalent to the first one; see \cite{LefevreHasegawa03}. Vallette generalized these results to all types of algebras encoded by any operad satisfying a technical condition: it is an augmented operad. Augmented operads are related to the dual notion of conilpotent cooperads by an adjunction called the operadic bar-cobar adjunction $\Omega \dashv B$. Vallette embedded the category of algebras over an augmented operad $\PPP$ into category of coalgebras over a cooperad $\PPP^\ac$ called the Koszul dual of $\PPP$. He transferred the model structure on the category of $\PPP$-algebras to the category of $\PPP^\ac$-coalgebras and got again a Quillen equivalence between these two model categories.\\

However the operads describing types algebras with units do not satisfy the technical condition to be augmented. To extend the result of Vallette to categories of algebras over any operad, one first needs to modify the operadic bar-cobar adjunction. Inspired by the work of Hirsh and Mill\`es  in \cite{HirshMilles12}, we introduce an adjunction \`a la bar-cobar relating dg operads to curved conilpotent cooperads.
 $$
\xymatrix{\text{Curved conilpotent cooperads} \ar@<1ex>[r]^(0.67){\Omega_u} & \text{dg Operads} \ar@<1ex>[l]^(0.33){B_c}}
$$
Moreover, any morphism of dg operads $f$ from a cobar construction $\Omega_u \CCC$ of a curved conilpotent cooperad $\CCC$ to an operad $\PPP$ comes equipped with an adjunction $\Omega_f \dashv B_f$ relating $\PPP$-algebras to $\CCC$-coalgebras.
 $$
\xymatrix{\CCC-\text{coalgebras} \ar@<1ex>[r]^(0.54){\Omega_f} & \PPP-\text{algebras} \ar@<1ex>[l]^(0.46){B_f}}
$$
The model structure of $\PPP$-algebras can be transferred to the category of $\CCC$-coalgebras along this adjunction.

\begin{thmintro}
 Let $\alpha: \CCC \to  \PPP$ be an operadic twisting morphism and let $\Omega_\alpha \dashv B_\alpha $ be the bar-cobar adjunction between $\PPP$-algebras and $\CCC$-coalgebras induced by $\alpha$. There exists a model structure on the category of $\CCC$-coalgebras whose cofibrations (resp. weak equivalences) are morphisms whose image under $\Omega_\alpha$ is a cofibration (resp. weak equivalence). With this model category structure,  the adjunction $\Omega_\alpha \dashv B_\alpha $ is a Quillen adjunction.
\end{thmintro}

To prove this theorem, we use new techniques coming from category theory. Specifically, we utilize a Theorem of \cite{BHKKRS14} involving presentable categories.\\

We study in details the particular case where the morphism of operads $f$ from $\Omega_u \CCC$ to $\PPP$ is a quasi-isomorphism; for instance if $f$ is the identity $\iota$ of $\Omega_u \CCC$. In this case, the Quillen adjunction $\Omega_\iota \dashv B_\iota$ is a Quillen equivalence. We show that the fibrant $\CCC$-coalgebras are the images of the $\Omega_u \CCC$-algebras under the functor $B_\iota$. So, switching from the category of $\Omega_u \CCC$-algebras to the category of $\CCC$-coalgebras by the functor $B_\iota$ amounts to introduce new morphisms between $\Omega_u \CCC$-algebras. These new morphisms can be built using obstruction methods. Moreover, any $\Omega_u \CCC$-algebra becomes cofibrant in this new context.\\

This article also deals with enrichments of the category of $\PPP$-algebras for any differential graded operad $\PPP$, and of the category of $\CCC$-coalgebras for any curved cooperad $\CCC$. These two categories are enriched in simplicial sets in a way that recovers the mapping spaces. Moreover, they are tensored, cotensored and enriched in cocommutative coalgebras. These cocommutative coalgebras encode the formal deformations of morphisms of algebras over an operad. In the context of nonsymmetric operads and nonsymmetric cooperads, this enrichment can be extended to all coassociative coalgebras. These coassociative coalgebras encode in single objects both the mapping spaces and the deformation of morphisms.\\

Finally, we apply the framework developed here to concrete operads like the operad $\uAs$ of unital associative algebras and the operad $\uCom$ of unital commutative algebras. For these two operads, the process of curved Koszul duality developed in \cite{HirshMilles12} relates the curved cooperads $\uAs^\ac$ and $\uCom^\ac$ to respectively the operads $\uAs$ and $\uCom$.

\subsection*{Layout}

The article is organized as follows. In the first part, we recall several notions about category theory, and homological algebra. In the second part, we recall the notions of operads, cooperads, algebras over an operad and coalgebras over a cooperad. We also prove some results, as the presentability of the category of coalgebras over a curved cooperad, that we will need in the sequel. The third part deals with enrichments of the category of algebras over an operad and of the category of coalgebras over a curved cooperad; specifically, we study enrichments over simplicial sets, cocommutative coalgebras and coassociative coalgebras. In the fourth part, we introduce an adjunction \`a la bar-cobar between operads and curved cooperads related to a notion of twisting morphism. We use it to define an adjunction between $\PPP$-algebras and $\CCC$-coalgebras for a twisting morphism from a curved cooperad $\CCC$ to an operad $\PPP$. In the fifth section, we recall the projective model structure on the category of algebras over an operad. We describe models for the mapping spaces and we show that the enrichment over cocommutative coalgebras encodes deformations of morphisms. The sixth part transfers the projective model structure on $\PPP$-algebras along the previous adjunction to obtain a model structure on $\CCC$-coalgebras and a Quillen adjunction. The seventh part deals with these model structures in the case where the operad $\PPP$ is the cobar construction $\Omega_u \CCC$ of $\CCC$. In particular, the adjunction induced is a Quillen equivalence. Finally, in the eighth part, we apply the formalism developed in the previous sections to study the examples of unital associative algebras and unital commutative algebras. 

\subsection*{Acknowledgements} This article is the second part of my PhD thesis. I would like to thank my advisor Bruno Vallette for his precious advice and careful review of this paper. I also would like to thank Damien Calaque and Kathryn Hess for reviewing my thesis. Finally, the Laboratory J.A. Dieudonn\'e in the University of Nice provided excellent working conditions.

\subsection*{Conventions and notations}

\begin{itemize}
 \itemt We work over a field $\mbk$.
 \itemt The category of graded $\mbk$-modules is denoted $\gMod$. The category of chain complexes is denoted $\dgMod$. They are endowed with their usual closed symmetric monoidal structures. The internal hom is denoted $[\ ,\ ]$. The category of chain complexes is also endowed with its projective model structure where the weak equivalences are the quasi-isomorphisms and where the fibrations are the degreewise surjections. The degree of an homogeneous element $x$ of a graded $\mbk$-module or a chain complex is denoted $|x|$.
 \itemt For any integer $n$, let $D^n$ be the chain complex generated by one element in degree $n$ and its boundary in degree $n-1$. Let $S^n$ be the chain complex generated by a cycle in degree $n$.
 \itemt The category of simplicial set is denoted $\sSet$. It is endowed with its Kan--Quillen model structure; see \cite[I.11.3]{GoerssJardine09}.
  \itemt The following type of diagram
$$
\xymatrix{\catC \ar@<1ex>[r]^(0.5){L} & \catD \ar@<1ex>[l]^(0.5){R}}
$$
means that the functor $R$ is right adjoint to the functor $L$.
\itemt For any graded $\mbk$-module $\VV$ endowed with a filtration $(F_n \VV)_{n \in \mbn}$, the graded complex associated to this filtration is denoted $G\VV$. In other words,
 $$
 G\VV = \bigoplus_n G_n \VV
 $$
 where $G_n \VV= F_n \VV / F_{n-1} \VV$. If $\VV$ is a chain complex such that $(F_n \VV)_{n \in \mbn}$ is a filtration of chain complexes, that is $d(F_n \VV)\subset F_n \VV$ for any integer $n$, then $G\VV$ inherits a structure of chain complex.
\end{itemize}
\vspace{1cm}

\section{Preliminaries}

In this first section, we recall some categorical concepts like the presentability and the notions of enrichment, tensoring and cotensoring. Moreover, we describe several notions of coalgebras like coassociative coalgebras and cocommutative coalgebras that have been extensively studied respectively in \cite{GetzlerGoerss99} and in \cite{Hinich01}. More specifically the category of coassociative coalgebras admits a model structure related by a Quillen adjunction to the category of simplicial sets; the category of conilpotent cocommutative coalgebras admits a model structure Quillen equivalent to the projective model structure on Lie algebras. Finally, we describe the Sullivan polynomial algebras.

\subsection{Presentable categories}

\begin{defin}[Presentable category]
  Let $\catC$ be a cocomplete category. An object $X$ of $\catC$ is called \textit{compact} if for any filtered diagram $F: I \ra \catC$ the map $\colim (\hom_\catC (X,F)) \ra \hom_\catC (X ,\colim F)$ is an isomorphism. The category $\catC$ is said to be \textit{presentable} if there exists a set of compact objects such that any object of $\catC$ is the colimit of a filtered diagram involving only these compact objects. 
\end{defin}
 
 The following proposition is a classical result of category theory.
 
\begin{prop}\cite{AdamekRosicky94}\label{prop:presentfolk}
 A functor $L: \catC \ra \catD$ between presentable categories is a left adjoint if and only if it preserves colimits.
\end{prop}
   
  \subsection{Tensoring, cotensoring and enrichment}

In this section, we recall the definition of tensored-cotensored-enriched category over a monoidal category. See \cite{Borceux94} for the original reference.

\begin{defin}[Action, coaction]\label{defin:tensored}
  Let $(\catE, \otimes , \II)$ be a monoidal category and let $\catC$ be a category.
\begin{itemize}
\itemt An enrichment of $\catC$ over $\catE$ is a bifunctor $[-,-]: \catC^{op} \times \catC \ra \catE$ together with functorial morphisms
$$
\begin{cases}
 \gamma_{X,Y,Z}: [Y,Z]\otimes[X,Y] \ra [X,Z]\ ,\\
 \upsilon_X : \II \ra [X,X]  
\end{cases}
$$
for any objects $X$, $Y$ and $Z$ of $\catC$ and which are composition and unit in terms of the following commutative diagrams.
$$
\xymatrix{[Y,Z]\otimes[X,Y]\otimes[V,X] \ar[r]^(0.58){\gamma_{X,Y,Z} \otimes Id} \ar[d]_{Id \otimes \gamma_{V,X,Y} \otimes Id} & [X,Z]\otimes[V,X] \ar[d]^{\gamma_{V,X,Z}}\\
[Y,Z]\otimes[V,Y] \ar[r]_{\gamma_{V,Y,Z}} & [V,Z]}
$$
$$
\xymatrix{[X,Y]\otimes[X,X]\ar[rd] & [X,Y] \ar[r]^(0.4){\upsilon_Y \otimes Id} \ar[l]_(0.4){Id \otimes \upsilon_X} \ar[d]^{Id} & [Y,Y]\otimes[X,Y] \ar[ld]\\
&[X,Y]}
$$
 \itemt A right action of $\catE$ on $\catC$ is a functor 
 $$
 - \otimes - : \catC \times \catE \ra \catC
 $$
 together with functorial isomorphisms
 $$
 \begin{cases}
 X \otimes (\Aa \otimes \BB)  \simeq (X \otimes \Aa ) \otimes \BB \ ,\\
X \otimes \II  \simeq X\ ,
\end{cases}
 $$
 for any $X \in \catC$, any $\Aa, \BB \in \catE$; these functors are compatible with the monoidal structure of $\catE$ in terms of the following commutative diagrams.
 $$
 \xymatrix{\big( (X \otimes \Aa) \otimes \BB \big) \otimes \CC \ar[r] \ar[d] & \big(X \otimes (\Aa \otimes \BB )\big) \otimes \CC  \ar[r] & X \otimes \big( (\Aa \otimes \BB ) \otimes \CC \big)  \ar[d]\\     
   (X \otimes \Aa) \otimes  (\BB  \otimes \CC ) \ar[rr] &&  X \otimes \big( \Aa \otimes ( \BB  \otimes \CC ) \big)}
   $$
   
   $$
 \xymatrix{ (X \otimes \II) \otimes \Aa  \ar[rd] \ar[rr] && X \otimes (\II \otimes \Aa ) \ar[ld] \\
 & X \otimes \Aa}
 $$
 \itemt A left coaction of $\catE$ on $\catC$ is a functor:
 $$
 \langle -, - \rangle: \catE^{op} \times \catC \ra \catC
 $$
 together with functorial isomorphisms
 $$
 \begin{cases}
 \langle \Aa \otimes \BB, X \rangle \simeq \langle \Aa \langle \BB, X \rangle \rangle\ ,\\
\langle \II, X \rangle  \simeq X\ .
\end{cases}
 $$
which satisfy the commutative duals of the diagrams above.
\end{itemize}

\end{defin}
   
\begin{defin}[Category tensored-cotensored-enriched over a monoidal category]\label{def:tce}
  Let $\catE$ be a monoidal category and let $\catC$ be a category. We say that $\catC$ is \textit{tensored-cotensored-enriched} over $\catE$ if there exists three functors:
  $$
\begin{cases}
 \{-,-\}: \catC^{op} \times \catC \ra \catE\\
 - \triangleleft - : \catC \times \catE \ra \catC\\
 \langle -, - \rangle: \catE^{op} \times \catC \ra \catC
\end{cases}
  $$
 together with functorial isomorphisms
 $$
 \hom_\catC (X \triangleleft \Aa , Y) \simeq \hom_\catE (\Aa, \{X,Y\}) \simeq \hom_\catC (X, \langle \Aa,  Y \rangle)\ ,
 $$
 for any $X,Y \in \catC$, any $\Aa, \BB \in \catE$ and where $\II$ is the monoidal unit of $\catE$, such that $ - \triangleleft - $ defines a right action of $\catE$ on $\catC$.
\end{defin}

The axioms and terminology of these notions are justified by the following proposition.

\begin{prop}\label{prop:tensenrich}
If the category $\catC$ is tensored-cotensored-enriched over $\catE$, then, it is enriched in the usual sense and the functor $\langle -, - \rangle$ is a left coaction in the sense of Definition \ref{defin:tensored}.
\end{prop}

\begin{proof}[Proof]
Suppose that the category $\catC$ is tensored-cotensored-enriched over $\catE$. On the one hand, let us define the composition relative to the enrichment $\{-,-\}$. For any object $X,Y$ of $\catC$, the identity morphism of $\{X,Y\}$ defines a morphism $X \triangleleft \{X,Y\} \ra Y$. So for any objects $X$, $Y$, $Z$, we have a map
 $$
 X \triangleleft (\{X,Y\} \otimes  \{Y,Z\}) \simeq (X \triangleleft \{X,Y\}) \triangleleft  \{Y,Z\} \ra Y\triangleleft \{Y,Z\} \ra Z
 $$
 and hence a map $\{X,Y\} \otimes  \{Y,Z\} \ra \{X,Z\}$. So is defined the composition. The coherence diagrams of Definition \ref{defin:tensored} ensure us that the composition is associative and gives us a unit. On the other hand, let us show that the functor $\langle -, - \rangle$ is a left coaction. For any $X,Y \in \catC$ and any $\Aa, \BB \in \catE$, we have functorial isomorphisms:
 
\begin{align*}
 \hom_\catC (X, \langle \Aa \otimes \BB , Y \rangle) &\simeq \hom_\catC (X \triangleleft ( \Aa \otimes \BB), Y ) \simeq \hom_\catC ( (X \triangleleft  \Aa) \triangleleft \BB, Y )\\ & \simeq \hom_\catC (X \triangleleft \Aa, \langle  \BB , Y \rangle) \simeq \hom_\catC (X , \langle \Aa \langle  \BB , Y \rangle \rangle)\ .
\end{align*}
By the Yoneda lemma, this gives us a functorial isomorphism $\langle \Aa \otimes \BB , Y \rangle \simeq \langle \Aa \langle  \BB , Y \rangle \rangle$. This functorial isomorphism satisfies the coherence conditions of Definition \ref{defin:tensored} because the functorial isomorphism $X \triangleleft ( \Aa \otimes \BB) \simeq (X \triangleleft  \Aa) \triangleleft \BB$ satisfies the coherence conditions of the same definition.

\end{proof}

\begin{prop}\label{prop:presenttensor}
 Let $\catE$ be a presentable monoidal category and let $\catC$ be a presentable category. 
 
\begin{itemize}
 \itemt Suppose that there exists a right action $- \triangleleft -$ of $\catE$ on $\catC$ and that for any $\Aa \in \catE$ and for any $X \in \catC$, the functors $X \triangleleft -: \catE \ra \catC$ and $- \triangleleft \Aa : \catC \ra \catC$ preserve colimits. Then, $\catC$ is tensored-cotensored-enriched over $\catE$.
 \itemt Suppose that there exists a left coaction $\langle -,  - \rangle$ of $\catE$ on $\catC$ and that there exists a functor 
 $$
 - \triangleleft - : \catC \times \catE \ra \catC
 $$
 together with a functorial isomorphism
 $$
 \hom_\catC (X \triangleleft \Aa , Y) \simeq  \hom_\catC (X, \langle \Aa,  Y \rangle)\ .
 $$
 Suppose moreover that the functor $\langle -,  Y \rangle: \catE^{op} \ra \catC$ sends colimits in $\catE$ to limits. Then, $\catC$ is tensored-cotensored-enriched over $\catE$.
\end{itemize}

\end{prop}

\begin{proof}[Proof]
 The first point is a direct consequence of Proposition \ref{prop:presentfolk}. Let us prove the second point. Since, $\catE$ left coacts on $\catC$, by the same arguments as in the proof of Proposition \ref{prop:tensenrich}, we can show that the bifunctor $- \triangleleft -$ is a right action of $\catE$ on $\catC$. Moreover, since the functors $\langle -,  Y \rangle$ preserve limits, then any functor of the form $X \triangleleft -$ preserves colimits. The result is then a direct consequence of the first point. 
\end{proof}

\begin{defin}[Homotopical enrichment]\label{defin:almosthomotop}\leavevmode
Let $\mathsf M$ be a model category and let $\catE$ be a model category with a monoidal structure. We say that  $\mathsf M$ is homotopically enriched over $\catE$ if it enriched over $\catE$ and if for any cofibration $f: X \ra X'$ in $\mathsf M$ and any fibration $g : Y \ra Y'$ in $\mathsf M$, the morphism in $\catE$:
 $$
 \{X',Y\} \ra \{X',Y\} \times_{\{X,Y'\}} \{X,Y\}
 $$ 
 is a fibration. Moreover, we require this morphism to be a weak equivalence whenever $f$ or $g$ is a weak equivalence.
\end{defin}

\subsection{Coalgebras}

\begin{defin}[Coalgebras]\leavevmode
A \textit{coassociative coalgebra} $\CCC=(\CC,\Delta ,\epsilon)$ is a chain complex $\CC$ equipped with a coassociative coproduct $\Delta : \CC \to \CC \otimes \CC$ and a counit $\epsilon : \CC \to \mbk$ such that $Id_\CC = (Id_\CC \otimes \epsilon) \Delta = (\epsilon \otimes Id_\CC) \Delta$. The kernel of the map $\epsilon$ is denoted $\ov \CC$. The coalgebra $\CCC$ is said cocommutative if $\Delta = \tau \Delta$ where
$$
\tau (x\otimes y) = (-1)^{|x||y|} y \otimes x\ .
$$
A \textit{graded} atom is a nonzero element $1 \in \CC$ such that $\Delta 1 = 1 \otimes 1$. In this context, let us define the map $\ov \Delta: \ov \CC \to \ov \CC \otimes  \ov \CC$ as follows:
$$
\ov \Delta x := \Delta x - 1 \otimes x - x \otimes 1  \in \ov \CC \otimes \ov \CC\ .
$$
A graded atom $1$ is called a \textit{dg atom} if $d1=0$. A \textit{conilpotent coalgebra} $\CCC=(\CC,\Delta,\epsilon, 1)$ is the data of a coassociative coalgebra $(\CC,\Delta,\epsilon)$ together with a graded atom such that, for any $x \in \ov \CC$, there exists an integer $n$ such that
$$
\ov\Delta^n x := (Id_\CC^{\otimes n-1} \otimes \ov\Delta)  \cdots (Id_\CC \otimes \ov\Delta) \ov\Delta (x) =0\ .
$$
A conilpotent cocommutative coalgebra $\CCC$ is said to be a \textit{Hinich coalgebra} if $1$ is a dg atom. We denote by $\uCog$ be the category of coassociative coalgebras and by $\uCocom$ the category of cocommutative coalgebras. Let $\uNilCocom$ (resp. $\Hinich$) be the full subcategory of $\uCocom$ made up of conilpotent cocommutative coalgebras (resp. Hinich coalgebras).
\end{defin}

Any conilpotent coalgebra $\CCC$ has a canonical filtration called the coradical filtration
$$
F_{n}^{rad}\CC := \mbk \cdot 1 \oplus \{x \in \ov\CC |\ \ov\Delta^{n+1} x =0\}\ ,
$$
which is not necessarily stable under the codifferential $d$.

\begin{prop}
Let $f$ be a morphism of coalgebras between two conilpotent coalgebras $\CCC=(\CC,\Delta,\epsilon, 1)$ and $\DDD=(\DD,\Delta',\epsilon', 1')$. Then, $f(1)=1'$.
\end{prop}

\begin{proof}[Proof]
Let $x \in \ov\DD$ such that $f(1)= 1' + x$. Since $\Delta f(1)=(f \otimes f) \Delta (1)$, then $\ov \Delta x = x \otimes x$. Since there exits an integer $n$ such that $\ov\Delta^n(x)= x \otimes \cdots \otimes x=0$, then $x=0$.
\end{proof}

\begin{prop}
 The categories $\uCog$, $\uCocom$ and $\uNilCocom$ and $\Hinich$ are presentable. The forgetful functor from $\uCog$ to the category of chain complexes has a right adjoint called the cofree counital coalgebra functor. The same statement holds for the category $\uCocom$. The functor $\CCC \mapsto \ov \CC$ from the category $\Hinich$ to the category of chain complexes has a right adjoint. The tensor product of the category of chain complexes induces closed symmetric monoidal structures on the categories $\uCog$ and $\uCocom$.
\end{prop}

\begin{proof}[Proof]
 The results are proven in \cite[2.1, 2.2,2.5]{AnelJoyal13} for the category $\uCog$. The methods used apply mutatis mutandis for the other categories.
 \end{proof}

\begin{thm}[\cite{GetzlerGoerss99}]
 The full sub category $\uCog^{\geq 0}$ of $\uCog$ made up of nonnegatively graded coalgebras admits a model structure whose cofibrations are the monomorphisms and whose weak equivalences are the quasi-isomorphisms.
\end{thm}

The category $\Hinich$ is related to the category of Lie-algebras by an adjunction described in \cite{Quillen69}:
$$
\xymatrix{\Hinich \ar@<1ex>[r]^(0.54){\mathcal{L}} & \Liealg\ . \ar@<1ex>[l]^(0.46){\mathcal{C}}}
$$

\begin{thm}\cite{Hinich01}
 There exists a model structure on the category $\Hinich$  whose cofibrations are monomorphisms and whose weak equivalences are morphisms whose image under the functor $\mathcal{L}$ is a quasi-isomorphism. The class of weak equivalences is contained in the class of quasi-isomorphisms. Moreover, the adjunction $\mathcal L \dashv \mathcal C$ is a Quillen equivalence when the category of Lie algebras is equipped with its projective model structure whose fibrations (resp. weak equivalences) are surjections (resp. quasi-isomorphisms) (see \cite{Hinich97}). 
\end{thm}

\begin{defin}[Deformation problems]
Let $\Artinalg$ be the category of nonpositively graded local finite dimensional dg commutative algebra. A deformation problem is a functor from the category $\Artinalg$ to the category of simplicial sets.
\end{defin}

Lurie showed in \cite{Lurie11} that a suitable infinity-category of deformation problems (called formal moduli problems), is equivalent to the infinity-category of Lie algebras if the characteristic of the base field $\mbk$ is zero. Therefore, it is equivalent to the infinity-category of Hinich coalgebras. In that perspective, any Hinich coalgebra $\CCC$ induces a deformation problem as follows:
$$
R \in \Artinalg \mapsto \Map_{\Hinich}(R^*, \CCC)\ .
$$

\begin{rmk}
 We use the definition of Hinich of a deformation problem given in \cite{Hinich01}. We do not describe here the homotopy theory of such deformation problems nor a precise link with the work of Lurie who uses the framework of quasi categories (see \cite{Lurie11}). In the sequel, we will only use the fact that, for any morphism of deformation problems $f: X \to Y$, if $f(R)$ is a weak equivalence of simplicial sets for any algebra $R \in \Artinalg$, then $f$ is an equivalence of deformation problems.  
\end{rmk}

\subsection{Coalgebras and simplicial sets}

In this subsection, we describe a Quillen adjunction between the category of simplicial sets and the category of coassociative coalgebras. This adjunction is part of the Dold--Kan correspondence. From a simplicial set $X$, one can produce a chain complex $DK(X)$ called the normalized Moore complex. In degree $n$, $DK(X)_n$ is the sub-vector space of $\mbk \cdot X_n$ which is the intersection of the kernels of the faces $d_0$, \ldots, $d_{n-1}$. The differential is $(-1)^n d_n$. Moreover, the Alexander-Whitney map makes the functor $DK$ comonoidal. Then, the diagonal map $X \ra X \times X$ gives to $DK(X)$ a structure of coalgebras. Thus, we have a functor $DK^c$ from simplicial sets to the category $\uCog$ of coassociative coalgebras. This functor $DK^c$ admits a right adjoint $N$ defined by
$$
N(\CC)_n:=\hom_{\uCog}(DK^c(\Delta[n]),\CC)\ .
$$
Actually, we have the following sequence of adjunctions,
$$
\xymatrix{\sSet \ar@<1ex>[r]^(0.46){DK^c} & \uCog^{\geq 0} \ar@<1ex>[l]^(0.55){N} \ar@<1ex>[r]^(0.54){in} & \uCog \ar@<1ex>[l]^(0.46){tr} }
$$
where $in$ is the embedding of $\uCog^{\geq 0}$ into $\uCog$ and where $tr$ is the truncation.

\begin{prop}
The above adjunction between $\uCog^{\geq 0}$ and $\sSet$ is a Quillen adjunction.
\end{prop}

\begin{proof}[Proof]
 The functor $DK^c$ carries monomorphisms to monomorphisms and weak homotopy equivalences to quasi-isomorphisms; see \cite[III.2]{GoerssJardine09}.
\end{proof}

\subsection{The Sullivan algebras of polynomial forms on standard simplicies}

\begin{defin}[Sullivan polynomial algebras]\cite{Sullivan77}
 For any integer $n \in \mbn$, the $n^{th}$ \textit{algebra of polynomial forms} is the following differential graded unital commutative algebra: 
 $$
 \Omega_n:= \mbk [t_0, \ldots , t_n, dt_0, \ldots , dt_n]/(\Sigma t_i =1)
 $$
 where the degree of $t_i$ is zero and where $d_{\Omega_n} (t_i)=dt_i$. In particular, $\sum dt_i =0$.
 \end{defin}
  Any map of finite ordinals $\phi : [n] \ra [m]$ defines a morphism of differential graded unital commutative algebra:
\begin{align*}
 \Omega(\phi) :\Omega_{m} &\ra \Omega_n\\
 t_i & \mapsto \sum_{\phi(j)=i} t_j \ .
\end{align*}
Therefore, the collection $\{\Omega_n\}_{n \in \mbn}$ defines a simplicial differential graded commutative algebra. Moreover, one can extend this construction to a contravariant functor $\Omega_\bullet$ from simplicial sets to differential graded unital commutative algebras such that $\Omega_{\Delta[n]}=\Omega_n$. This functor is part of an adjunction.
$$
\xymatrix{\sSet \ar@<1ex>[r]^(0.35){\Omega_\bullet} & \uCom - \mathsf{alg}^{op} \ar@<1ex>[l]}
$$

\begin{prop}\cite[8]{BousfieldGugenheim76}\label{prop:sullivanalgadj}
When the characteristic of the field $\mbk$ is zero, then the category $\uCom-\mathsf{alg}$ of differential graded unital commutative algebras admits a projective model structure where fibrations (resp. weak equivalences) are degreewise surjections (resp. quasi-isomorphisms). In that context, the adjunction between simplicial sets and $\uCom-\mathsf{alg}$ is a Quillen adjunction.
\end{prop}

\section{Operads, cooperads, algebras and coalgebras}

The purpose of this section is to recall the definitions of operads, cooperads, algebras over an operad and coalgebras over a cooperad that we will use in the sequel; we refer the reader to the book \cite{LodayVallette12}. Moreover, we prove that the category of coalgebras over a curved cooperad is presentable.

\subsection{Operads and cooperads}

We recall here the definitions of operads and cooperads. We refer to the book \cite{LodayVallette12} and to the article \cite{HirshMilles12}.

\begin{defin}[Symmetric modules]
  Let $\mbs$ be the groupoid whose objects are integers $n \in \mbn$ and whose morphisms are:
 $$
\begin{cases}
 \hom_\mbs(n,m)=  \emptyset \text{ if }n \neq m\\
  \hom_\mbs(n,n)= \mbs_n \text{ otherwise .}
\end{cases}
$$
 A graded $\mbs$-module (resp. dg $\mbs$-module) is a presheaf on $\mbs$ valued in the category of graded $\mbk$-modules (resp. chain complexes). The name $\mbs$-module will refer both to graded $\mbs$-modules and dg $\mbs$-modules. We say that a $\mbs$-module $\VV$ is reduced if $\VV(0)= \{0\}$.
\end{defin}

 The category of $\mbs$-modules has a monoidal structure which is as follows: for any $\mbs$-modules $\VV$ and $\WW$, and for any $n \geq 1$:

$$
(\VV \circ \WW) (n):=\bigoplus_{k \geq 1} \VV(k) \otimes_{\mbs_k} \left( \bigoplus_{ X_1  \sqcup \cdots \sqcup X_k = \{1, \ldots, n\} } \WW( \# X_1) \otimes \cdots \otimes \WW( \# X_k)  \right)
$$
where $\# X_i$ is the cardinal of the set $X_i$. For $n=0$,
$$
(\VV \circ \WW) (0):=\VV (0) \oplus \Big(\bigoplus_{k \geq 1 }   \VV(k) \otimes_{\mbs_k} \big( \WW(0) \otimes \cdots \otimes \WW(0)  \big)\Big)\ .
$$
 The monoidal unit is given by the $\mbs$-module $\II$ which is $\mbk$ in arity $1$ and $\{0\}$ in other arities.

\begin{nota}
 
\begin{itemize}\leavevmode
 \itemt For any dg $\mbs$-module $\VV$, we will denote by $\VV^{grad}$ the underlying graded $\mbs$-module.
  \itemt Let $f: \VV \ra \VV'$ and $g: \WW \ra \WW'$ and $h: \WW \ra \WW'$ be three morphisms of $\mbs$-modules. Then, we denote by $f \circ (g;h)$ the map from $\VV \circ \WW$ to $\VV' \circ \WW'$ defined as follows:
$$
f \circ (g;h) :=\sum_{i+j=n-1}  f \otimes_{\mbs_n} (g^{\otimes i} \otimes h \otimes g^{\otimes j})\ .
$$
In the case where $g$ is the identity, we use the notation $f \circ' h$.
$$
f \circ'h := f \circ (Id;h)\ .
$$
\itemt For any two graded $\mbs$-modules (resp. dg $\mbs$-modules) $\VV$ and $\WW$, we denote by $[\VV,\WW]$ the graded $\mbk$-module (resp. chain complex):
 $$
 [\VV,\WW]_n := \prod_{k \geq 0\atop l \in \mbn} \hom_{\mbk[\mbs_n]}(\VV(k)_l , \WW(k)_{l+n})\ .
 $$
 In that context morphisms of chain complex from $X$ to $[\VV,\WW]$ are in one-to-one correspondence with morphism of $\mbs$-modules from $X \otimes \VV$ to $\WW$.
\end{itemize}
\end{nota}

\begin{prop}\cite[6]{LodayVallette12}
 If the characteristic of the field $\mbk$ is zero, then the operadic Kunneth
$$
H(\VV \circ \WW) \simeq H(\VV) \circ H(\WW)\ ,
$$
holds for any dg $\mbs$-modules $\VV$ and $\WW$,where $H$ denotes the homology.
\end{prop}

\begin{defin}[Operads]
A graded operad $\PPP= (\PP,\gamma, 1)$ (resp. dg operad) is a monoid in the category of graded $\mbs$-modules (resp. dg $\mbs$-modules). We denote by $\Operad$ the category of dg operads.
\end{defin}

\begin{eg}
For any graded $\mbk$-module (resp. chain complex) $\VV$, $\End_\VV$ is the graded operad (resp. dg operad) defined by:
$$
\End_\VV (n) := \hom (\VV^{\otimes n}, \VV) \ .
$$
The composition in the operad $\End_\VV$ is given by the composition of morphisms of graded $\mbk$-modules (resp. chain complexes).
\end{eg}

A degree $k$ derivation $d$ on a graded operad $\PPP=(\PP, \gamma,1)$ is the data of degree $k$ maps $d:\PP(n) \ra \PP(n)$ which commute with the action of $\mbs_n$ and such that
$$
d\ \gamma = \gamma\  (d \circ Id + Id \circ' d) \ .
$$

\begin{prop}\cite[5]{LodayVallette12}
 The forgetful functor from operads to $\mbs$-modules has a left adjoint called the free operad functor and denoted $\Tfree$. For any $\mbs$-module $\VV$, $\Tfree \VV$ is the $\mbs$-module made up of trees whose vertices are filled with elements of $\VV$ with coherent arity. The composition is given by the grafting of trees. 
\end{prop}
 
There is a one-to-one correspondence between the degree $k$ derivation on the graded free operad $\Tfree \VV$ and the degree $k$ maps from $\VV$ to $\Tfree \VV$. Indeed, from such a map $u$ one can produce the derivation $D_u$ such that for any tree $T$ labeled by elements of $\VV$:
$$
D_u (T) := \sum_{v} Id \otimes \cdots \otimes u(v) \otimes \cdots \otimes Id\ . 
$$
where the sum is taken over the vertices of the tree $T$.

\begin{defin}[Cooperads]
 A cooperad $\CCC=(\CC, \Delta, \epsilon)$ is a comonoid in the category of $\mbs$-modules. We denote by $\ov \CC$ the kernel of the morphism $\epsilon : \CC \ra \II$. A cooperad $\CCC$ is said to be coaugmented if it is equipped with a morphism of cooperads $\II \ra \CCC$. In this case, we denote by $1$ the image of the unit of $\mbk$ into $\CC(1)$. A coaugmented cooperad $\CCC$ is said to be conilpotent if the process of successive decomposition stabilizes in finite time for any element. A precise definition is given in \cite[5.8.6]{LodayVallette12}.
 \end{defin}
 
The forgetful functor from conilpotent cooperads to $\mbs$-modules which sends $\CCC$ to $\ov \CC$ has a right adjoint sending $\VV$ to the tree module $\Tfree (\VV)$ with the decomposition given by the degrafting of trees. We denote it by $\Tfree^c (\VV)$. We also denote by $\delta : \CCC \ra \Tfree^c (\ov \CC )$ the counit of the adjunction. Any conilpotent cooperad is equipped with a filtration called the coradical filtration
$$
F^{rad}_n \CC (m) :=\{p \in \CC(m) | \delta (p) \in \Tfree^{\leq n} (\ov \CC) (m) \}
$$
where the symbol $ \Tfree^{\leq n}$ denotes the trees with at most $n$ vertices. In particular, $F^{rad}_0 \CC = \II$.
\begin{nota}
Let $\CC$ be coaugmented cooperad and $m$ be an integer. We denote by $\Delta_m$ the composite map
$$
\Delta_m : \xymatrix{\CC \ar[r]^{\Delta} &  \CC \circ \CC \ar[r] & \Tfree (\ov \CC) \ar@{->>}[r] & \Tfree^m (\ov \CC)\ .}
$$
\end{nota}
 A degree $k$ coderivation on a cooperad $\CCC=(\CC,\Delta,\epsilon)$ is a degree $k$ map $d$ of $\mbs$-modules from $\CC$ to $\CC$ such that 
$$
\Delta\  d =(d \circ Id + Id \circ' d)\Delta \ .
$$
 If the cooperad is coaugmented, we also require that $d(1)=0$. Let $\Tfree^c (\VV)$ be a cofree conilpotent cooperad. There is a one-to-one correspondence between degree $k$ coderivation on $\Tfree^c (\VV)$ and degree $k$ maps from $\ov \Tfree (\VV)$ to $\VV$. Indeed, such a map $u$ is uniquely extended by the following coderivation $D_u$ defined on any tree $T$ labeled by elements of $\VV$ as follows:
$$
D_u (T) := \sum_{T' \subset T}  Id \otimes \cdots \otimes u(T') \otimes \cdots \otimes Id\ ,
$$
where the sum is taken on the subtrees $T'$ of $T$.

\begin{defin}[Curved cooperads]
A curved cooperad $\CCC= (\CC,\Delta, \epsilon, 1 , d , \theta)$ is a coaugmented graded cooperad equipped with a degree $-2$ map $\theta : \CC(1) \ra \mbk$ and a degree $-1$ coderivation $d$ such that
\begin{align*}
& d^2 = (\theta \otimes Id -  Id \otimes \theta)  \Delta_2\ , \\
&\theta d =0\ .
\end{align*}
A morphism of curved cooperads is a morphism of cooperads $\phi: \CCC \ra \DDD$ which commutes with the coderivations and such that $\theta_\CCC = \theta_\DDD\phi$. We denote by $\cCoop$ the category of curved conilpotent cooperads.
\end{defin}

The coradical filtration of a conilpotent cooperad has the following property with respect to the decomposition map.

\begin{lemma}\label{lem:cooptech}
 Let $\CCC=(\CC,\Delta,\epsilon,1)$ be a conilpotent cooperad. Then
 $$
 \Delta (F^{rad}_n \CC) \subset \sum_{p_0 + \cdots + p_k \leq n} (F^{rad}_{p_0} \CC)(k) \otimes_{\mbs_k} (F^{rad}_{p_1} \CC \otimes \cdots \otimes F^{rad}_{p_k} \CC)\ .
 $$
\end{lemma}

\begin{proof}[Proof]
It suffices to prove the result for cofree cooperads. Indeed, any conilpotent cooperad $\CCC$ is equipped with a map $\delta: \CC \ra \Tfree^c (\ov \CC)$ such that $F_n^{rad}\CC = \delta^{-1}(F_n^{rad} \Tfree^c(\ov \CC))$.
\end{proof}

\begin{lemma}\label{lemma:curvcofree}
 Let $\CCC= \Tfree^c(\VV)$ be a cofree conilpotent graded cooperad equipped with a degree $-2$ map $\theta : \Tfree(\VV)(1) \twoheadrightarrow \VV (1) \ra \mbk$. Let  $\phi : \ov \Tfree \VV \ra \VV$ be a degree $-1$ map and let $D_\phi$ be the corresponding coderivation on $\CCC$. Then, the triple $(\Tfree^c \VV, D_\phi,\theta)$ is a curved cooperad if and only if $\phi$ satisfies the following equation:
 $$
 \phi D_\phi = (\theta \otimes \pi_\VV - \pi_\VV \otimes \theta )\Delta_2
 $$
where $\pi_\VV$ is the projection $\Tfree (\VV) \ra \VV$.
\end{lemma}

\begin{proof}[Proof]
If $(\Tfree^c \VV, D_\phi,\theta)$ is a curved cooperad, then $\phi D_\phi = \pi_\VV D^2_\phi= (\theta \otimes \pi_\VV - \pi_\VV \otimes \theta )\Delta_2$. Conversely, suppose that  $\phi D_\phi = (\theta \otimes \pi_\VV - \pi_\VV \otimes \theta )\Delta_2$. For any tree $T$ labeled by elements of $\VV$, one can prove that
 $$
 D^2_\phi (T) =\sum_{T' \subset T}  Id \otimes (\phi D_\phi(T'))  \otimes Id\ .
 $$
 Actually, it is the sum over every arity one vertex $v$ of
\begin{itemize}
 \itemt $\pm \theta (v)(T-v)$ if $v$ is the bottom vertex or a top vertex.
 \itemt $\pm (\theta(v)(T-v) -\theta(v)(T-v))=0$ otherwise.
\end{itemize}
Hence $(\Tfree^c \VV, D_\phi,\theta)$ is a curved cooperad.
\end{proof}

 There exists notions of $\mbn$-modules, nonsymmetric operads, nonsymmetric cooperads and their morphisms defined for instance in \cite[5.9]{LodayVallette12}. We will speak about the nonsymmetric context to refer to these ones. Notice that the operadic Kunneth formula holds in the nonsymmetric context without the assumption that the characteristic of the field $\mbk$ is zero.

\subsection{Algebras over an operad}
\begin{defin}[Algebras over an operad]
Let $\PPP=(\PP,\gamma,1)$ be an operad. A $\PPP$-\textit{module} $\AAA=(\Aa,\gamma_\AAA)$ is a left module in the category of $\mbs$-module, that is a $\mbs$-module $\Aa$ equipped with a map $\gamma_{\AAA}:\PP \circ \Aa \ra \Aa$ such that the following diagrams commute.
 $$
 \xymatrix{\PP \circ \PP \circ \Aa \ar[r]^(0,6){Id \circ \gamma_{\AAA}} \ar[d]_{\gamma_\AAA \circ Id} & \PP \circ \Aa \ar[d]^{\gamma_\AAA}\\ \PPP \circ \Aa \ar[r]_{\gamma_\AAA} & \Aa}
\hspace{1cm}
 \xymatrix{\II \circ \Aa \ar[r]^{1 \circ Id} \ar@/_2pc/[rr]^{Id} & \PPP \circ \Aa \ar[r]^{\gamma_\Aa} & \Aa}
 $$
 A morphism of $\PPP$-modules from $\AAA$ to $\BBB=(\BB,\gamma_\BBB)$ is a morphism of $\mbs$-modules $f:\Aa \ra \BB$ such that $\gamma_{\BBB} (Id \circ f)= f \gamma_{\AAA}$. A $\PPP$-\textit{algebra} is a $\PPP$-module $\Aa$ concentrated in arity $0$. We denote by $\Palg$ the category of $\PPP$-algebras. 
  \end{defin}

The forgetful functor from the category of $\PPP$-modules to the category of $\mbs$-modules has a left adjoint given by
 $$
 \VV \mapsto \PP \circ \VV
 $$
 The images of this left adjoint functor are called the free $\PPP$-modules.

\begin{defin}[Ideal]
   An \textit{ideal} of a $\PPP$-algebra $\AAA$ is a sub-chain complex $\BB \subset \Aa$ such that for any $p \in \PP(n)$ and $(x_i)_{i=1}^n \in \Aa$ ($n\geq 1$):
  $$
  \gamma_\AAA \big( p \otimes_{\mbs_n} ( x_1 \otimes \cdots \otimes x_n) \big) \in \BB
  $$
  whenever one of the $x_i$ is in $\BB$ (for $n \geq 1$). Then the quotient $\Aa/\BB$ has an induced structure of $\PPP$-algebra. 
\end{defin}

\begin{defin}[Derivation]
 Let $\PPP$ be a graded operad and let $\AAA$ be a $\PPP$-module. Suppose that the graded operad $\PPP$ is equipped with a degree $k$ derivation $d_\PPP$. Then, a \textit{derivation} of $\AAA$ is a degree $k$ map $d_\AAA$ from $\Aa$ to $\Aa$ such that
  $$
 d_\AAA\  \gamma_{\AAA} = \gamma_{ \AAA}\ \big( d_\PPP \circ Id_\AAA + Id \circ' d_\AAA \big)\ .
  $$
\end{defin}
 
Let $\PPP$ be a graded operad equipped with a degree $k$ derivation $d_\PPP$. There is a one-to-one correspondence between the derivations of a free $\PPP$-module $\AAA= \PP \circ \VV$ and the degree $k$ maps $\VV \ra \PP \circ \VV$. Indeed, any such map $u: \VV \ra \PP \circ \VV$ is uniquely extended by the derivation $D_u= d_\PPP \circ Id + D'$ on $\AAA$ where:
  $$
  D' := \xymatrix{\PP \circ \VV \ar[r]^(0.45){Id \circ u} & \PP \circ \PP \circ \VV \ar[r]^(0.55){\gamma \circ Id} & \PP \circ \VV } \ .
  $$

\subsection{Coalgebras over a cooperad}

\begin{defin}[Comodules and coalgebras over a cooperad]
 Let $\CCC= (\CC,\Delta, \epsilon)$ be a cooperad. A $\CCC$-\textit{comodule} $\DDD=(\DD, \Delta_{\DDD})$ is a left $\CCC$-comodule in the category of $\mbs$-modules, that is a $\mbs$-module $\DD$ together with a morphism $\Delta_{\DDD}: \DD \ra \CC \circ \DD$ such that the following diagrams commute.
 $$
 \xymatrix{\DD \ar[r]^{\Delta_{\DDD}} \ar[d]_{\Delta_{\DDD}} & \CC \circ \DD \ar[d]^{Id \circ \Delta_{\DDD}}\\
 \CC \circ \DD \ar[r]_(0.45){\Delta_\CCC \circ Id} & \CC \circ \CC \circ \DD}
 \hspace{1cm}
 \xymatrix{\DD \ar[r]^{\Delta_{\DDD}} \ar@/_2pc/[rr]^{Id} & \CC \circ \DD \ar[r]^{\epsilon \circ Id} & \DD}
 $$
A $\CCC$-\textit{coalgebra} is a $\CCC$-comodule concentrated in arity $0$.
\end{defin}

\begin{rmk}
 Our notion of $\CCC$-coalgebra recovers actually a notion called sometimes in the literature conilpotent $\CCC$-coalgebra; see \cite[5.4.8]{LodayVallette12}.
\end{rmk}

  Let $\CCC$ be a coaugmented cooperad. Then, the forgetful functor from the category of $\CCC$-comodules to the category of $\mbs$-modules has a right adjoint  which sends $\VV$ to $\CCC \circ \VV$. The images of the right adjoint are called the cofree $\CCC$-comodules.
 
 \begin{defin}[Coderivation]
 Let $\CCC$ be a graded cooperad and let $\DDD=(\DD, \Delta_{\DDD})$ be a $\CCC$-comodule. Suppose that $\CCC$ is equipped with a degree $k$ coderivation $d_\CCC$. A \textit{coderivation} on $\DDD$ is a degree $k$ map $d_\DDD$ from $\DD$ to $\DD$ such that
 $$
 \Delta_{\DDD} d_\DDD= (d_\CCC \circ Id + Id \circ' d_\DDD) \Delta_{\DDD}\ . 
$$
\end{defin}

 Let $\CCC$ be a cooperad equipped with a degree $k$ coderivation and let $\VV$ be a graded $\mbk$-module. Then, there is a one-to-one correspondence between the coderivations on the $\CCC$-coalgebra $\CCC \circ \VV$ and the degree $k$ maps $\CCC \circ \VV \ra \VV$. Indeed, any such map $u$ induces the coderivation
$$
D_u:= (d_\CCC \circ  Id_\VV) + (Id \circ (\pi ; u))(\Delta_\CCC \circ Id_\VV)\ ,
$$
where $\pi = \epsilon \circ Id : \CCC \circ \VV \ra \VV$.

\begin{defin}[Coalgebra over a curved cooperad]
 Let $\CCC$ be a curved cooperad. A $\CCC$-\textit{coalgebra} is a graded $\CCC^{grad}$-coalgebra $\DDD=(\DD, \Delta_{\DDD})$ together with a coderivation $d_\DD$ such that
$$
d_\DD^2 = (\theta_\CCC \circ Id) \Delta_{\DDD}
$$
\end{defin}

\begin{prop}\label{prop:curvedshortcut}
 Let $\CCC=(\CC,\Delta,\epsilon,1,d,\theta)$ be a conilpotent curved cooperad and let $\VV$ be a graded $\mbk$-module. There is a one-to-one correspondence between the degree $-1$ maps $\phi: \CC \circ \VV \ra \VV$ such that
 $$
\phi D_\phi := \phi \big(Id \circ (\pi ; \phi)\big)(\Delta_\CCC \circ Id_\VV) +  \phi (d_\CC \circ Id_\VV ) = \theta \circ Id_\VV
 $$
 and the structures of $\CCC$-coalgebras (where $\CCC$ is considered as a curved cooperad) on the graded cofree coalgebra $\CCC^{grad} \circ \VV$.
\end{prop}

\begin{proof}[Proof]
A structure of $\CCC$-coalgebra on $\CCC^{grad} \circ \VV$ amounts to the data of a degree $-1$ coderivation $D_\phi$ such that $D_\phi^2=(\theta \circ Id_\CCC \circ Id_\VV)(\Delta_\CCC \circ Id_\VV)$. If $D^2_\phi= (\theta \circ Id_\CCC \circ Id_\VV)(\Delta_\CCC \circ Id_\VV)$, then $\phi D_\phi= \theta \circ Id_\VV$. Conversely, suppose that $\phi D_\phi = \theta \circ Id$. We have:
\begin{align*}
 D^2_\phi &= (d^2_\CCC \circ Id_\CC) + \big(Id \circ (\pi ; \phi)\big)(\Delta \circ Id) (d_\CCC \circ Id_\VV) + \big(d_\CCC \circ (\pi ; \phi)\big)(\Delta \circ Id) \\
 &+  \big(Id \circ (\pi ; \phi)\big)(\Delta \circ Id) \big(Id \circ (\pi ; \phi)\big)(\Delta \circ Id) \ .
\end{align*}
On the one hand,
\begin{align*}
 &\big(Id \circ (\pi ; \phi)\big)(\Delta \circ Id) \big(Id \circ (\pi ; \phi)\big)(\Delta \circ Id)  +\big(Id \circ (\pi ; \phi)\big)(\Delta \circ Id) (d_\CCC \circ Id_\CC) + \big(d_\CCC \circ (\pi ; \phi)\big)(\Delta \circ Id) \\
 &= \big(Id \circ (\pi ; \phi)\big)(Id \circ' D_\phi)(\Delta \circ Id) =\big(Id \circ (\pi ; \phi D_\phi)\big)(\Delta \circ Id) =\Big(\big(Id \circ (\epsilon ; \theta)\big) \Delta\Big)\circ Id\ .
\end{align*}
On the other hand,
$$
(d_\CCC^2 \circ Id) = \big((\theta \circ Id)\Delta\big) \circ Id+ \Big( \big( \sum Id \otimes_\mbs (\epsilon^{\otimes i} \otimes \theta \otimes \epsilon^{\otimes j} ) \big) \Delta\Big)\circ Id\ .
$$
Hence $D^2_\phi = \big((\theta \circ Id)\Delta\big) \circ Id_\VV$.
\end{proof}

\begin{defin}[Coradical filtration]
  Any $\CCC$-coalgebra $\DDD= (\DD,\Delta_\DDD)$ over a conilpotent cooperad $\CCC$ admits a filtration called the \textit{coradical filtration} and defined as follows.
 $$
 F^{rad}_n \DD := \{x \in \DD | \Delta_\DDD (x) \in (F^{rad}_n \CC) \circ \DD \}\ .
 $$
\end{defin}

\begin{prop}\label{prop:cogtech}
Let $\CCC$ be a conilpotent cooperad and let $\DDD$ be a $\CCC$-coalgebra. For any integer $n$:
$$
 \Delta_{\DDD} (F^{rad}_n \DD) \subset \sum_{i_0+ i_1 + \cdots + i_k = n} (F^{rad}_{i_0} \CC) (k) \otimes_{\mbs_k} ( F^{rad}_{i_1} \DD \otimes \cdots \otimes F^{rad}_{i_k} \DD)\ .
 $$
\end{prop}

\begin{lemma}\label{lemtech:filt}
Let $\VV$ and $\WW$ be two graded $\mbk$-modules equipped with filtrations $(F_n \VV)_{n \in \mbn}$ and $(F_n \WW)_{n \in \mbn}$, and let $\phi: \VV \ra \WW$ be an injection such that $F_n \VV= \phi^{-1}(F_n \WW)$ for any integer $n$. Then, there exists a map $\psi: \WW \ra \VV$ such that $\psi \phi = Id$ and $\psi (F_n \WW) = F_n \VV$ for any $n \in \mbn$.
\end{lemma}

\begin{proof}[Proof]
For an integer $n \geq -1$, suppose that we have built  a sub graded $\mbk$-module $\UU_n$ of $F_{n}\WW$ such that $F_m \WW = \phi(F_m \VV) \oplus ( \UU_n \cap F_m \WW)$ for any $m \leq n$. Let $\UU'_n$ be a sub-graded $\mbk$-module of $F_{n+1} \WW$ algebraic complement to $\phi(F_{n+1} \VV) \oplus \UU_n$. Then, let $\UU_{n+1}:= \UU_n \oplus \UU'_n$. Finally, let $\UU:=\colim\  \UU_n$. We define $\psi$ by
$$
\psi=
\begin{cases}
\phi^{-1} \text{ on }\phi(\VV)\ ,\\
0 \text{ on }\UU\ .
\end{cases}
$$
 \end{proof}

\begin{proof}[Proof of Proposition \ref{prop:cogtech}:]
The map $\Delta_{\DDD}: \DD \ra \CC \circ \DD$ is actually a morphism of $\CCC$-coalgebras such that $\Delta_\DDD^{-1}(F^{rad}_n \CC \circ \DD) = F^{rad}_n \DD$. By Lemma \ref {lemtech:filt}, there exists a map of graded $\mbk$-modules $\nabla: \CC \circ \DD \ra \DD$ such that $\nabla \Delta_\DDD=Id_\DD$ and $\nabla (F_n \CC \circ \DD)= F_n \DD$. Then, the following diagram is commutative.
$$
\xymatrix{\DD \ar[r]^{\Delta} \ar[d]_{\Delta} & \CC \circ \DD \ar[d]^{\Delta \circ Id}\\ 
\CC \circ \DD \ar[r]_{Id \circ \Delta} \ar@/_2pc/[rr]_{Id} & \CC \circ \CC \circ \DD \ar[r]_{Id \circ \nabla} & \CC \circ \DD}
$$
By Lemma \ref{lem:cooptech}, we know that: 
$$
(\Delta \circ Id )\Delta (F^{rad}_n  \DD) \subset  \sum_{i_0 + \cdots + i_k =n} F^{rad}_{i_0} \CC(k) \otimes_{\mbs_k} (F^{rad}_{i_1} \CC \circ \DD \otimes \cdots \otimes F^{rad}_{i_k} \CC \circ \DD)\ ,
$$
Moreover, we know that:
$$
(Id \circ \nabla)\big( F^{rad}_{i_0} \CC(k) \otimes_{\mbs_k} (F^{rad}_{i_1} \CC \circ \DD \otimes \cdots \otimes F^{rad}_{i_k} \CC \circ \DD)) \subset F^{rad}_{i_0} \CC(k) \otimes_{\mbs_k} (F^{rad}_{i_1}  \DD \otimes \cdots \otimes F^{rad}_{i_k} \DD)\ .
$$
So, we have:
$$
\Delta(F^{rad}_n \DD) = (Id \circ \nabla)(\Delta \circ Id )\Delta (F^{rad}_n \DD) \subset \sum_k \sum_{i_0 + \cdots + i_k =n}  F^{rad}_{i_0} \CC(k) \otimes_{\mbs_k} (F^{rad}_{i_1}  \DD \otimes \cdots \otimes F^{rad}_{i_k} \DD)\ .
$$
\end{proof}

\subsection{Presentability}

This subsection deals with the presentability of the category of algebras over an operad and the presentability of the category of coalgebras over a conilpotent curved cooperad.

\begin{thm}\cite[5.2]{DrummondColeHirsh14}\label{thm:algpresentable}
Let $\PPP$ be a dg-operad. Then the category $\Palg$ of $\PPP$-algebras is presentable.
\end{thm}

The essence of the last theorem is that any $\PPP$-algebra is the colimit of a filtered diagram of finitely presented $\PPP$-algebras.

\begin{thm}\label{thm:cogcompact}
Let $\CCC$ be a conilpotent curved cooperad. The category $\Ccog$ of $\CCC$-coalgebras is presentable.
\end{thm}

The essence of this theorem is that any $\CCC$-coalgebra is the colimit of a filtered diagram of finite dimensional $\CCC$-coalgebras. Since the category of $\CCC$-coalgebras does not seem to be comonadic over a known presentable category, we cannot use the same kind of arguments as in the proof of \cite[5.2]{DrummondColeHirsh14}.

\begin{lemma}
 The category $\Ccog$ is cocomplete.
\end{lemma}

\begin{proof}[Proof]
 The colimit of a diagram of $\CCC$-coalgebras is its colimit in the category of graded $\mbk$-modules, together with the obvious decomposition map and coderivation map.
\end{proof}

\begin{lemma}\label{lemma:finitedimsubcoalg}
 For any $\CCC$-coalgebra $\DDD=(\DD,\Delta_\DDD)$ and any finite dimensional sub-graded $\mbk$-module $\VV \subset \CC$, there exists a finite dimensional sub-$\CCC$-coalgebra $\EEE$ of $\DDD$ which contains $\VV$.
\end{lemma}

\begin{proof}[Proof]
 Let us prove the result by induction on the coradical filtration of $\DDD$. Suppose first that $\VV \subset F_0 \DD$. Then, $\VV+ d \VV$ is a sub $\CCC$-coalgebra of $\DDD$. Then, suppose that, for any finite dimensional sub-graded $\mbk$-module $\WW \in F^{rad}_n\DD$, there exists a finite dimensional sub $\CCC$-coalgebra $\EEE$ of $F^{rad}_n \DDD$ which contains $\WW$. Consider now a finite dimensional sub-graded $\mbk$-module $\VV \subset  F_{n+1}\DD$. By Proposition \ref{prop:cogtech}, for any element $x \in F^{rad}_{n+1}\DD$, $\Delta_\DDD (x)-1 \otimes x \in \CC\circ F^{rad}_n{\DD}$. Since we are working with conilpotent $\CCC$-coalgebras, there exists a finite dimensional sub graded $\mbk$-module $\VV(x)$ of $F^{rad}_n \DD$ such that $\Delta_\DDD (x)-1 \otimes x \in \CC \circ \VV(x)$.  Let $(e_i)_{i=1}^k$ be a linearly free family of elements of $\VV$ such that $\VV=\VV \cap F^{rad}_n\DD \oplus  \bigoplus_{i=1}^k \mbk . e_i$. By the induction hypothesis, let $\EE$ be a finite dimensional sub $\CCC$-coalgebra of $\DDD$ which contains
 $$
 \VV \cap F^{rad}_n \DD \oplus \sum \VV(e_i) + \VV(d_\DD e_i)\ .
 $$
 Then, the sum
 $$
 \EE + \sum_i (\mbk . e_i \oplus \mbk . d_\DD e_i)
 $$
 is a finite dimensional sub $\CCC$-coalgebra of $\DDD$ which contains $\VV$. 
 \end{proof}

Finally, we show that a finite dimensional $\CCC$-coalgebras is a compact object.

\begin{prop}\label{prop:compact}
A finite dimensional $\CCC$-coalgebra is a compact object.
\end{prop}

We need the following technical lemma.

\begin{lemma}\label{lemtechcompact}
 Let  $D: I \ra \Ccog$ be a filtered diagram. Let $x \in D(i)$ for an object $i$ of $I$. If the image of $x$ in $\colim D$ is zero, then, there exists an object $i'$ of $I$ and a map $\phi: i \ra i'$ such that $D(\phi)(x)=0$.
\end{lemma}

\begin{proof}[Proof]
The colimit of the diagram $D$ is the cokernel of the map
$$
g:\bigoplus_{f : j \ra j'} D(j) \ra \bigoplus_{i \in Ob(I)} D(i) 
$$
such that for any morphism $f:j \ra j'$of $I$, the morphism $g$ sends $x \in D(j)$ to $x - D(f)(x)$. Let $x \in D(i)$ whose image in $\colim D$ is zero. Then, there exists an element $y= \sum y_f$ of $\bigoplus_{f :j \ra j'} D(j)$ such that $g(y) = x$. Let $i'$ be a cocone in $I$ of the finite diagram made up of the morphisms $f$ such that $y_f \neq 0$. Then, the image in $D(i')$ of $\sum y_f$ is the same as the image in $D(i')$ of $\sum D(f)(y_f)$. Hence, the image of $x$ in $D(i')$ is zero.
\end{proof}

\begin{proof}[Proof of Proposition \ref{prop:compact}]
  Let  $D: I \ra \Ccog$ be a filtered diagram and let $\DDD= (\DD, \Delta_\DDD)$ be a finite dimensional $\CCC$-coalgebra. We have to show that the canonical map
  $$
  \colim \big( \hom_{\Ccog}(\DDD, D) \big) \ra \hom_{\Ccog}(\DDD,\colim D )
  $$
  is bijective.
  
\begin{itemize}
 \itemt Let us first show that it is surjective. Let $f : \DDD \ra \colim D$ be a map of $\CCC$-coalgebra and let $\DD'$ be the image of $f$ inside $\colim D$ which is also a sub-$\CCC$-coalgebra of $\colim D$. Let $\{e_a\}_{a=1}^n$ be a basis of the graded $\mbk$-module $\DD'$. Since the diagram $D$ is filtered, there exists an object $i$ of $I$ and elements $x_a \in D(i)$ whose image in $\colim D$ is $e_a$. Let $\EEE$ be the smallest sub $\CCC$-coalgebra of $D(i)$ which contains all the $x_a$ and let $\EEE'$ be the image of $\EEE$ in $\colim D$. Notice that $\EEE'$ contains $\DD'$ and that the map $\EEE \ra \EEE'$ is surjective. By Lemma \ref{lemtechcompact} and since $\EEE$ is finite dimensional, there exists an object $i'$ and a map $\phi:i \ra i' $ such that, the map $\EEE'':=D(\phi)(\EEE) \ra \EEE'$ is an isomorphism of $\CCC$-coalgebras. So, let $\DDD''$ be the sub $\CCC$-coalgebra of $\EEE''$ which is the image of $\DD'$ through the inverse isomorphism $\EEE' \ra \EEE''$. Hence, the map $\DDD \ra \DDD' \ra \colim D$ factors through the map $\DDD \ra \DDD' \simeq \DDD'' \ra D(i')$ and so the canonical map $\colim \big( \hom_{\Ccog}(\DDD, D) \big) \ra \hom_{\Ccog}(\DDD,\colim D )$ is surjective.
 \itemt Let us show that it is injective. Let $f \in \hom_{\Ccog}(\DDD, D(i))$ and $g \in \hom_{\Ccog}(\DDD, D(j))$ be two maps whose images in $\hom_{\Ccog}(\DDD,\colim D)$ are the same; it is denoted $h$. Since the category $I$ is filtered, there exists an object $k$ together with maps $\phi:i \ra k$ and $\psi: j \ra k$. Then $D(\phi)f(\DDD) + D(\psi)g(\DDD)$ is a finite dimensional sub $\CCC$-coalgebra of $D(k)$ whose image in $\colim D$ is $h(\DDD)$. As in the previous point (by Lemma \ref{lemtechcompact}), the exists a map $\zeta: k \ra k'$ in $I$ such that the map $u:D(\zeta) \big(D(\phi)f(\DDD) + D(\psi)g(\CC)\big) \ra h(\DDD)$ is an isomorphism. Since the dimension (as a graded $\mbk$-module) of $D(\zeta) D(\phi)f(\DDD)$ and the dimension of $D(\zeta) D(\psi)g(\DDD)$ are both greater than the dimension of $h(\DDD)$, then we must have:
 $$
 D(\zeta) \big(D(\phi)f(\DDD) + D(\psi)g(\DDD)\big) = D(\zeta) D(\phi)f(\DDD) = D(\zeta) D(\psi)g(\DDD)\ .
 $$
 In this context, we have
 $$
 D(\zeta) D(\phi)f = u^{-1} h = D(\zeta) D(\psi)g\ .
 $$
 Hence, $f$ and $g$ represent the same element of $\colim \big( \hom_{\Ccog}(\DDD, D) \big)$.
\end{itemize}
\end{proof}

\begin{proof}[Proof of Theorem \ref{thm:cogcompact}]
The isomorphisms classes of finite dimensional $\CCC$-coalgebras form a set. By Proposition \ref{prop:compact}, any finite dimensional $\CCC$-coalgebra is a compact object of the category $\Ccog$. Moreover, any $\CCC$-coalgebra is the colimit of the diagram of its finite dimensional sub $\CCC$-coalgebras (with inclusions between them); this is a filtered diagram (and even a directed set). Hence, the category $\Ccog$ is presentable.
\end{proof}

\section{Enrichment}\label{section:enrich}

This section deals with several enrichments of the category of algebras of an operad and of the category of coalgebras of a curved conilpotent cooperad. Specifically, we prove that both the category of algebras over an operad and the category of coalgebras over a curved conilpotent cooperad are tensored, cotensored and enriched over cocommutative coalgebras and enriched over simplicial sets. In the nonsymmetric context, algebras over an operad and coalgebras over a curved conilpotent cooperad are tensored, cotensored and enriched over coassociative coalgebras; this leads to another enrichment in simplicial sets.

\subsection{Enrichment over coassociative coalgebras and cocommutative coalgebras} We show in this subsection that the category of algebras over an operad and the category of coalgebras over a curved conilpotent cooperad are tensored-cotensored-enriched (see Definition \ref{def:tce}) over the category $\uCocom$ of counital cocommutative coalgebras. Moreover, in the nonsymmetric context, they are tensored-cotensored-enriched over the category $\uCog$ of coassociative coalgebras. We will use these enrichments in the sequel to describe respectively deformations of morphisms and mapping spaces

\subsubsection{Enrichment of $\PPP$-algebras over coalgebras} Let $\PPP=(\PP,\gamma,1)$ be a dg operad. For any counital cocommutative coalgebra $\CCC= (\CC,\Delta_\CCC,\epsilon)$ and any $\PPP$-algebra $\AAA=(\Aa,\gamma_\AAA)$, the chain complex $[\CC, \Aa]$ has a canonical structure of $\PPP$-algebra as follows.
\begin{itemize}
 \item[$\triangleright$] For any $p \in \PP(n)$ ($n \geq 1$), and for any $f_1$, \ldots, $f_n \in [\CC, \Aa]$ and any $x=\CC$, then
 $$
 \gamma_{[\CC,\Aa]} \big(p \otimes_{\mbs_n}( f_1 \otimes \cdots \otimes f_n)\big) (x) = \gamma_{\AAA} (p \otimes -) (f_1 \otimes \cdots \otimes f_n) \Delta_\CCC^{n-1}(x)
 $$
  \item[$\triangleright$] For any $p \in \PP(0)$:
 $$
 \gamma_{[\CC,\Aa]} (p) =  \gamma_{\AAA} (p) \epsilon_\CC
 $$
\end{itemize}
 
 The chain complex $[\CC,\Aa]$ together with its structure of $\PPP$-algebra is denoted $[\CCC,\AAA]$.
 
\begin{lemma}
 The assignment $\CCC, \AAA \mapsto [\CCC, \AAA]$ defines a left coaction (see Definition \ref{defin:tensored}) of the category $\uCocom$ of counital cocommutative coalgebras on the category $\Palg$ of $\PPP$-algebras .
\end{lemma}

\begin{proof}[Proof]
The construction is functorial covariantly with respect to $\PPP$-algebras and contravariantly with respect to counital cocommutative coalgebras. Moreover, for any counital cocommutative coalgebras $\CCC$ and $\DDD$, and any $\PPP$-algebra $\AAA$ there is an isomorphism of chain complexes
$$
\rho_{\CCC,\DDD,\AAA}:[\CC \otimes \DD, \Aa] \ra [\CC,[\DD,\Aa]]
$$
such that $\rho_{\CCC,\DDD,\AAA} (f)(x)(y) = f (x \otimes y)$. This is a morphism of $\PPP$-algebra which is functorial in $\CCC$, $\DDD$ and $\AAA$, and it satisfies the coherence conditions of Definition \ref{defin:tensored}.
\end{proof}

One can define a left adjoint to the functor $[\CCC, -] $ as follows. Let $\AAA \triangleleft \CCC $ be the quotient of the free $\PPP$-algebra $\PP\circ (\Aa \otimes \CC)
 $
 by the ideal $I$ generated by the relations
 $$
\begin{cases}
 \gamma_\AAA \big(p \otimes_{\mbs_n} ( y_1 \otimes \cdots \otimes y_n) \big) \otimes x \sim  \sum (-1)^{\sum_{i < j} |x_{(i)}||y_j|} p \otimes_{\mbs_n}\big( (y_1 \otimes x_{(1)} ) \otimes \cdots \otimes (y_n \otimes x_{(n)})\big) \\
 \gamma_\AAA (p) \otimes x \sim \epsilon (x) p\  \text{for any } p \in \PP(0)\ ,
\end{cases}
 $$
with $\Delta^{n-1} (x) = \sum x_{(1)} \otimes \cdots \otimes x_{(n)}$.

\begin{thm}\label{thm:enrich}
The category of $\PPP$-algebras is tensored-cotensored-enriched over the category $\uCocom$ of counital cocommutative coalgebras. The right action is given by the functor $-\triangleleft -$ and the left coaction is given by the functor $[- ,- ]$. We denote the enrichment by $\{-,-\}$.
\end{thm}

\begin{proof}[Proof]
 Since the functor $[-,-]$ defines a coaction of the category of counital cocommutative coalgebras on the category of $\PPP$-algebras, since the functor $[- , \Aa]$ sends colimits to limits and since the functor $[\CCC, -]$ is left adjoint to the functor $- \triangleleft \CCC$, then we can conclude by Proposition \ref{prop:presenttensor}.
\end{proof}

Let us describe $\{\AAA,\AAA'\}$ for two $\PPP$-algebras $\AAA$ and $\AAA'$. This is the final sub-coalgebra of the cofree cocommutative coalgebra $F([\Aa,\Aa'])$ such that the following diagram commutes
$$
\xymatrix{ \{\AAA,\AAA'\} \ar[rrr] \ar[d]_{(\epsilon,Id,\Delta,\ldots)} &&& [\Aa,\Aa'] \ar[d] \\
 \prod_{n \geq 0}\{\AAA,\AAA'\}^{\otimes n}/\mbs_n \ar[r]  & \prod_{n \geq 0} [\Aa,\Aa']^{\otimes n}/\mbs_n \ar[r]   & [\PP \circ \Aa , \PP \circ \Aa']  \ar[r]   & [\PP \circ \Aa, \Aa']\ . }
$$
where the map $\prod_{n \geq 0} [\Aa,\Aa']^{\otimes n}/\mbs_n \ra   [\PP \circ \Aa , \PP \circ \Aa']$ sends $f_1 \otimes \cdots \otimes f_n$ to $Id_{\PP(n)} \otimes_{\mbs_n} (f_1 \otimes \cdots \otimes f_n)$.

\subsubsection{Enrichment of $\CCC$-coalgebras over coalgebras} Let $\CCC=(\CC,\Delta,\epsilon,1,d,\theta)$ be a curved conilpotent cooperad.\\

 For any $\CCC$-coalgebra $\DDD= (\DD,\Delta_\DDD)$ and any counital cocommutative coalgebra $\EEE=(\EE,\Delta_\EEE,\epsilon)$, the tensor product $\DD \otimes \EE$ has a structure of $\CCC$-coalgebra given by
$$
\xymatrix{\DD \otimes \EE \ar[rr]^(0.35){\bigoplus_n \Delta_n \otimes \Delta^{n-1}}&& \bigoplus_n (\CC(n) \otimes_{\mbs_n} \DD^{\otimes n} ) \otimes \EE^{\otimes n} \ar[r] & \bigoplus_n \CC(n) \otimes_{\mbs_n} (\DD \otimes \EE)^{\otimes n}  \ .}
$$ 
\begin{thm}\label{thm:thmenrich2}
 The category $\Ccog$ of $\CCC$-coalgebras is tensored-cotensored-enriched over the category of cocommutative counital coalgebras. The right action is given by the construction $-\otimes -$. We denote the left coaction by $\langle - ,- \rangle$ and the enrichment by $\{-,-\}$.
\end{thm}

\begin{proof}[Proof]
The assignment $\DDD, \EEE \mapsto \DDD \otimes \EEE$ defines a right action of the category of counital cocommutative coalgebras on the category of $\CCC$-coalgebras. Moreover, the functor $\DDD \otimes -$ and the functor $- \otimes \EEE$ preserve colimits. We conclude by Proposition \ref{prop:presenttensor}.
\end{proof}

If $\DDD$ and $\DDD'$ are two $\CCC$-coalgebras, then the cocommutative counital hom coalgebra $\{\DDD,\DDD'\}$ is the final sub-coalgebra of the cofree counital cocommutative coalgebra $F([\DD,\DD'])$ over the chain complex $[\DD,\DD']$ such that the following diagram commutes.
$$
\xymatrix{\{\DDD,\DDD'\} \ar[rrr]  \ar[d]_{(\epsilon,Id,\Delta,\ldots)} &&& [\DD,\DD']\ar[d]\\
 \prod_{n \geq 0} \{\DDD,\DDD'\}^{\otimes n}/\mbs_n \ar[r] & \prod_{n \geq 0} [\DD,\DD']^{\otimes n}/\mbs_n \ar[r] & [\CC \circ \DD, \CC \circ \DD']  \ar[r] &[\DD,\CC \circ \DD']}
$$

\subsubsection{Morphisms are atoms}

\begin{prop}\label{prop:atom}
For any two $\PPP$-algebras $\AAA$ and $\AAA'$, the dg atoms of the cocommutative coalgebra $\{\AAA, \AAA'\}$ are the morphisms of $\PPP$-algebras from $\AAA$ to $\AAA'$. Similarly, for any two $\CCC$-coalgebras $\DDD$ and $\DDD'$, the dg atoms of the cocommutative coalgebra $\{\DDD, \DDD'\}$ are the morphisms of $\CCC$-coalgebras from $\DDD$ to $\DDD'$.
\end{prop}

\begin{proof}[Proof]
We have
 $$
 \hom_{\uCocom} (\mbk, \{\AAA, \AAA'\}) \simeq \hom_{\Palg} (\AAA \triangleleft \mbk ,  \AAA') \simeq \hom_{\Palg} (\AAA ,  \AAA')\ .
 $$ 
\end{proof}

\subsubsection{Nonsymmetric context} In the nonsymmetric context, we can get rid of the cocommutativity condition.

\begin{prop}\leavevmode
\begin{itemize}
 \item[$\triangleright$] If $\PPP$ is a nonsymmetric operad, then the category of $\PPP$-algebras is tensored-cotensored-enriched over the category $\uCog$ of counital coassociative coalgebras.
  \item[$\triangleright$] If $\CCC$ is a nonsymmetric conilpotent curved cooperad, then the category of $\CCC$-coalgebras is tensored-cotensored-enriched over the category $\uCog$ of counital coassociative coalgebras.
  \end{itemize}
  We denote by $\{-,-\}^{ns}$ these two enrichments over counital coassociative coalgebras.
\end{prop}

\begin{proof}[Proof]
 The proof is similar to the proofs of Theorem \ref{thm:enrich} and of Theorem \ref{thm:thmenrich2}.
\end{proof}

The inclusion functor $\uCocom \hookrightarrow \uCog$ is a left adjoint (since it preserves colimits). Let $R$ be its right adjoint. It sends any counital coassociative coalgebra to its final cocommutative subcoalgebra.

\begin{prop}
 For any $\PPP$-algebras $\AAA$ and $\AAA'$, the cocommutative coalgebra, $\{\AAA,\AAA'\}$ is the final cocommutative subcoalgebra $R(\{\AAA,\AAA'\}^{ns})$ of $\{\AAA,\AAA'\}^{ns}$.
Similarly, for any $\CCC$-coalgebras $\DDD$ and $\DDD'$, the cocommutative coalgebra, $\{\DDD,\DDD'\}$ is the final cocommutative subcoalgebra $R(\{\DDD,\DDD'\}^{ns})$ of $\{\DDD,\DDD'\}^{ns}$.
\end{prop}

\begin{proof}[Proof]
 For any cocommutative coalgebra $\EEE$, we have:
 $$
 \hom_{\uCocom} (\EEE ,\{\AAA,\AAA'\}) \simeq \hom_{\Palg} (\AAA \triangleleft \EEE, \AAA') \simeq \hom_{\uCog} (\EEE, \{\AAA,\AAA'\}^{ns}) \simeq \hom_{\uCocom}(\EEE,R(\{\AAA,\AAA'\}^{ns}))\ .
 $$
 Since these isomorphisms are functorial, $R(\{\AAA,\AAA'\}^{ns})$ is isomorphic to $\{\AAA,\AAA'\}$.
\end{proof}

\subsection{Simplicial enrichment} In this section, we recall the fact that the Sullivan polynomials forms algebras allow one to enrich the category of algebras over an operad. See for instance \cite{Hinich01}.

\subsubsection{General case}

Let $A$ be a unital commutative $\mbk$-algebra. The functor $ A \otimes -: \dgMod \ra \dgMod_{A}$ is strong symmetric monoidal and comonoidal. Hence it induces several functors:
\begin{itemize}
 \itemt from operads to operads enriched in $A$-modules,
  \itemt from cooperads to cooperads enriched in $A$-modules,
  \itemt from $\PPP$-algebras (in the category of $\mbk$-modules) to $A \otimes \PPP$-algebras (in the category of $A$-modules),
  \itemt from $\CCC$-coalgebras (in the category of $\mbk$-modules) to $A \otimes \CCC$-coalgebras (in the category of $A$-modules).
\end{itemize}
Applying this to the case of the Sullivan algebras of polynomial forms on standard simplicies leads us to the following proposition.

\begin{prop} \label{prop:enrichsimplgen}
Let $\PPP$ be a dg operad and let $\CCC$ be a curved conilpotent cooperad. The category of $\PPP$-algebras and the category of $\CCC$-coalgebras are enriched in simplicial sets as follows:
 $$
 \HOM(\AAA,\AAA')_n:= \hom_{\Palg} (\AAA, \Omega_n \otimes \AAA') \simeq \hom_{\Omega_n \otimes \Palg} (\Omega_n \otimes \AAA, \Omega_n \otimes \AAA' )\ ,
 $$
$$
\HOM(\DDD,\DDD')_n :=\hom_{\Omega_n \otimes \Ccog} (\Omega_n \otimes \DDD, \Omega_n \otimes \DDD')\ .
$$
\end{prop}

\begin{proof}[Proof]
 The only point that needs to be cleared up is the simplicial structure on $\HOM(\DDD,\DDD')$. Let $\phi : [m] \ra [n]$ be a map between finite ordinals . We want to define $\phi^*: \HOM(\DDD,\DDD')_n \ra \HOM(\DDD,\DDD')_m$. An element of $\HOM(\DDD,\DDD')_n$ is a morphism of graded $\mbk$-modules $f$ from $\DD$ to $\Omega_n \otimes \DD'$ such that $fd_\DD=(d_{\Omega_n} \otimes Id_{\DD'} +Id_{\Omega_n} \otimes d_{\DD'})$ and such that the following diagrams commute
 $$
 \xymatrix{\DD \ar[d]_\Delta \ar[rr]^f && \Omega_n \otimes \DD' \ar[d]^{Id \otimes \Delta}\\
 \CC \circ \DD \ar[r]_{Id \circ f} & \CC \circ (  \Omega_n \otimes \DD) \ar[r] & \Omega_n \otimes (\CC \circ \DD') \ , }
 \xymatrix{\DD \ar[r]^f \ar[d]_{\theta_\DD} &\Omega_n \otimes \DD' \ar[d]^{Id_A \otimes \theta_{\DD'}} \\
 \mbk \ar[r] & \Omega_n\ ,}
 $$
 where the map $\CC \circ ( \Omega_n \otimes \DD') \ra  \Omega_n \otimes (\CC \circ \DD') $ in the first diagram is the following map
 $$
 x \otimes_{\mbs_k} \big((a_1 \otimes x_1)  \otimes \cdots \otimes (a_k \otimes x_k)\big) \mapsto (-1)^{|x|(\sum |a_i|)}(-1)^{\sum_{i>j}  |a_i| |x_j|}   (a_1 \cdots a_k) \otimes \big( x \otimes_{\mbs_k} (x_1 \otimes \cdots \otimes  x_k)\big)\ .
 $$
 Then, $\phi^* (f) = (\Omega[\phi] \otimes Id) f$ where $\Omega[\phi]: \Omega_n \ra \Omega_m$ is the structural map induced by $\phi$.
\end{proof}

\begin{prop}\label{prop:sullivanfinitelim}
 For any simplicial set $X$ which is the colimit of a finite diagram of simplicies $\Delta[n]$ and for any $\PPP$-algebras $\AAA$ and $\AAA'$, we have:
 $$
 \hom_{\sSet}(X,\HOM(\AAA, \AAA')) \simeq \hom_{\Palg}(\AAA,\Omega_X \otimes \AAA')\ .
 $$
\end{prop}

\begin{proof}[Proof]
It suffices to notice that the functor from commutative algebras to $R \otimes \PPP$-algebras $R \mapsto R \otimes \AAA' $ preserves finite limits.
\end{proof}

\begin{rmk}
 The enrichment of the category of $\PPP$-algebras and of the category of $\CCC$-coalgebras over simplicial sets that we described above is a part of a more general enrichment over functors from the category of unital commutative algebras to simplicial sets:
\begin{align*}
 R &\mapsto (\hom_{\Palg} (\AAA, \Omega_n \otimes R \otimes \BBB))_{n \in \mbn}\\
 R &\mapsto (\hom_{\Omega_n \otimes R \otimes \Ccog} (\Omega_n \otimes R \otimes  \DDD, \Omega_n \otimes R \otimes \DDD')\ .)_{n \in \mbn}\ .
\end{align*}
\end{rmk}

\subsubsection{Nonsymmetric context} In the nonsymmetric context, we can use some associative algebras instead of the commutative Sullivan algebras to define a simplicial enrichment. Let $\Lambda_n$ be the linear dual of the Dold-Kan coalgebra over the standard simplex: 
$$
\Lambda_n := DK^c(\Delta[n])^*\ .
$$
This defines a simplicial unital associative algebra.\\

Besides, let $\PPP$ be a nonsymmetric dg operad. For any $\PPP$-algebra $\AAA=(\Aa,\gamma_\AAA)$, and for any associative algebra $A$, $A \otimes \Aa$ has a canonical structure of $\PPP$-algebra.
\begin{defin}[Nonsymmetric simplicial enrichment of algebras over an operad]
 For any two $\PPP$-algebras $\AAA$ and $\BBB$, let $\HOM^{ns}(\AAA,\BBB)$ be the following simplicial set:
 $$
\HOM^{ns} (\AAA,\BBB)_n:=  \hom_{\Palg} (\AAA,\Lambda_n \otimes \BBB)\ .
 $$
 This defines a simplicial enrichment of the category of $\PPP$-algebras over simplicial sets.
 \end{defin}
 
 Let $\CCC$ be a nonsymmetric curved conilpotent cooperad. For any associative algebra $A$ and for any two $\CCC$-coalgebras $\DDD=(\DD,\Delta_\DDD)$ and $\EEE=(\EE,\Delta_\EE)$, we denote by $\hom_{A,\CCC} (\DDD, \EEE)$ the set of morphisms of graded $\mbk$-modules $f$ from $\DD$ to $A \otimes \EE$ which commute with the coderivations such that the following diagrams commute.
 $$
 \xymatrix{ \DD \ar[rr]^f \ar[d]_{\Delta} && A \otimes \EE \ar[d]^{Id_A \otimes \Delta_\EEE}\\
 \CC \circ_{ns} \DD \ar[r] & \CC \circ_{ns} (A \otimes \EE) \ar[r] &A \otimes (\CC \circ_{ns} \EE) }
 \xymatrix{\DD \ar[r]^f \ar[d]_{\theta_\DD} &A\otimes \EE \ar[d]^{Id_A \otimes \theta_\EE} \\
 \mbk \ar[r] & A}
 $$

\begin{defin}[Nonsymmetric simplicial enrichment of coalgebras over a curved cooperad]
  For any two $\CCC$-coalgebras $\DDD$ and $\DDD'$, let $\HOM^{ns} (\DDD,\DDD')_n$ be the following simplicial set:
 $$
\HOM^{ns} (\DDD,\DDD')_n:=  \hom_{\Lambda_n,\CCC} (\DDD, \DDD')\ .
 $$
 This defines a simplicial enrichment of the category of $\CCC$-coalgebras over simplicial sets.
\end{defin}

These simplicial enrichments are related to the enrichments over coassociative coalgebras that we described above.

\begin{prop}\label{prop:nsenrichsimp}
 For any two $\PPP$-algebras $\AAA$ and $\BBB$ and for any two $\CCC$-coalgebras $\DDD$ and $\DDD'$, we have functorial isomorphisms
 $$
 \HOM^{ns} (\AAA,\BBB) \simeq N(\{\AAA,\BBB\}^{ns})\ ,
 $$
 $$
 \HOM^{ns} (\DDD,\DDD') \simeq N(\{\DDD,\DDD'\}^{ns}) \ .
 $$
\end{prop}

\begin{proof}[Proof]
The proof for $\PPP$-algebras is straightforward. Let us prove the result for the $\CCC$-coalgebras. A morphism of graded $\mbk$-modules $f$ from $\DD$ to $\Lambda_n \otimes \DD'$ is equivalent to a morphism from $\DD \otimes DK^c(\Delta[n])$ to $\DD'$. In that context, $f$ belongs to $\hom_{\Lambda_n, \CCC}(\DDD, \DDD')$ if and only if the corresponding morphism from $\DD \otimes DK^c(\Delta[n])$ to $\DD'$ is a morphism of $\CCC$-coalgebras. So
 \begin{align*}
\HOM^{ns}(\AAA,\BBB)_n &:= \hom_{\Lambda_n, \CCC}(\DDD, \DDD')  \simeq \hom_{\Ccog}(\DDD \otimes DK^c(\Delta[n]), \DDD' )\\
& \simeq \hom_{\uCog} (DK^c(\Delta[n]),\{ \DDD,\DDD' \}^{ns} ) \simeq \hom_{\sSet} (\Delta[n] , N(\{\DDD,\DDD'\}^{ns}) )\ .
\end{align*}
\end{proof}

\section{Bar-cobar adjunctions}
The usual bar-cobar adjunction relates nonunital algebras to non-counital conilpotent coalgebras, see \cite[Chapter 2]{LodayVallette12}. It can be extended to nonunital operads and conilpotent cooperads, see \cite{GetzlerJones94}. Besides, as a direct consequence of the work of Hirsh and Mill\`es, \cite{HirshMilles12}, there exists an adjunction \`a la bar-cobar relating unital algebras with curved conilpotent coalgebras. We extend it to operads and curved conilpotent cooperads.\\

The bar-cobar adjunction $\Omega_u \dashv B_c$ is a tool to compute resolutions of operads. But it has other aspects: any morphism of operads from the cobar construction $\Omega_u \CCC$ of a curved conilpotent cooperad $\CCC$ to an operad $\PPP$ gives rise to a new adjunction \`a la bar cobar between $\CCC$-coalgebras and $\PPP$-algebras.

\subsection{Operadic bar-cobar construction}\label{subsection:operadicbarcobar}

The usual operadic bar-cobar adjunction (see \cite[Chapter 6]{LodayVallette12}) relates augmented operads to differential graded conilpotent cooperads. The bar construction $B \PPP$ of an operad $\PPP$ does use the augmentation of $\PPP$ as it is the graded cofree cooperad on the suspension of $\ov \PPP$. If $\PPP$ is not augmented, one can try to add an element to $\PPP$ whose boundary is the unit of $\PPP$ and try the same computation. This is the new bar-construction; its output is no more a differential graded cooperad but a curved cooperad.\\

The new curved bar functor $B_c$ has also a left adjoint $\Omega_u$ whose formula looks like the usual operadic cobar functor. Again, as in \cite[Chapter 6]{LodayVallette12}, this adjunction is related to a notion of twisting morphism.

\begin{defin}[Operadic bar construction]
 The \textit{operadic bar construction} of a dg operad $\PPP=(\PP,\gamma_\PPP,1)$ is the curved conilpotent cooperad $B_c\PPP= (\Tfree^c (s\PP \oplus \mbk \cdot v), D, \theta)$ where $s\PP$ is the suspension of the $\mbs$-module $\PP$ and where $v$ is an arity $1$, degree $2$ element. It is equipped with the coderivation $D$ which extends the following map from $ \ov \Tfree (s \PP \oplus \mbk \cdot v)$ to $s \PP \oplus \mbk \cdot v$:
\begin{align*}
 \Tfree (s\PP \oplus v) \ra \Tfree^{\leq 2} (s\PP \oplus v) & \ra s\PP \oplus v\\
 sx &\mapsto -sd_\PP x\\
 sx \otimes sy &\mapsto (-1)^{|x|} s \gamma_\PPP (x \otimes y)\\
 v &\mapsto s1\ .
\end{align*}
It has the following curvature map :
\begin{align*}
\theta:  \ov\Tfree (s\PP \oplus v) \ra s\PP \oplus \mbk \cdot v \ra  \mbk \cdot v &\ra \mbk \\
 v & \mapsto 1\ .
\end{align*}
\end{defin}

\begin{prop}
The map $\theta$ is actually a curvature for the coderivation, that is $D^2 =  (\theta \otimes Id -  Id \otimes \theta )\Delta_2$.
\end{prop}

\begin{proof}[Proof]
Let $\pi$ be the projection from $B_c \PPP$ to $s\PP$. By Proposition \ref{lemma:curvcofree}, it suffices to prove that $\pi D^2= (\theta \otimes \pi - \pi \otimes \theta) \Delta_2$. This is a straightforward calculation.
\end{proof}

\begin{defin}[Operadic cobar construction]
 The \textit{operadic cobar construction} of a curved conilpotent cooperad $\CCC=(\CC,\Delta,\epsilon, 1, \theta) $ is the dg operad $\Omega_u \CCC =(\Tfree s^{-1}\CC,D)$ where $D$ is the following degree $-1$ derivation
$$
s^{-1}x \mapsto \theta(x) 1 -s^{-1} dx - \sum (-1)^{|x_{(1)}|} s^{-1}x_{(1)} \otimes s^{-1}x_{(2)}\ .
$$
where $\Delta_{2}(x)= \sum x_1 \otimes x_2$.

\end{defin}

\begin{prop}
 The derivation $D$ squares to zero.
\end{prop}

\begin{proof}[Proof]
 It suffices to prove the result for any element of the form $s^{-1}x$, which is a straightforward calculation.
\end{proof}

\begin{defin}[Operadic twisting morphism]
Let $\CCC=(\CC,\Delta,\epsilon, 1 , d ,\theta)$ be a curved conilpotent cooperad and let $\PPP=(\PP,\gamma_\PPP,1_\PPP,d)$ be a dg operad. An \textit{operadic twisting morphism} from $\CCC$ to $\PPP$ is a degree $-1$ map of $\mbs$-modules (or $\mbn$-modules in the nonsymmetric case):
 $$
 \alpha: \ov \CC \ra \PP
 $$
 such that 
 $$
 \partial (\alpha) + \gamma (\alpha \otimes \alpha)\Delta_2 = \Theta
 $$ 
 where $\Theta (x) = \theta(x) 1_\PPP$, for any $x \in \CCC$. We denote by $\Tw(\CCC, \PPP)$ the set of operadic twisting morphisms from $\CCC$ to $\PPP$.
\end{defin}

\begin{prop}\label{prop:barcobarcurvedaslv}
We have the following functorial isomorphisms:
$$
\hom_{\Operad}(\Omega_u \CCC, \PPP) \simeq Tw(\CCC, \PPP) \simeq \hom_{\cCoop} (\CCC, B_c \PPP)\ .
$$
\end{prop}

\begin{proof}[Proof]
 Proving the existence of the functorial isomorphism $\hom_{\Operad}(\Omega_u \CCC, \PPP) \simeq Tw(\CCC, \PPP)$ is similar to the proof of \cite[3.4.1]{HirshMilles12}. Let us show that we have a functorial isomorphism $Tw(\CCC, \PPP) \simeq \hom_{\cCoop} (\CCC, B_c \PPP)$. Let $\alpha: \CCC \ra \PPP$ be an operadic twisting morphism. We obtain a degree zero map from $\ov \CC$ to $s\PP \oplus \mbk \cdot v$ as follows:
\begin{align*}
 \ov \CC &\ra s\PP \oplus \mbk \cdot v\\
 c & \mapsto s \alpha (x) + \theta_{\CCC} (x)\ .
\end{align*}
This induces a morphism of graded cooperads $f_\alpha:\CCC \ra B_c \PPP = \Tfree^c (s\PP \oplus \mbk \cdot v)$ such that $\theta_\CCC = \theta_{B_c\PPP} f_\alpha$. Since $\partial (\alpha) + \gamma_\PPP (\alpha \otimes \alpha)\Delta_2 = \Theta$, then $f_\alpha$ commutes with the coderivations and so is a morphism of curved cooperads. Conversely, from any morphism of curved cooperads $f$ from $\CCC$ to $B_c \PPP$, one obtains a twisting morphism as follows:
$$
\xymatrix{\CCC \ar[r]^f & B_c \PPP  \ar@{->>}[r] & s\PPP \ar[r] & \PPP\ . }
$$
The two constructions that we described are inverse one to another.
 \end{proof}

Hence, the functors $\Omega_u$ and $B_c$ realize an adjunction between the category of dg operads and the category of curved conilpotent cooperads.
 $$
\xymatrix{\cCoop \ar@<1ex>[r]^(0.48){\Omega_u} & \Operad \ar@<1ex>[l]^(0.52){B_c}}
$$

\subsection{Twisted products}

Let $\alpha: \ov \CCC \ra \PPP$ be an operadic twisting morphism.

\begin{defin}[Twisted $\PPP$-module]
 For any $\CCC$-comodule $\DDD$, let $\PPP \circ_\alpha \DDD$ be the free $\PPP^{grad}$-module $\PP \circ \DD$ equipped with the unique derivation which extends the map
\begin{align*}
 \DD &\ra \PP \circ \DD\\
 x &\mapsto d_\DD(x) - (\alpha \circ Id )\Delta(x)\ .
\end{align*}
\end{defin}

\begin{defin}[Twisted $\CCC$-comodule] For any $\PPP$-module $\AAA$, let $\CCC \circ_\alpha \AAA$ be the cofree $\CCC^{grad}$-comodule  $\CC \circ \Aa$ equipped with a unique coderivation which extends the map
\begin{align*}
 \CC \circ \Aa &\ra \Aa\\
 x &\mapsto  \big( d_\Aa (\epsilon_\CCC \circ Id) + \gamma_{\AAA} (\alpha \circ Id) \big) (x)\ .
\end{align*}
\end{defin}

\begin{prop}
The derivation of $\PPP \circ_\alpha \DDD$ squares to zero. Hence, $\PPP \circ_\alpha \DDD$ is a dg $\PPP$-module.
Similarly, the coderivation of $\CCC \circ_\alpha \AAA$ squares to $(\theta \circ Id)\Delta$. Hence, $\CCC \circ_\alpha \AAA$ is a $\CCC$-comodule. 
\end{prop}

\begin{proof}[Proof]
To prove the first point, it suffices to show that $ \pi D^2 = 0$, which is a straightforward calculation. To prove the second point, it suffices to show that $\pi D^2 = (\theta \circ Id)$, which is a straightforward calculation.
\end{proof}

\begin{defin}[Twisting morphism relative to an operadic twisting morphism]
 For any $\CCC$-comodule $\DDD=(\DD,\Delta_\DDD)$ and any $\PPP$-module $\AAA=(\Aa,\gamma_\AAA)$ an $\alpha$-twisting morphism from $\DDD$ to $\AAA$ is a degree $0$ map $\phi: \DD \to \Aa$ such that
 $$
 \partial(\phi) + \gamma_\Aa(\alpha \circ \phi)\Delta_\CC = 0\ .
 $$
 We denote by $\Tw_\alpha (\DDD, \AAA)$ the set of $\alpha$-twisting morphisms from $\DDD$ to $\AAA$.
\end{defin}

\begin{prop}\label{prop:adjmodcomod}
There are functorial isomorphisms
 $$
 \hom_{\Pmod} (\PPP \circ_\alpha \DDD, \AAA) \simeq \Tw_\alpha (\DDD, \AAA) \simeq \hom_{\Ccomod} (\DDD, \CCC \circ_\alpha \AAA)
 $$
 for any $\CCC$-comodule $\DDD$ and any $\PPP$-module $\AAA$.
\end{prop}

\begin{proof}[Proof]
 The proof is similar to \cite[11.3.2]{LodayVallette12}.
\end{proof}

\subsection{Bar-cobar adjunction for algebras over an operad and coalgebras over a cooperad}

Following \cite[Chapter 11]{LodayVallette12}, we call the previous functors respectively the bar construction for $\PPP$-algebras and the cobar construction for $\CCC$-coalgebras.

\begin{defin}[Bar construction and cobar construction relatives to an operadic twisting morphism]
Let $\alpha :\ov \CCC \ra \PPP$ be an operadic twisting morphism. The $\alpha$-bar construction is the functor from $\PPP$-algebras to $\CCC$-coalgebras defined by:
$$
B_\alpha \AAA := \CCC \circ_{\alpha} \AAA\ .
$$
The $\alpha$-cobar construction is the functor from $\CCC$-coalgebras to $\PPP$-algebras defined by:
$$
\Omega_\alpha \DDD := \PPP \circ_\alpha \DDD\ .
$$
\end{defin}

We already know, by Proposition \ref{prop:adjmodcomod} that $\Omega_\alpha$ is left adjoint to $B_\alpha$. Moreover, this adjunction is enriched over cocommutative coalgebras and simplicial sets.

\begin{prop}\label{prop:enrichedadjun}
The functors $\Omega_\alpha$ and $B_\alpha$ induce functorial isomorphisms of counital cocommutative coalgebras and of simplicial sets:
 $$
 \{\Omega_\alpha \DDD , \AAA \} \simeq \{\CC, B_\alpha \AAA \}\ ,\ \ \HOM (\Omega_\alpha \DDD , \AAA ) \simeq \HOM( \DDD, B_\alpha \AAA )\ ,
 $$
 for any $\CCC$-coalgebra $\DDD$ and any $\PPP$-algebra $\AAA$;
\end{prop}

\begin{lemma}
 We have a functorial isomorphism:
 $$
 \Tw_\alpha ( \DDD \otimes \EEE , \AAA) \simeq \Tw_\alpha ( \DDD , [\EEE, \AAA])\ .
 $$
 for any $\CCC$-coalgebra $\DDD$, any $\PPP$-algebra $\AAA$ and any counital cocommutative coalgebra $\EEE$.
\end{lemma}

\begin{proof}[Proof]
The set of morphisms of graded $\mbk$-modules from $\DD \otimes \EE$ to $\Aa$ is in bijection with the set of morphisms of graded $\mbk$-modules from $\DD$ to $[\EE, \Aa]$. This bijection and its inverse preserve $\alpha$-twisting morphisms.
\end{proof}

\begin{proof}[Proof of Proposition \ref{prop:enrichedadjun}]
 On the one hand, for any cocommutative coalgebra $\EEE$, we have:
\begin{align*}
  \hom_{\uCocom}(\EEE , \{\Omega_\alpha \DDD, \AAA \})  & \simeq \hom_{\Palg} (\Omega_\alpha \DDD, [\EEE, \AAA])\\
  & \simeq \Tw_\alpha ( \DDD, [\EEE, \AAA]) \simeq  \Tw_\alpha ( \DDD \otimes \EEE , \AAA)\\&  \simeq \hom_{\Ccog} (\DDD \otimes \EEE , B_\alpha \AAA)\\& \simeq \hom_{\uCocom} (\EEE , \{ \DDD , B_\alpha \AAA \})\ .
\end{align*}
On the other hand, there is a functorial isomorphism $\HOM (\Omega_\alpha \DDD , \AAA ) \simeq \HOM( \DDD, B_\alpha \AAA )$ because the bar-cobar adjunction still works when we work with $\Omega_n$ as base ring instead of $\mbk$. 

\end{proof}

\section{Homotopy theory of algebras over an operad}

In this section, we recall a result of Hinich stating that for any dg operad $\PPP$, the category of $\PPP$-algebras admits a projective model structure whose weak equivalences are quasi-isomorphisms (see \cite{Hinich97} and \cite{BergerMoerdijk03}). Moreover, we show that the simplicial enrichment of the category of $\PPP$-algebras that we described above gives models for the mapping spaces. Finally, we show that the enrichment over cocommutative coalgebras introduced in Section \ref{section:enrich} encodes deformation of morphisms of $\PPP$-algebras.

\subsection{Model structure on algebras over an operad} 

 We recall here results about model structures on the category of algebras over an operad. 

\begin{defin}[Right induced model structures] Consider the following adjunction.
 $$
\xymatrix{\catC \ar@<1ex>[r]^(0.5){L} & \catD \ar@<1ex>[l]^(0.5){R}}
$$
 Suppose that $\catC$ admits a cofibrantly generated model structure. We say that $\catD$ admits a model structure \textit{right induced} by the adjunction $L \dashv R$ if it admits a model structure whose weak equivalences (resp. fibrations) are the morphisms $f$ such that $R(f)$ is a weak equivalence (resp. a fibration) and whose generating cofibrations (resp. generating acyclic cofibrations) are the images under $L$ of the generating cofibrations (resp. generating acyclic cofibrations) of $\catC$.
 \end{defin}

\begin{defin}[Admissible operad]\leavevmode
 An operad $\PPP$ is said to be \textit{admissible} if the category of $\PPP$-algebras admits a projective model structure, that is a model structure right induced by the adjunction 
 $$
\xymatrix{\dgMod \ar@<1ex>[r]^(0.45){\PPP \circ -} & \Palg \ar@<1ex>[l]}
$$
whose right adjoint is the forgetful functor.
\end{defin}

\begin{thm}\cite{Hinich97}
 Any nonsymmetric operad is admissible. When the characteristic of the field $\mbk$ is zero, then any operad is admissible.
\end{thm}

\subsection{Mapping spaces}

The simplicial enrichments of the category of $\PPP$-algebras described above give us models for the mapping spaces.

\begin{prop}\label{prop:enrichhomotalg}
Suppose that the characteristic of the field $\mbk$ is zero. Let $\PPP$ be a dg operad. The assignment $\AAA,\AAA'\mapsto \HOM(\AAA,\AAA')$ (resp. $\AAA,\AAA'\mapsto \HOM(\AAA,\AAA')^{ns}$ in the nonsymmetric context) defines an homotopical enrichment of the category of $\PPP$-algebras over the category of simplicial sets. Moreover, for any cofibrant $\PPP$-algebra $\AAA$ and any fibrant $\PPP$-algebra $\AAA'$, the simplicial set $\HOM(\AAA,\AAA')$ (resp. $\HOM(\AAA,\AAA')^{ns}$ in the nonsymmetric context) is a model of the mapping space $Map(\AAA,\AAA')$.
\end{prop}

\begin{rmk}
 The characteristic zero assumption is not necessary in the nonsymmetric context.
\end{rmk}

\begin{proof}[Proof]
Let $f:\AAA \ra \AAA'$ and $g: \BBB \ra \BBB'$ be respectively a cofibration and a fibration of $\PPP$-algebras. Let $h:X \ra Y$ be a monomorphism of simplicial sets. We suppose that $X$ and $Y$ are colimits of finite diagrams made up of simplicies $\Delta[n]$. Consider a square as follows.
$$
\xymatrix{X \ar[r] \ar[d] & \HOM(\AAA',\BBB) \ar[d]\\ Y  \ar[r]& \HOM(\AAA',\BBB')\times_{\HOM(\AAA,\BBB')} \HOM(\AAA,\BBB)}
$$
By Proposition \ref{prop:sullivanfinitelim}, it induces the following square
$$
\xymatrix{\Aa \ar[r] \ar[d] & \Omega_Y \otimes \BBB \ar[d]\\ \Aa'  \ar[r]&  \Omega_Y \otimes \BBB' \times_{\Omega_X \otimes \BBB'}\Omega_X \otimes \BBB\ ,}
$$
which has a lifting whenever $f$, $g$ or $h$ is a weak equivalence; indeed, by Propostion \ref{prop:sullivanalgadj}, the map $\Omega_Y \to \Omega_X$ is a fibration and it is an acyclic fibration whenever $h$ is an acyclic cofibration. Besides, to prove that $\HOM(\AAA,\AAA')$ is a model of the mapping space $Map(\AAA,\AAA')$, it suffices to notice that $\{\Omega_n \otimes \AAA'\}_{n \in \mbn}$ is a Reedy fibrant resolution of the constant simplicial $\PPP$-algebra $\AAA'$. The result in the nonsymmetric context can be proven in a similar way.
\end{proof}

\subsection{Deformation theory of morphisms of algebras over an operad}
We know that the category of $\PPP$-algebras is enriched over the category $\uCocom$ of cocommutative coalgebras. In this subsection, we show that for any $\PPP$-algebras $\AAA$ and $\BBB$, the cocommutative coalgebra $\{\AAA,\BBB\}$ encodes the deformation theory of morphisms from $\AAA$ to $\BBB$. We suppose in this subsection that the field $\mbk$ is algebraically closed.\\

Any morphism of $\PPP$-algebras $f: \AAA \to \BBB$ defines a defomation problem $\Def(f)$.
\begin{align*}
 \Artinalg &\to \sSet\\
 R &\mapsto \Map (\AAA, \BBB \otimes R)\times^h_{\Map (\AAA, \BBB )} \{f\}\simeq \HOM(\AAA,\BBB \otimes R) \times_{\HOM(\AAA,\BBB)} \{f\} \ .
\end{align*}

The following theorem is a direct consequence of a result by Chuang, Lazarev and Mannan (\cite[Theorem 2.9]{ChuangLazarevMannan14}). It is proven in Appendix.

\begin{thm}\label{thm:decomp}
Suppose that the base field $\mbk$ is algebraically closed and that its characteristic is zero. Let $\CCC= (\CC,\Delta,\epsilon)$ be a dg cocommutative coalgebra over an algebraically closed field of characteristic zero and let $A$ be its set of graded atoms. There exists a unique decomposition $\CCC \simeq \bigoplus_{a \in A} \CCC_a$ where $\CCC_a$ is a sub-coalgebra of $\CCC$ which contains $a$ and which belongs to the category $\uNilCocom$. Moreover, a morphism of dg cocommutative coalgebras $f: \bigoplus_{a \in A} \CCC_a \to \bigoplus_{b \in B} \DDD_b$ is the data of a function $\phi: A \to B$ and of a morphism $f_a: \CCC_a \to \DDD_{\phi(a)}$ for any $a \in A$. 
\end{thm}

We know from Proposition \ref{prop:atom} that a morphism $f$ of $\PPP$-algebras from $\AAA$ to $\BBB$ is a dg atom of the cocommutative coalgebra $\{\AAA,\BBB\}$. Applying Theorem \ref{thm:decomp} to the cocommutative coalgebra $\{\AAA,\BBB\}$, we obtain the conilpotent cocommutative coalgebra $\{\AAA,\BBB\}_f$. This is in particular an Hinich coalgebra which encodes a deformation problem $R \mapsto \Map (R^*,\{\AAA,\BBB\}_f)$. We show in the next proposition that this deformation problem is $\Def(f)$.

\begin{thm}\label{prop:deformation}
 Suppose that $\AAA$ is a cofibrant $\PPP$-algebra. Then, the deformation problem induced by the conilpotent cocommutative coalgebra $\{\AAA, \BBB\}_f$ is $Def(f)$.
\end{thm}

\begin{lemma}\label{lemma:cocomreedy}
If $\AAA$ is a cofibrant $\PPP$-algebra, the simplicial Hinich coalgebra
$$
\{\AAA,\Omega_n \otimes \BBB \}_f
$$ 
is a Reedy fibrant replacement of the constant simplicial Hinich coalgebra $\{\AAA,\BBB\}_f$.
\end{lemma}

\begin{proof}[Proof]
Let $g:X \to Y$ be a monomorphism of simplicial sets which are finite colimits of standard simplicies $\Delta[n]$. Let $h: \CCC_1 \to \CCC_2$ be a monomorphism of Hinich coalgebras. Consider the following square:
$$
 \xymatrix{\CCC_1 \ar[d] \ar[r]  &  \{\AAA,  \Omega_{Y} \otimes \BBB\}_f \ar[d]^{\{\AAA, \Omega[g] \otimes \BBB\}}\\
 \CCC_2 \ar[r] & \{\AAA,  \Omega_{X} \otimes \BBB \}_f\ .}
 $$
Any morphism of cocommutative coalgebras from a conilpotent cocommutative coalgebra $\CCC$ to $\{\AAA,\BBB\}$ such that the atom of $\CCC$ targets the atom $f$ of $\{\AAA,\BBB\}$ is a morphism from $\CCC$ to $\{\AAA,\BBB\}_f$. So, lifting the previous square amounts to lift the following square of $\PPP$-algebras.
 $$
 \xymatrix{\emptyset \ar[r] \ar[d] & [\CCC_2,  \Omega_{Y} \otimes \BBB] \ar[d]\\ \AAA \ar[r] & [\CCC_1,   \Omega_{Y} \otimes \BBB] \times_{[\CCC_1,  \Omega_{X} \otimes \BBB]}  [\CCC_2,  \Omega_{Y} \otimes \BBB]}
 $$
This is possible whenever, $g$ or $h$ is a weak equivalence, since any weak equivalence of Hinich coalgebras is in particular a quasi-isomorphism. So, in particular, any face map $\{\AAA, \Omega_{n+1} \otimes \BBB\} \to\{\AAA, \Omega_{n} \otimes \BBB\}$ is an acyclic fibration of Hinich coalgebras and, for any integer $n\in \mbn$, the morphism $\{\AAA,  \Omega_{n} \otimes \BBB\} \to \{\AAA, \Omega_{\partial\Delta[n]} \otimes \BBB\}$ is a fibration.  
\end{proof}

\begin{proof}[Proof of Theorem \ref{prop:deformation}]
By Lemma \ref{lemma:cocomreedy}, the deformation problem induced by the Hinich coalgebra $\{\AAA,\BBB\}_f$ is equivalent to the following deformation problem:
$$
R \in \Artinalg \mapsto (\hom_{\Hinich}(R^*, \{\AAA, \Omega_n \otimes \BBB\}_f)_{n \in \mbn}\ .
$$ 
We have:
\begin{align*}
\hom_{\Hinich} (R^*,\{\AAA,  \Omega_n \otimes \BBB\}_f) &\simeq \hom_{\uCocom} (R^*,\{\AAA,  \Omega_n \otimes \BBB\}) \times _{\hom_{\uCocom}(\mbk,\{\AAA, \Omega_n \otimes \BBB\})} \{f\}\\
&\simeq \hom_{\Palg} (\AAA \triangleleft R^*,  \Omega_n \otimes \BBB) \times _{\hom_{\Palg}(\AAA, \Omega_n \otimes \BBB)} \{f\}\\
&\simeq \hom_{\Palg} (\AAA , R \otimes \Omega_n \otimes \BBB) \times _{\hom_{\Palg}(\AAA, \Omega_n \otimes \BBB)} \{f\}\ .
\end{align*} 
Since the simplicial sets $(\hom_{\Palg} (\AAA,  R \otimes  \Omega_n \otimes \BBB))_{n \in \mbn}$ and $(\hom_{\Palg}(\AAA, \Omega_n \otimes \BBB))_{n \in \mbn}$ are Kan complexes and models of respectively $\Map (\AAA,R \otimes\BBB )$ and $\Map(\AAA,\BBB)$ and since the map between them is a fibration, then the simplicial set
$$
(\hom_{\Palg} (\AAA ,  R \otimes \Omega_n \otimes \BBB) \times _{\hom_{\Palg}(\AAA, \Omega_n \otimes \BBB)} \{f\})_{n \in\mbn}
$$
is a model of the homotopy pullback $\Map (\AAA, R \otimes \BBB )\times^h_{\Map (\AAA, \BBB )} \{f\}$.
\end{proof}


\section{Model structures on coalgebras over a cooperad}

In this section, we show that, for any operadic twisting morphism $\alpha:\CCC \to \PPP$, the  projective model structure on the category of $\PPP$-algebras can be transferred through the cobar construction functor $\Omega_\alpha$ to the category of $\CCC$-coalgebras. This result is in the vein of similar results by Hinich \cite{Hinich01}, Lefevre-Hasegawa \cite{LefevreHasegawa03}, Vallette \cite{Vallette14} and Positselski \cite{Positselski11}. However, we use a new method for the proof that uses the presentability of the category of algebras over an operad and of the category of coalgebras over a curved conilpotent cooperad; specifically, we use a theorem proved by Bayeh, Hess, Karpova, Kedziorek, Riehl and Shipley in \cite{BHKKRS14} and \cite{HKRS15}. 

\subsection{Model structure induced by a twisting morphism}

\begin{defin}[Left induced model structures]
 Consider the following adjunction.
  $$
\xymatrix{\catC \ar@<1ex>[r]^(0.5){L} & \catD \ar@<1ex>[l]^(0.5){R}}
$$
Suppose that $\catD$ admits a model structure. We say that $\catC$ admits a model structure left induced by the adjunction $L \dashv R$ if it admits a model structure whose weak equivalences (resp. cofibrations) are the morphisms $f$ such that $L(f)$ is a weak equivalence (resp. a cofibration).
\end{defin}

Here is the main theorem of the present article.

\begin{thm}\leavevmode \label{thmprincipal}
Let $\PPP$ be a dg operad, let $\CCC$ be a curved conilpotent cooperad and let $\alpha$ be an operadic twisting morphism between them. Suppose that the characteristic of the base field $\mbk$ is zero. We know that the category of $\PPP$-algebras admits a projective model structure. Then, the category of $\CCC$-coalgebras admits a model structure left induced by the adjunction $\Omega_\alpha \dashv B_\alpha$. We call it the $\alpha$-model structure. In the nonsymmetric context, we can drop the assumption that the characteristic of the field $\mbk$ is zero.
\end{thm}

To prove this theorem, we will use the following result.

\begin{thm}\cite{BHKKRS14}\cite{HKRS15} \label{thm:hessandcie}
Consider an adjunction
 $$
\xymatrix{\catC \ar@<1ex>[r]^(0.5){L} & \mathsf M \ar@<1ex>[l]^(0.5){R}}
$$
between presentable categories. Suppose that $\mathsf M$ is endowed with a cofibrantly generated model structure. Then, there exists a left induced model structure on $\mathsf C$ if the morphisms which have the right lifting property with respect to left induced cofibrations are left induced weak equivalences. In particular, this is true if the category $\mathsf C$ has a cofibrant replacement functor, and if any cofibrant object has a cylinder.
\end{thm}

From now on, a \textit{weak equivalence} (resp. \textit{cofibration}) of $\CCC$-coalgebras is a morphism whose image under $\Omega_\alpha$ is a weak equivalence (resp. cofibration). An \textit{acyclic cofibration} is a morphism which is both a cofibration and a weak equivalence. A \textit{fibration} is a morphism which has the right lifting property with respect to all acyclic cofibrations and an \textit{acyclic fibration} is a morphism which is both a fibration and a weak equivalence. Here is the proof.

\begin{proof}[Proof of Theorem \ref{thmprincipal}]
 Proposition \ref{prop:cofib} ensures us that the cofibrations of the category of $\CCC$-coalgebras are the monomorphisms. Hence, any object is cofibrant. Then, Propostion \ref{prop:proptech} provides us with a cylinder for any object if the characteristic of $\mbk$ is zero. In the nonsymmetric context, Proposition \ref{prop:proptechns} provides us with a cylinder. We conclude by Theorem \ref{thm:hessandcie}.
\end{proof}

\subsection{Cofibrations}

\begin{prop}\label{prop:cofib}
 The class of cofibrations of $\CCC$-coalgebras is the class of monomorphisms.
\end{prop}

\begin{lemma}\label{lem:cofib}
Let $f : \DDD \ra \EEE$ be a monomorphism of $\CCC$-coalgebras such that $\Delta (\EE) \subset \CC \circ f(\DD)$. Then, $f$ is a cofibration.
\end{lemma}

\begin{proof}[Proof]
 We can decompose the graded $\mbk$-module $\EE$ as $\EE = \DD \oplus \FF$. The coderivation $d_\EE$ corresponds then to the following matrix.
 $$
 \begin{pmatrix}
 d_{\DD} & \phi \\
 0 & d_{\FF}
\end{pmatrix}
 $$
 Consider the following diagram of $\PPP$-algebras:
 $$
 \xymatrix{\PP \circ (s^{-1} \FF) \ar[r] \ar@{^(->}[d]^{\PP(\phi)} & \Omega_{\alpha} \DD \\
 \PP \circ (s^{-1} \FF \oplus \FF)}
 $$
 where the horizontal map sends $ s^{-1}x$ to $\phi (x) + (\alpha \circ Id) \Delta (x)$. The vertical map is a cofibration and $f$ is the pushout of this vertical map along the horizontal map. Hence, $f$ is a cofibration.
\end{proof}

\begin{proof}[Proof of Proposition \ref{prop:cofib}]
 Let $f: \DDD \ra \EEE$ be a cofibration. Then $\Omega_{\alpha} (f)$ it is a monomorphism. Since the following square is commutative, then $f$ is a monomorphism.
 $$
 \xymatrix{\Omega_\alpha \DDD \ar[r] & \Omega_\alpha \EEE\\
 \DDD \ar[u] \ar[r] & \EEE \ar[u]}
 $$
 Conversely, if $f$ is a monomorphism, then, it can be decomposed into the transfinite composition of the maps $f_n = \DDD + F^{rad}_{n-1} \EEE \ra \DDD + F^{rad}_n \EEE$. Since the maps $f_n$ satisfy the conditions of Lemma \ref{lem:cofib}, they are cofibrations. So $f$ is a cofibration.
\end{proof}

\subsection{Filtered quasi-isomorphism}

\begin{defin}[Filtered quasi-isomorphism]
Let $\DDD$ and $\EEE$ be two $\CCC$-coalgebras. A morphism of $\CCC$-coalgebras $f$ from $\DDD$ to $\EEE$ is said to be a \textit{filtered quasi-isomorphism} if the induced morphisms between the graded complexes relative to the coradical filtrations are quasi-isomorphisms, that is, if for any integer $n$, the morphism from $G^{rad}_n \DDD$ to $G^{rad}_n \EEE$ is a quasi-isomorphism.
\end{defin}
 
\begin{prop}\label{prop:filtered:we}
If the characteristic of $\mbk$ is zero, a filtered quasi-isomorphism is a weak equivalence of $\CCC$-coalgebras. The characteristic zero assumption is not necessary in the nonsymmetric context.
\end{prop}

We will use the following classical result.

\begin{thm}\cite[XI.3.4]{MacLane95}\label{maclane-homology}
 Let $f: A \rightarrow B$ be a map of filtered chain complexes. Suppose that the filtrations are bounded below and exhaustive. If for any integer $n$, the map $G_n A \rightarrow G_n B$ is a quasi-isomorphism, then $f$ is a quasi-isomorphism.
\end{thm}
 
\begin{proof}[Proof of Proposition \ref{prop:filtered:we}]
 Consider the following filtration on $\Omega_\alpha \DDD$ (resp. $\Omega_\alpha \EEE$)
 $$
 F_n \Omega_\alpha \DDD= \PP (0) \oplus \sum_{k\geq 1 \atop p_1 + \cdots + p_k = n} \PP(k) \otimes_{\mbs_k} \big( F^{rad}_{p_1} \DD \otimes \cdots  \otimes F^{rad}_{p_k} \DD\big)\ .
 $$
It is clear that $\Omega_\alpha(f)$ sends $F_n \Omega_\alpha \DDD$ to $F_n\Omega_\alpha \EEE$ for any integer $n$. Moreover, we have
$$
G_n \Omega_\alpha \DDD=\sum_{k\geq 1 \atop p_1 + \cdots + p_k = n} \PP (k) \otimes_{\mbs_k} \big( G^{rad}_{p_1} \DD \otimes \cdots  \otimes G^{rad}_{p_k} \DD \big)\ .
$$
Then, by the operadic Kunneth formula, $G_n (\Omega_\alpha(f)): G_n \Omega_\alpha \DDD \rightarrow G_n \Omega_\alpha \EEE$ is a quasi-isomorphism for any $n \in \mbn$. Hence, by Theorem \ref{maclane-homology}, $\Omega_u(f)$ is a quasi-isomorphism.
\end{proof}

\begin{rmk}\label{rmk:admissiblefilt}
The coradical filtration is not the only filtration whose notion of filtered quasi-isomorphism gives us weak equivalences. An exhaustive filtration $(F_n \DD)_{n \in  \mbn}$ is called \textit{admissible} if:
 $$
\begin{cases}
\Delta (F_n \DD) \subset  \sum_{p_0 + p_1 + \cdots + p_k = n}  F_0^{rad}\CC \otimes_{\mbs_k} \big( F_{p_1} \DD \otimes \cdots \otimes F_{p_k} \DD \big)\ ,\\
  d (F_n \DD) \subset F_n \DD\ ,\\
  d^2 (F_n \DD) \subset F_{n-1} \DD\ .
\end{cases}
 $$
 Using similar arguments as in the proof just above, we can prove that a filtered quasi-isomorphism with respect to two admissible filtrations is a weak equivalence.
\end{rmk}

\subsection{Cylinder object}

\begin{prop}\label{prop:proptech}
Let $\DDD=(\DD,\Delta_\DDD)$ be a $\CCC$-coalgebra. Let $\AAA=(\Aa,\gamma_\AAA)$ be a cylinder of $\Omega_\alpha (\DDD)$ such that the structural map $p: \AAA \ra \Omega_\alpha (\DDD)$ is an acyclic fibration. The following diagram
$$
\xymatrix{B_\alpha \Omega_\alpha (\DDD \oplus  \DDD) \ar[r] & B_\alpha (\AAA) \ar[r]^(0.4){B_\alpha p} & B_\alpha \Omega_\alpha (\DDD) \\
\DDD \oplus \DDD \ar[u] \ar[r] &\EEE \ar[u] \ar[r] & \DDD \ar[u]}
$$
where $\EEE:=  B_\alpha (\AAA)\times_{B_\alpha (\Omega_\alpha \DDD)} \DDD$, provides us with a cylinder $\EEE=(\EE,\Delta_\EEE)$ for the $\CCC$-coalgebra $\DDD$. 
\end{prop}

\begin{lemma}\label{lemma:pullback}
 The pullback $\EEE$ is the final sub-graded $\CCC^{grad}$-coalgebra of $B_\alpha \AAA$ whose image in $B_\alpha \Omega_\alpha \DDD$ is in the image of the morphism $\DDD \ra B_\alpha \Omega_\alpha \DDD$.
\end{lemma}

\begin{proof}[Proof]
 Let $\FFF=(\FF,\Delta_\FFF)$ be the final sub-graded $\CCC^{grad}$-coalgebra of $B_\alpha \AAA$ whose image in $B_\alpha \Omega_\alpha \DDD$ is in the image of the morphism $\DDD \ra B_\alpha \Omega_\alpha \DDD$. Proving that $\FFF$ is the underlying $\CCC^{grad}$-coalgebra of $\EEE$ amounts to prove that $\FF$ is stable under the coderivation $D$ of $B_\alpha \AAA$. We prove it by induction on the coradical filtration of $\FF$. First, by the maximality property of $\FFF$, $F^{rad}_0\FF$ is stable under $D$. Then suppose that $F^{rad}_n\FF$ is stable under $D$ for an integer $n \geq 0$. Let $x$ be an element of $F^{rad}_{n+1} \FF$. On the one hand $B_\alpha(p)D(x) = D (B_\alpha(p)(x))$. Since $B_\alpha(p)(x)$ is in the image of $\DDD$ and since this image is stable under the coderivation of $B_\alpha \Omega_\alpha \DDD$, then $B_\alpha(p)D(x)$ is the image of $\DDD$. On the other hand, we have
 $$
 \Delta (D(x)) = (d_\CC \circ Id + Id \circ' D)\Delta(x)\ .
 $$
 So, since $\Delta(x) - 1_\CCC \otimes x \in \CC\circ (F_n^{rad} \FF)$, and since $F_n^{rad} \FF$ is stable under $D$ by the inductivity assumption, then
 $$
 \Delta (D(x)) - 1_\CCC \otimes D(x) = (d_\CC \circ Id + Id \circ' D)(\Delta(x) -1_\CCC \otimes x) \in \CCC\circ (F_n^{rad} \FF)\ . 
 $$
 By these two points, $\FF + \mbk \cdot D(x)$ is a sub- graded $\CCC^{grad}$-coalgebra of $B_\alpha (\AAA)$ whose image in $B_\alpha \Omega_\alpha \DDD$ is in the image of $\DDD$. By the maximality property of $\FFF$, then $D(x) \in \FF$. So, $F^{rad}_{n+1} \FF$ is stable under $D$. Hence, by induction $\FF$ is stable under $D$.
\end{proof}

To prove Proposition \ref{prop:proptech}, we will show that the pullback map $\EEE \ra \DDD$ is a filtered quasi-isomorphism. Since $\Omega_\alpha \DDD$ is a cofibrant $\PPP$-algebra, there exists a right inverse $q: \Omega_\alpha \DDD \ra \AAA$ to the acyclic fibration $p:\AAA \ra \Omega_\alpha\DDD$. Then, let us decompose $\Aa$ as $\Aa= \Omega_\alpha \DDD \oplus K$. The chain complex $K$ is acyclic. So, let $h:K \rightarrow K$ be a degree $1$ map such that $\partial (h)=Id_K$. It can be extended to a map
\begin{align*}
 B_\alpha \AAA \twoheadrightarrow \Aa \twoheadrightarrow K &\ra \Aa \\
 x &  \mapsto h(x)\ .
\end{align*}
The zero map is a coderivation on the graded cooperad $\CCC^{grad}$. Then, let $D_h$ be the degree $1$ coderivation of $(B_\alpha  \AAA)^{grad} $ relative to the zero coderivation on $\CCC^{grad}$ and whose projection on $\Aa$ is $h$. In other words, $
D_h = Id_\CCC \circ' h$.

\begin{lemma}
 The sub-$\CCC$-coalgebra $\EEE$ of $B_\alpha \AAA$ is stable under $D_h$.
\end{lemma}

\begin{proof}[Proof]
By Lemma \ref{lemma:pullback}, it suffices to prove that the final sub-graded $\CCC^{grad}$-coalgebra of $B_\alpha \AAA$ whose image in $B_\alpha \Omega_\alpha \DDD$ lies inside $\DDD$, is stable under $D_h$. To that purpose, we use the same arguments as in the proof of Lemma \ref{lemma:pullback} and the fact that $B_\alpha (p) D_h= 0$.
\end{proof}

\begin{proof}[Proof of Proposition \ref{prop:proptech}]
Since the map $\DDD \oplus \DDD \ra \EEE$ is a monomorphism and so a cofibration, it suffices to show that the map $\EEE \ra \DDD$ is a weak equivalence. We show that it is a filtered quasi-isomorphism. Let $n\in \mbn$ ; let us show that the map $G_n \EE \ra G_n \DD$ is a quasi-isomorphism. To that purpose, consider the following filtration on $B_\alpha \AAA$ : 
$$
F'_k B_\alpha \AAA := \sum_{i \leq k} \CC \otimes_{\mbs} (K^{\otimes i} \otimes (\Omega_\alpha \DDD)^{\otimes j})
$$
This filtration is stable under the coderivations $d$ and $D_h$ and it induces a filtration on $G^{rad}_n \EE$. It is clear that the morphism $G'_0 G^{rad}_n \EE \ra G^{rad}_n \DD$ is an isomorphism. Moreover, for any integer $k \geq 1$, $\partial(D_h)= k . Id$ on $G'_k G^{rad}_n \EE$. Since the characteristic of $\mbk$ is zero, $G'_k G^{rad}_n \EE$ is acyclic. By Theorem \ref{maclane-homology}, the map $G_n \EE \ra G_n \DD$ is a quasi-isomorphism.
\end{proof}

\begin{prop}\label{prop:proptechns}
In the nonsymmetric context, $\DDD \otimes DK^c(\Delta[1])$ provides us with a cylinder for the $\CCC$-coalgebra $\DDD$.
\end{prop}

\begin{proof}[Proof]
Since $G^{rad}_n (\DD \otimes DK^c(\Delta[1])) = G^{rad}_n(\DD) \otimes DK^c(\Delta[1])$ and since the map $DK^c(\Delta[1]) \ra \mbk$ is a quasi-isomorphism, then $\DD \otimes DK^c(\Delta[1]) \ra \DD$ is a filtered quasi-isomorphism and so a weak equivalence.
\end{proof}

\subsection{Simplicial enrichment in the nonsymmetric context}

\begin{prop}\label{prop:enrichhomotop}
 In the nonsymmetric context, the assignment $\DDD, \DDD' \mapsto \HOM^{ns}(\DDD,\DDD')$ defines an homotopical enrichment of the category of $\CCC$-coalgebras together its $\alpha$-model structure over the category of simplicial sets. Moreover, if $\DDD'$ is fibrant $ \HOM^{ns}(\DDD,\DDD')$ provides a model for the mapping space $\Map(\DDD,\DDD')$.
\end{prop}

\begin{proof}[Proof]
 Let $f: \DDD \ra \DDD'$ be a cofibration of $\CCC$-coalgebras, let $g: \EEE \ra \EEE'$ be a fibration of $\CCC$-coalgebras and let $h:X \ra Y$ be a cofibration (i.e. a monomorphism) of simplicial sets. Consider the following square.
 $$
 \xymatrix{X \ar[r] \ar[d] & \HOM^{ns}(\DDD', \EEE) \ar[d]\\
  Y \ar[r] & \HOM^{ns}(\DDD', \EEE') \times_{\HOM^{ns}(\DDD,\EEE')}  \HOM^{ns}(\DDD,\EEE)}
 $$
 It induces a square 
 $$
 \xymatrix{ \DDD' \otimes DK^c(X) \coprod_{\DDD \otimes DK^c(X)}  \DDD \otimes DK^c(Y) \ar[r] \ar[d] & \EEE \ar[d] \\
 \DDD' \otimes DK^c(Y)  \ar[r] & \EEE'\ .}
 $$
 The left vertical map is a monomorphism and so a cofibration.
\begin{itemize}
 \itemt If the morphism $g:\EEE \ra \EEE'$ is an acyclic fibration, then the square has a lifting.
 \itemt Suppose that the morphism $h:X \ra Y$ is an acyclic cofibration. Then the morphism $\DDD \otimes DK^c(X) \ra \DDD \otimes DK^c(Y)$ is a filtered quasi-isomorphism and a cofibration, so it is an acyclic cofibration. Hence, its pushout $\DDD' \otimes DK^c(X) \ra \DDD' \otimes DK^c(X) \coprod_{\DDD \otimes DK^c(X)} \DDD \otimes DK^c(Y)$ is also an acyclic cofibration. Moreover, the map $\DDD' \otimes DK^c(X) \ra \DDD' \otimes DK^c(Y)$ is a filtered quasi-isomorphism and so a weak equivalence. So, by the 2-out-of-3 rule, the morphism $ \DDD' \otimes DK^c(X) \coprod_{ \DDD\otimes DK^c(X)}  \DDD \otimes DK^c(Y) \ra  \DDD' \otimes DK^c(Y)$ is a weak equivalence. Since, it is a cofibration, it is an acyclic cofibration and the square has a lifting.
 \itemt Suppose that the morphism $f:\DDD \ra \DDD'$ is an acyclic cofibration. Then, the morphism $\DDD \otimes DK^c(X) \ra \DDD' \otimes DK^c(X)$ is an acyclic cofibration. This is a consequence of the fact that, $\Omega_\alpha (\DDD \otimes DK^c(X)) =  (\Omega_\alpha \DDD) \triangleleft  DK^c(X)$, and that for any fibration of $\PPP$-algebras $\AAA \ra \AAA'$, the morphism $[DK^c(X), \AAA] \ra [DK^c(X), \AAA']$ is also a fibration. Then, the same arguments as in the previous point show us that the morphism $\DDD' \otimes DK^c(X) \coprod_{\DDD \otimes DK^c(X)} \DDD \otimes DK^c(Y) \ra \DDD' \otimes DK^c(Y)$ is an acyclic cofibration and so the square has a lifting.
\end{itemize}
In particular, we have proved that any coface map $\DDD \otimes DK^c(\Delta[n]) \to \DDD \otimes DK^c(\Delta[n+1])$ is an acyclic cofibration. Moreover, for any integer $n$, the morphism $\DDD \otimes DK^c(\partial\Delta[n]) \to \DDD \otimes DK^c(\Delta[n])$ is a cofibration. So, the cosimplicial $\CCC$-coalgebra $(\DDD \otimes DK^c(\Delta[n]))_{n \in \mbn}$ is a Reedy cofibrant  replacement of the constant cosimplicial $\CCC$-coalgebra $\DDD$. So, if $\DDD'$ is fibrant the simplicial set $\left(\hom_{\Ccog}(\DDD \otimes DK^c(\Delta[n]), \DDD')\right)_{n \in \mbn}$ which is isomorphic to $\HOM^{ns}(\DDD,\DDD')$ is a model for the mapping space $\Map(\DDD,\DDD')$.
\end{proof}

\subsection{Changing operads and cooperads} In this subsection, we explore how the left induced model structure on coalgebras over a curved conilpotent cooperad is modified when we change the underlying operadic twisting morphism. This is inspired by \cite{DrummondColeHirsh14} where a similar study is done in the context augmented dg operads and dg conilpotent cooperads. \\

Recall first that a morphism of dg operads $f:\PPP \ra \QQQ$ induces an adjunction between their categories of algebras
 $$
\xymatrix{\Palg \ar@<1ex>[r]^(0.5){f_!} & \Qalg \ar@<1ex>[l]^(0.5){f^*}}
$$
whose right adjoint $f^*$ sends a $\QQQ$-algebra $\AAA$ to the same underlying chain complex. This adjunction is a Quillen adjunction with respect to the projective model structures; see \cite{BergerMoerdijk03}. Similar things happen for coalgebras over curved conilpotent cooperads.

\begin{prop}
 Let $f:\CCC \ra \DDD$ be a morphism of curved conilpotent cooperads. It induces an adjunction between their categories of coalgebras
 $$
\xymatrix{\Ccog \ar@<1ex>[r]^(0.5){f_*} & \Dcog \ar@<1ex>[l]^(0.5){f^!}}
$$
whose left adjoint $f_*$ sends a $\CCC$-coalgebra $\EEE$ to the same underlying graded $\mbk$-module.
\end{prop}
 
\begin{proof}[Proof]
Let $\EEE=(\EE,\Delta, d)$ be a $\CCC$-coalgebra. It has a structure of $\DDD$-coalgebra defined by the composite map
$$
\EE \xrightarrow{\Delta} \CCC \circ \EE \xrightarrow{f \circ Id} \DDD \circ \EE\ .
$$ 
This defines the functor $f_*$. Since it preserves colimits and since the category of $\CCC$-coalgebras and the category of $\DDD$-coalgebras are presentable, then $f_*$ has a right adjoint by Proposition \ref{prop:presentfolk}.
\end{proof}

Besides, let us fix a dg operad $\PPP$. The canonical operadic twisting morphism $\pi: B_c\PPP \ra \PPP$ is universal in the sense that any operadic twisting morphism $\alpha$ from a curved conilpotent cooperad $\CCC$ to $\PPP$ is equivalent to a morphism of curved cooperads $f$ from $\CCC$ to $B_c \PPP$; then, $\alpha=\pi f$. In that context, the cobar functor $\Omega_\alpha$ can be decomposed as $
\Omega_\alpha = \Omega_\pi\  f_*\ ,
$ and the $\alpha$-model structure on the category of $\CCC$-coalgebras is the model structure left induced by the $\pi$-model structure on the category of $B_c \PPP$-coalgebras.\\

On the other hand, let us fix a curved conilpotent cooperad $\CCC$. The canonical operadic twisting morphism $\iota: \CCC \to \Omega_u \CCC$ is universal in the sense that any operadic twisting morphism $\alpha: \CCC \ra \PPP$ is equivalent to the data of a morphism of operads $f$ from $\Omega_u \CCC$ to $\PPP$; then $\alpha =f \iota$. A direct consequence of the following proposition is that the model structure on $\CCC$-coalgebras induced by the universal operadic twisting morphism $\iota: \CCC \ra \Omega_u \CCC$ is universal in the sense that any $\alpha$-model structure is a left Bousfield localization of this $\iota$-model structure.

\begin{prop}\label{prop:bousfieldloc}
 Let $\alpha: \CCC \to \PPP$ be an operadic twisting morphism and let $f: \PPP \to \QQQ$ be a morphism of dg operads. The $(f \alpha)$-model structure on the category of $\CCC$-coalgebras is the left Bousfield localization of the $\alpha$-model structure with respect to $(f \alpha)$-weak equivalences. Moreover, if the Quillen adjunction $f_! \dashv f^*$ is a Quillen equivalence, the $(f \alpha)$-model structure coincides with the $\alpha$-model structure. 
\end{prop}

\begin{proof}[Proof]
 The cofibrations of the $\alpha$-model structure and the cofibrations of the $(f \alpha)$-model structure are both the monomorphisms. Moreover, the functor $f_!$ is a left Quillen adjoint functor. So, for any $\alpha$-weak equivalence $g$, since $\Omega_\alpha (g)$ is a weak equivalence between cofibrant objects, then $\Omega_{(f\alpha)}(g) = f_! \Omega_\alpha(g)$ is a weak equivalence. So the $\alpha$-weak equivalences are in particular $(f\alpha)$-weak equivalences. So is proven the fact that the $(f \alpha)$-model structure is a left Bousfield localization of the $\alpha$-model structure. Suppose now that the adjunction $f_! \dashv f^*$ is a Quillen equivalence. Then for any $\CCC$-coalgebra $\EEE$, the morphism
 $$
 \Omega_\alpha \EEE \ra f^* f_! \Omega_\alpha \EEE = f^* \Omega_{f\alpha} \EEE 
 $$
 is a quasi-isomorphism. Since the functor $f^*$ is the identity on the underlying chain complexes, the following commutative square ensures us that a morphism $g:\EEE \ra \EEE'$ of $\CCC$-coalgebras is a $\alpha$-weak equivalence if and only if it is an $f\alpha$-weak equivalence. 
$$
\xymatrix{f^* \Omega_{(f\alpha)} \EEE  \ar[rr]^{f^*\Omega_{(f\alpha)} (g)} && f^* \Omega_{(f\alpha)} \EEE' \\
 \Omega_\alpha \EEE \ar[u] \ar[rr]_{\Omega_\alpha (g)}  &&\Omega_\alpha \EEE'\ . \ar[u]}
$$
\end{proof}

\section{The universal model structure}

In the previous section, we studied model structures on categories of coalgebras over a curved conilpotent cooperad which are induced by an operadic twisting morphism $\alpha$. In this section, we investigate the particular case where the operadic twisting morphism is the universal twisting morphism $\iota: \CCC \ra \Omega_u \CCC$ for any curved conilpotent cooperad $\CCC$. This model structure is universal in the sense that, for any operadic twisting morphism $\alpha: \CCC \to \PPP$, the $\alpha$-model structure on the category of $\CCC$-coalgebras is obtained from the $\iota$-model structure by Bousfield localization. We will show that the adjunction $\Omega_\iota \dashv B_\iota$ is a Quillen equivalence, that the fibrant $\CCC$-coalgebras in the $\iota$-model structure are the images of the $\Omega_u \CCC$-algebras under the functor $B_\iota$, and we will describe the cofibrations, the weak equivalences and the fibrations between them. Moreover, we will prove that the enrichment of $\CCC$-coalgebras over simplicial sets that we described above computes the mapping spaces expected by the model structure.\\

We suppose here that the characteristic of the field $\mbk$ is zero. This assumption is not necessary in the nonsymmetric context.

\subsection{Quillen equivalence}

\begin{thm}\label{thm:quilleneq}
The adjunction $\Omega_\iota \dashv B_\iota$ relating $\CCC$-coalgebras to $\Omega_u \CCC$-algebras is a Quillen equivalence.
\end{thm}

\begin{proof}[Proof]
 Let us show that for any $\Omega_u \CCC$-algebra $\AAA=(\Aa,\gamma_\AAA)$, the map $\Omega_\iota B_\iota \AAA = \Omega_u \CCC \circ_\iota \CCC \circ_\iota \AAA  \ra \AAA$ is a quasi isomorphism. The coradical filtration of $ \CCC$ induces a filtration on $\Omega_u \CCC$:
\begin{align*}
& F_0 \Omega_u \CCC:= \mbk . 1 \\
& F_n \Omega_u \CCC := \mbk . 1 \oplus \sum_{i_1 + \ldots + i_k=n\atop k \geq 1} s^{-1}F^{rad}_{i_1}\ov \CC \otimes \cdots \otimes s^{-1}F^{rad}_{i_k}\ov \CC\ ,\text{ for }n\geq 1\ .
\end{align*}
It induces a filtration on $\Omega_u \CCC \circ_\iota \CCC $ and on $\Omega_u \CCC \circ_\iota \CCC \circ_\iota \AAA$:
\begin{align*}
 &F_n (\Omega_u \CCC \circ_\iota \CCC):= F_n \Omega_u \CCC (0) \oplus \sum_{i_0+ \cdots + i_k=n\atop k\geq 1} F_{i_0} (\Omega_u \CCC) (k) \otimes_{\mbs_k} \big( F^{rad}_{i_1}  \CC \otimes \cdots \otimes F^{rad}_{i_k}  \CC  \big)\ ,\\
& F_n (\Omega_u \CCC \circ_\iota \CCC \circ_\iota \AAA) := F_n(\Omega_u \CCC \circ_\iota \CCC ) \circ \Aa\ .
\end{align*}
Then $G (\Omega_u \CCC \circ_\iota  \CCC) = \Omega_u (G\CCC) \circ_{G \iota}  G\CCC$. By \cite[Lemma 6.5.14]{LodayVallette12}, the map $\Omega_u (G\CCC) \circ_{G \iota}  G\CCC \ra \II$ is a quasi-isomorphism. So, the map
$$
G (\Omega_u \CCC \circ_\iota \CCC \circ_\iota \Aa ) \ra G \Aa
$$
is a quasi-isomorphism (here $G\Aa$ is the graded complex corresponding to the constant filtration $F_n \Aa = \Aa$). Hence, by Theorem \ref{maclane-homology}, the map $\Omega_u \CCC \circ_\iota \CCC \circ_\iota \AAA \ra \AAA$ is a quasi-isomorphism. Besides, for any $\CCC$-coalgebra $\DDD$, the morphism $\DDD \ra B_\iota \Omega_\iota \DDD$ is a weak equivalence by the 2-out-of-3 rule and by the fact that the map $\Omega_\iota B_\iota \Omega_\iota \DDD \ra \Omega_\iota \DDD$ is a quasi-isomorphism. Hence, the adjunction $\Omega_\iota \dashv B_\iota$ is a Quillen equivalence.
\end{proof}

\subsection{Fibrant objects} The purpose of this subsection is to describe the fibrant objects of the $\iota$-model structure.

\begin{defin}[Quasi-cofree $\CCC$-coalgebras]
 A $\CCC$-coalgebra is said to be \textit{quasi-cofree} if its underlying $\CCC^{grad}$-coalgebra is cofree, that is isomorphic to a coalgebra of the form $\CCC^{grad} \circ \VV$. A morphism of quasi-cofree $\CCC$-coalgebras $F: \CCC \circ \VV \ra \CCC \circ \WW$ (together with choices of cogenerators $\VV$ and $\WW$) is said to be \textit{strict} if there exists a map $f: \VV \ra \WW$ such that $F= Id \circ f$.
\end{defin}

\begin{prop}\label{prop:ccoghpalg}
The functor $B_\iota$ is an embedding of the category of $\Omega_u \CCC$-algebras into the category of $\CCC$-coalgebras whose essential image is spanned by quasi-cofree $\CCC$-coalgebras. Moreover, a morphism of $\CCC$-coalgebras $B_\iota \AAA = \CCC \circ_\iota \AAA \to B_\iota \AAA' = \CCC \circ_\iota \AAA'$ is in the image of $B_\iota$ if and only if it is strict.
\end{prop}

\begin{proof}[Proof]
It is straightforward to prove that the functor $B_\iota$ is faithful and conservative. Moreover it is clear that the images of the functor $B_\iota$ are in particular quasi-cofree $\CCC$-coalgebras and strict morphisms. Conversely, let $\DDD:= \CCC \circ \Aa$ be a quasi-cofree $\CCC$-coalgebra. Its coderivation extends the degree $-1$ map $d_\Aa \oplus \gamma: \Aa \oplus \ov \CC \circ \Aa \ra \Aa$. The map $\gamma$ gives us a degree $-1$ map from $\ov \CC$ to the operad $\End_\Aa$. Since, the coderivation which extends $d_\Aa \oplus \gamma$ squares to $(\theta \circ Id)\Delta$, then $\alpha$ is a twisting morphism and so induces a morphism of operads from $\Omega_u \CCC$ to $\End_\Aa$, which is an $\Omega_u \CCC$-algebra structure on $\Aa$. Then, $\DDD \simeq B_\iota \Aa$. Besides, let $F= Id \circ f$ be a strict morphism from $B_\iota \AAA$ to $B_\iota \BBB$. Since $F$ commutes with the coderivations, then $f$ is a morphism of $\Omega_u\CCC$-algebras.
\end{proof}

\begin{thm}
The fibrant $\CCC$-coalgebras in the $\iota$-model structure are the quasi-cofree $\CCC$-coalgebras (and so the objects in the essential image of the functor $B_\iota$).
\end{thm}

\begin{proof}[Proof]
Let $\DDD$ be a fibrant object. Since the morphism $\DDD \ra B_\iota \Omega_\iota \DDD$ is an acyclic cofibration, the following square has a lifting\ .
 $$
 \xymatrix{\DDD \ar[r]^{Id} \ar[d] &  \DDD \ar[d] \\ B_{\alpha} \Omega_{\alpha} \DDD \ar[r] & {*}\ .}
 $$
Hence, $\DDD$ is a retract of a quasi-cofree $\CCC$-coalgebra. By Lemma \ref{lemma:cofreeretract}, it is a quasi-cofree $\CCC$-coalgebra. Conversely, a quasi-cofree $\CCC$-coalgebra is fibrant since it is isomorphic to the image under $B_\iota$ of a $\Omega_u\CCC$-algebra which is fibrant.
\end{proof}

\begin{lemma}\label{lemma:cofreeretract}
 A retract of a cofree graded $\CCC^{grad}$-coalgebra is a cofree graded $\CCC^{grad}$-coalgebra.
\end{lemma}

\begin{proof}[Proof]
 Let $\DDD= (\DD,\Delta_\DDD)$ be a graded $\CCC^{grad}$-coalgebra which is a retract of $\CCC \circ \VV$. On the one hand, the following diagram is a retract, that is the compositions of the horizontal maps give the identity on the bottom and on the top
 $$
 \xymatrix{G^{rad}_n\DD \ar[r] \ar[d] & G^{rad}_n (\CC \circ \VV) \ar[d] \ar[r] & G^{rad}_n\DD \ar[d]\\
 (G^{rad}_n \CC) \circ F^{rad}_0 \DD \ar[r] & (G^{rad}_n \CC) \circ F^{rad}_0 (\CC \circ \VV) \ar[r]& (G^{rad}_n \CC) \circ F^{rad}_0 \DD\ .}
 $$
 Since the middle vertical map is an isomorphism, then all the vertical maps are isomorphisms. On the other hand, the map $\epsilon \circ Id: \CC \circ \VV \ra \VV = F^{rad}_0 \CC \circ \VV$ gives us a map $\DD \ra F_0 \DD$ and hence a morphism of graded $\CCC$-coalgebras $f:\DDD \ra \CCC \circ F_0 \DD$. Let us show that $f$ is an isomorphism. It is clear that the map $F_0 \DD \ra F_0 (\CC \circ F_0 \DD)$ is an isomorphism. For any integer $n \geq 1$, the following diagram is commutative
 $$
 \xymatrix{G_n(\DD) \ar[rr]^f \ar[d]_{\Delta} && G_n (\CC \circ F_0 \DD) \ar[d]^{\Delta}\\
 (G_n \CC) \circ F_0 \DD \ar[rr]_(0.4){Id \circ f} && (G_n \CC) \circ F_0 (\CC \circ F_0 \DD) =  (G_n \CC) \circ F_0 \DD\ . }
 $$
 Since the vertical maps are isomorphisms and since the bottom horizontal map is an isomorphism, then the top horizontal map is also an isomorphism. Hence, the map $Gf:G \DD \ra G(\CC \circ F_0 \DD)$ is an isomorphism. By Theorem \ref{maclane-homology}, $f$ is an isomorphism.
 \end{proof}

\subsection{Cofibrations, fibrations and weak equivalences between fibrant objects} We show here that cofibrations, weak equivalences and fibrations between fibrant $\CCC$-coalgebras are easily characterized.

\begin{prop}\leavevmode\label{thm:cofibwefib}
Let $\AAA=(\Aa,\gamma_\AAA)$ and $\BBB=(\BB,\gamma_\BBB)$ be two $\Omega_u \CCC$-algebras and let $F: B_\iota \AAA \ra B_\iota \BBB$ be a morphism between their bar constructions. We denote by $f:B_\iota \AAA \ra \BB$ its projection $f= \pi_\BB F$ on $\BB$.
\begin{itemize}
 \itemt The morphism $F$ is a cofibration if and only if the restriction $f_{|\Aa}$ is a monomorphism.
 \itemt The morphism $F$ is a weak equivalence if and only $f_{|\Aa}$ is a quasi-isomorphism.
 \itemt The morphism $F$ is a fibration if and only if $f_{|\Aa}$ is an epimorphism.
\end{itemize}
\end{prop}

\begin{lemma}\label{lemma:acycliccofibchain}
The morphism of chain complexes $\Aa \to \Omega_\iota B_\iota \AAA$ which is the restriction to $\Aa$ of the canonical morphism $B_\iota \AAA \to B_\iota \Omega_\iota B_\iota \AAA $ is a quasi-isomorphism.
\end{lemma}

\begin{proof}[Proof]
It is a right inverse of the canonical morphism of $\Omega_u \CCC$-algebras $\Omega_\iota B_\iota \AAA \to \AAA$ which is a quasi-isomorphism.
\end{proof}

\begin{proof}[Proof of Proposition \ref{thm:cofibwefib}]
Note first that $f_{|\Aa}=F_{|\Aa}$. 
\begin{itemize}
 \itemt Suppose that $F$ is a cofibration, i.e. a monomorphism. Then, its restriction $F_{|\Aa}$ is also a monomorphism. Conversely, suppose that the map $f_{|\Aa}$ is a monomorphism. We can prove by induction that, for any integer $n$, the map $F: F^{rad}_n B_\iota \AAA \ra F^{rad}_n B_\iota \BBB$ is a monomorphism.
 \itemt By Lemma \ref{lemma:acycliccofibchain}, the maps $\Aa \ra \Omega_\iota B_\iota\AAA$ and $\BB \ra \Omega_\iota  B_\iota\BBB$ are quasi-isomorphisms. Consider the following diagram.
 $$
 \xymatrix{\Omega_\iota  B_\iota\AAA \ar[r]^{\Omega_\iota F}&   \Omega_\iota  B_\iota\BBB\\
 \Aa \ar[r]_{f_{|\Aa}} \ar[u] &  \BB \ar[u] }
 $$
It ensures us that $f_{|\Aa}$ is a quasi-isomorphism if and only if $\Omega_\iota F$ is a quasi-isomorphism, that is, if and only if $F$ is a weak equivalence.
\itemt Suppose that $F$ is a fibration. Since any chain complex can be considered as a $\CCC$-coalgebra (with $\Delta x = 1_\CCC \otimes x$), then any square of $\CCC$-coalgebras as follows has a lifting.
$$
\xymatrix{0 \ar[r] \ar[d] &  B_\iota\AAA \ar[d]\\ D^n \ar[r] &  B_\iota\BBB\ .}
$$
This ensures us that the map $f_{|\Aa}$ is an epimorphism. Conversely, suppose that $f_{|\Aa}$ is an epimorphism. By Lemma \ref{lemma:lefevrehasegawa}, there exists an isomorphism $G:B_\iota \AAA' \ra B_\iota \AAA$ such that $FG$ is in the image of the functor $B_\iota$. If we denote by $g$ the map from $\Aa'$ to $\Aa$ which underlies $G$, then $g$ is an isomorphism by Lemma \ref{lemma:iso}. Then $fg$ is a fibration of $\Omega_u \CCC$-algebras and so $FG= B_\iota (fg)$ is a fibration. Since $G$ is an isomorphism, then $F$ is a fibration.
\end{itemize}
\end{proof}

\begin{lemma}\label{lemma:lefevrehasegawa}
 Let $F:B_\iota\AAA \to B_\iota \BBB$ be a morphism of $\CCC$-coalgebras such that the underlying morphism $f: \Aa \ra \BB$ is surjective. Then, there exists an $\Omega_u \CCC$-algebra $\AAA'$ and an isomorphism of $\CCC$-coalgebras $G:B_\iota\AAA' \to B_\iota \AAA$ such that $FG$ is a strict morphism, that is in the image of the functor $B_\iota$.
\end{lemma}

\begin{proof}[Proof]
We build an isomorphism of graded $\CCC^{grad}$-coalgebras $G: \CCC \circ \Aa \ra \CCC \circ \Aa$ such that $FG$ is a strict morphism, that is of the from $Id_\CCC \circ h$. To that purpose we define inductively maps $g_n: F^{rad}_n \CCC \circ \Aa \ra \Aa$ such that $g_{n-1}$ is the restriction of $g_n$ to $F^{rad}_{n-1} \CCC \circ \Aa$ and such that we have the following equality between maps from $F^{rad}_n \CCC \circ \Aa$ to $\Aa$:
\begin{equation}\label{eqrec}
fg_n +   f (Id \circ g_{n-1}) (\ov\Delta \circ Id)=f\pi_\Aa\ ,
\end{equation}
where $\pi_\Aa = \epsilon \circ Id$ is the projection of $\CCC \circ \Aa$ on $\Aa$. First, let us choose $g_0= Id_\Aa$. Then, suppose that we have built $g_n$ satisfying Equation \ref{eqrec}. The map $f: \Aa \ra \BB$ and the injection of $F_n^{rad}\CCC \circ \Aa$ into $F_{n+1}^{rad}\CCC \circ \Aa$ give us the following square
 $$
 \xymatrix{\hom_\gMod (F_{n+1}^{rad} \CCC \circ \Aa,\Aa) \ar[r] \ar[d] & \hom_\gMod (F_{n+1}^{rad} \CCC \circ \Aa, \BB) \ar[d]\\
 \hom_\gMod (F_{n}^{rad} \CCC \circ \Aa, \Aa) \ar[r] & \hom_\gMod (F_{n}^{rad} \CCC \circ \Aa, \BB)\ .}
 $$
 The following map is surjective:
 $$
 \hom_\gMod (F_{n+1}^{rad} \CCC \circ \Aa, \Aa) \ra  \hom_\gMod (F_{n}^{rad} \CCC \circ \Aa, \Aa) \times_{ \hom_\gMod (F_{n}^{rad} \CCC \circ \Aa, \BB)}  \hom_\gMod (F_{n+1}^{rad} \CCC \circ \Aa, \BB)\ .
 $$
 So there exists an element of $\hom_\gMod (F_{n+1}^{rad} \CCC \circ \Aa,\Aa)$ whose image under this map is the pair $(g_n,f\pi_\Aa - f_{n+1}(Id \circ g_n) ( \ov \Delta \circ Id) )$. We can choose this element to be $g_{n+1}$. Thus, let $g$ be the map from $\CCC \circ \Aa$ to $\Aa$ whose restriction to $F^{rad}_n \CCC \circ \Aa$ is $g_n$ for any $n$. Let $G$ be the map of graded $\CCC^{grad}$-coalgebras which extends $g$. By Lemma \ref{lemma:iso}, the map $G$ is an isomorphism. Let us transfer the coderivation of $B_\iota \AAA$ to $\CCC \circ \Aa$ along the isomorphism $G$. This gives us a new $\Omega_u \CCC$-algebra structure on the chain complex $\Aa$ that we denote $\AAA'$. Finally, the morphism $FG$ is the image under the functor $B_\iota$ of the morphism of $\Omega_u \CCC$-algebras $fg_0: \AAA' \to \BBB'$. 
 \end{proof}

\begin{lemma}\label{lemma:iso}
 Let $F: \DDD = \CCC \circ \VV \ra  \EEE =\CCC \circ \WW$ be a morphism of quasi-cofree $\CCC$-coalgebras. Then, $F$ is an isomorphism if and only if its underlying map $f: \VV \ra \WW$ is an isomorphism.
\end{lemma}

\begin{proof}[Proof]
 Suppose first that $F$ is an isomorphism with inverse $G$. Let us denote by $g: \WW \to \VV$ the map underlying $G$. Then the map $g$ is inverse to $f$ and so $f$ is an isomorphism. Conversely, suppose that $f$ is an isomorphism. A straightforward induction shows that $F$ is both injective and surjective. 
\end{proof}

\subsection{Mapping spaces and deformation theory}

\begin{prop}\label{prop:enrichcoalggen}
For any cofibrant $\CCC$-coalgebra $\DDD$ and any fibrant $\CCC$-coalgebra $\EEE$, the simplicial set  $\HOM(\DDD,\EE)$ is a Kan complex and is a model for the mapping space $\Map(\CC,\DD)$ expected by the $\iota$-model structure.
\end{prop}

\begin{proof}[Proof]
Any fibrant $\CCC$-coalgebra $\EEE$ is isomorphic to the image under $B_\iota$ of a $\Omega_u \CCC$-algebra $\AAA$. So, we have:
 $$
 \HOM(\DDD, \EEE) \simeq \HOM (\DDD, B_\iota \AAA) \simeq \HOM (\Omega_\iota \DDD, \AAA) \simeq \Map (\Omega_\iota \DDD, \AAA) \simeq \Map (\DDD, B_\iota \AAA) \ . 
 $$
 Besides, we know from Proposition \ref{prop:enrichhomotalg} that $\HOM (\Omega_\iota \DDD, \AAA)$ is a Kan complex.
\end{proof}

\begin{cor}
Let $\alpha:\CCC \ra \PPP$ be an operadic twisting morphism. Let us endow the category of $\CCC$-coalgebras with the $\alpha$-model structure. For any cofibrant $\CCC$-coalgebra $\DDD$ and any fibrant $\CCC$-coalgebra $\EEE$, the simplicial set  $\HOM(\DDD,\EEE)$ is a Kan complex and is a model for the mapping space $\Map(\DDD,\EEE)$.
\end{cor}

\begin{proof}[Proof]
 It suffices to notice that fibrations and acyclic fibrations in the $\alpha$-model structure are in particular fibrations and acyclic fibrations in the $\iota$-model structure. Then, we can conclude by Proposition \ref{prop:enrichcoalggen}.
\end{proof}

Let $f: \DDD \to B_\iota\AAA$ be a morphism of $\CCC$-coalgebras. We know from Proposition \ref{prop:atom} that it is a dg atom of the cocommutative coalgebra $\{\DDD,B_\iota\AAA\}$. Consider the Hinich coalgebra $\{\DDD,B_\iota\AAA\}_f$ that appears from the decomposition described in Theorem \ref{thm:decomp}.

\begin{prop}
 The deformation problem induced by $\{\DDD,B_\iota\AAA\}_f$ is equivalent to the deformation problem
 $$
 R \in \Artinalg \mapsto \left(\hom_{R \otimes \Omega_n \otimes \Ccog} (R \otimes \Omega_n \otimes \CCC, R \otimes \Omega_n \otimes B_\iota\AAA)\right)_{n \in \mbn}\ .
 $$
\end{prop}

\begin{proof}[Proof]
 This is a direct consequence of Proposition \ref{prop:enrichedadjun} and Theorem \ref{prop:deformation}.
\end{proof}

\subsection{Algebras of the operad $\Omega_u \CCC$}

We have shown above that the adjunction $\Omega_\iota \dashv B_\iota$ is a Quillen equivalence. Moreover, in Proposition \ref{prop:ccoghpalg}, we have shown that fibrant $\CCC$-coalgebras are $\Omega_u \CCC$-algebras. So switching from the model category of $\Omega_u \CCC$-algebras to the model category of $\CCC$-coalgebras amounts to add new morphisms between any two $\Omega_u \CCC$-algebras. The weak equivalences and the fibrations of $\Omega_u \CCC$-algebras remain respectively weak equivalences and fibrations under this embedding but, in the category of $\CCC$-coalgebras, any monomorphism is a cofibration. In particular, any object is cofibrant. Subsequently, $\CCC$-coalgebras provide a convenient framework to study the homotopy theory of $\Omega_u \CCC$-algebras. For instance, the following proposition provides a tool to decide whether or not two $\Omega_u \CCC$-algebras are equivalent.

\begin{prop}
 Let $\AAA$ and $\BBB$ be two $\Omega_u \CCC$-algebras. There exists a chain of weak equivalences of $\Omega_u \CCC$-algebras between $\AAA$ and $\BBB$
 $$
 \xymatrix{\AAA = \AAA_0 \ar[r]^\sim & \AAA_1 & \cdots \ar[l]_\sim \ar[r]^\sim & \AAA_{n-1} & \AAA_n =\BBB \ar[l]_\sim }
 $$
 if and only if there exists a weak equivalence of $\CCC$-coalgebras between $B_\iota \AAA$ and $B_\iota \BBB$.
\end{prop}

\begin{proof}[Proof]
 Suppose that there exists a chain of weak equivalences from $\AAA$ to $\BBB$. Then, there exists a chain of weak equivalences between $B_\iota \AAA$ and $B_\iota \BBB$. Moreover, the objects of this chain are fibrant and cofibrant. So any morphism of this chain has an homotopical inverse. So there exists a weak equivalence from $B_\iota \AAA$ to $B_\iota \BBB$. Conversely, consider a weak equivalence $F$ from $B_\iota \AAA$ to $B_\iota \BBB$. Then, the following chain of weak equivalences of $\Omega_u \CCC$-algebras links $\AAA$ to $\BBB$.
 $$
 \xymatrix{\AAA & \Omega_\iota B_\iota \AAA \ar[l]_\sim \ar[r]^{\Omega_\iota (F)} & \Omega_\iota B_\iota \BBB \ar[r]^\sim & \BBB\ .}
 $$
\end{proof}

\subsection{Koszul morphisms}

In this subsection, we study the operadic twisting morphisms $\alpha: \CCC \ra \PPP$ such that the $\alpha$-model structure on the category of $\CCC$-coalgebras coincides with the universal $\iota$-model structure that we described above. Let $\alpha: \CCC \ra \PPP$ be an operadic twisting morphism. We denote by $\phi: \Omega_u(\CCC) \ra \PPP$ the morphism of operads induced by $\alpha$.

\begin{thm}\label{prop:quillengeneral}
The following assertions are equivalent.
\begin{enumerate}
\item The adjunction
$$
\xymatrix{\Omega_u(\CCC)-\mathsf{alg} \ar@<1ex>[r]^(0.55){\phi_!} & \Palg \ar@<1ex>[l]^(0.45){\phi^*}}
$$
is a Quillen equivalence.
\item The morphism of operads $\phi:\Omega_u(\CCC) \ra \PPP$ is a quasi-isomorphism.
\item The $\alpha$-model structure coincides with the $\iota$-model structure and $\Omega_\alpha \dashv B_\alpha$ is a Quillen equivalence.
\item For any $\PPP$-algebra $\AAA$, the map $\PPP \circ_\alpha \CCC \circ_\alpha \AAA \ra \AAA$ is a quasi-isomorphism, and for any $\CCC$-coalgebra $\DDD$, the morphism $\DDD \ra \CCC \circ_\alpha \PPP \circ_\alpha \DDD$ is a $\iota$-equivalence (it is the case if, for instance, it is a filtered quasi-isomorphism).

\item The morphisms of $\mbs$-modules $\Omega_u(\CCC) \circ_\iota \CCC \circ_\iota\Omega_u(\CCC) \ra \PPP \circ_\alpha \CCC \circ_\alpha\PPP$ and $\PPP \circ_\alpha \CCC \circ_\alpha\PPP \ra \PPP$ are quasi-isomorphisms.
\end{enumerate} 
\end{thm}

\begin{lemma}\label{lemma:schur}
 Let $f:\VV \ra \VV'$ be a morphism of dg $\mbs$-modules. Suppose that, for any chain complex $\WW$ (that is a $\mbs$-module concentrated in arity zero), the morphism $\VV \circ \WW \ra \VV' \circ \WW$ is a quasi-isomorphism. Then, $f$ is a quasi-isomorphism.
\end{lemma}

\begin{proof}[Proof]
By the operadic Kunneth formula, for any graded $\mbk$-module $\WW$, the map $H(\VV) \circ \WW \ra H(\VV') \circ \WW$ is an isomorphism. So, for any integer $n$, the map $f_n:H(\VV)(n) \otimes_{\mbs_n} \mbk^n \ra H(\VV')(n) \otimes_{\mbs_n} \mbk^n$ is an isomorphism. Let $(e_i)_{i=1}^n$ be a basis of $\mbk^n$. The map
$$
p \in H(\VV)(n) \mapsto p \otimes (e_1 \otimes \cdots \otimes e_n)  \mapsto f_n(p) \otimes (e_1 \otimes \cdots \otimes e_n) \mapsto f_n(p) \in H(\VV')(n)
$$
is an isomorphism. So, the morphism $H(\VV) \ra H(\VV')$ is an isomorphism.
\end{proof}

\begin{proof}[Proof of Theorem \ref{prop:quillengeneral}]\leavevmode
\begin{itemize}
 \itemt Let us first prove the equivalence between $(1)$ and $(2)$. Suppose $(2)$. Let $\AAA$ be a cofibrant $\Omega_u \CCC$-algebra and let $\BBB$ be a fibrant $\PPP$-algebra. Consider a map $f:\phi_! (\AAA) \ra \BBB$ and its adjoint map $g: \AAA \ra \phi^* (\BBB)$. The following diagram of $\Omega_u\CCC$-algebras is commutative.
 $$
 \xymatrix{\Omega_\iota B_\iota \AAA \ar[d]\ar[r] & \phi^*\phi_! \Omega_\iota B_\iota \AAA \ar[d] \\
 \AAA \ar[r] \ar@/_2pc/[rr]_{g}& \phi^*\phi_! \AAA \ar[r]^{\phi^* (f)} & \phi^* \BBB}
 $$  
 The left vertical map is a quasi-isomorphism. Since a left Quillen functor preserves weak equivalences between cofibrant objects and since $\phi^*$ preserves quasi-isomorphisms, then the right vertical map is a quasi-isomorphism. Besides $\phi_! \Omega_\iota B_\iota \AAA$ is actually $\Omega_\alpha B_\iota \AAA$. Since the morphism $\phi$ is a quasi-isomorphism, we can prove that the map $\Omega_\iota B_\iota \AAA \ra \phi^* \Omega_\alpha B_\iota \AAA$ is a filtered quasi-isomorphism for a well chosen filtration, and so is a quasi-isomorphism. So, by the 2-out-of-3 rule, the map $\AAA \ra \phi^* \phi_! \AAA$ is a quasi-isomorphism. Hence, $f$ is a quasi-isomorphism if and only if $\phi^* (f)$ is a quasi-isomorphism, if and only if $g$ is a quasi-isomorphism. So the assertion $(1)$ is true. Conversely, suppose $(1)$. Then, for any chain complex (considered as a $\CCC$-coalgebra) $\VV$, the map $\Omega_{\iota} \VV \ra \Omega_{\alpha} \VV$ is a quasi-isomorphism. So, by Lemma \ref{lemma:schur}, $(2)$ is true.
 \itemt Suppose $(1)$ and let us show $(3)$. By Proposition \ref{prop:bousfieldloc}, the $\alpha$-model structure coincides with the $\iota$-model structure. Moreover, since the adjunctions $\phi_! \dashv \phi^*$ and $\Omega_\iota \dashv B_\iota$ are both Quillen equivalences, then the adjunction $\phi_! \Omega_\iota \dashv B_\iota \phi^*$ which is $\Omega_\alpha \dashv B_\alpha$ is a Quillen equivalence.

\itemt Suppose $(3)$ and let us show $(4)$. Since $\Omega_\alpha \dashv B_\alpha$ is a Quillen equivalence, then $\Omega_\alpha B_\alpha \AAA \ra \AAA$ is a quasi-isomorphism for any $\PPP$-algebra $\AAA$ and $\DDD \ra B_\alpha \Omega_\alpha \DDD$ is an $\alpha$-weak equivalence for any $\CCC$-coalgebra $\DDD$. Since the $\alpha$-model structure coincides with the $\iota$-model structure, then $\DDD \ra B_\alpha \Omega_\alpha \DDD$ is a $\iota$-weak equivalence. So $(4)$ is true.
 \itemt Suppose $(4)$ and let us show $(5)$. For any $\PPP$-algebra $\AAA$, the morphism $\Omega_\alpha B_\alpha \AAA \ra \AAA$ is a quasi-isomorphism. Applying this to free $\PPP$-algebras and using Lemma \ref{lemma:schur}, we conclude that the map $\PPP \circ_\alpha \CCC \circ_\alpha\PPP \ra \PPP$ is a quasi-isomorphism. Moreover, for any $\PPP$-algebra $\AAA$, the following diagram commutes
 $$
 \xymatrix{\Omega_u(\CCC) \circ_\iota \CCC \circ_\iota \AAA \ar[r] \ar[rd] & \PPP \circ_\alpha \CCC \circ_\alpha \AAA \ar[d]\\&\AAA\ .}
 $$
 Since the composite map and the vertical map are quasi-isomorphisms (because $\Omega_\iota \dashv B_\iota$ and $\Omega_\alpha \dashv B_\alpha$ are Quillen equivalences), then, by the 2-out-of-3 rule, the horizontal map is a quasi-isomorphism. Applying this to free $\PPP$-algebras and using Lemma \ref{lemma:schur}, we conclude that the map $\Omega_u(\CCC) \circ_\iota \CCC \circ_\alpha\PPP \ra \PPP \circ_\alpha \CCC \circ_\alpha \PPP$ is a quasi-isomorphism. Besides, for any $\CCC$-coalgebra $\DDD$ the following diagram commutes.
 $$
 \xymatrix{\DDD \ar[r] \ar[rd] &\CCC \circ_\iota \Omega_u(\CCC) \circ_\iota \DDD  \ar[d]\\  &\CCC \circ_\alpha \PPP \circ_\alpha \DDD}
 $$ 
By the 2-out-of-3 rule, the vertical map is a $\iota$-weak equivalence. So the map $\Omega_u(\CCC) \circ_\iota \CCC \circ_\iota \Omega_u(\CCC) \circ_\iota \DDD  \ra \Omega_u(\CCC) \circ_\iota \CCC \circ_\alpha \PPP \circ_\alpha \DDD $ is a quasi-isomorphism. Applying this for $\CCC$-coalgebras which are just chain complexes and using Lemma \ref{lemma:schur}, we obtain that the map $\Omega_u(\CCC) \circ_\iota \CCC \circ_\iota \Omega_u(\CCC)  \ra \Omega_u(\CCC) \circ_\iota \CCC \circ_\alpha \PPP $ is a quasi-isomorphism.
 \itemt Suppose $(5)$ and let us show $(2)$. The following square of $\mbs$-modules is commutative.
 $$
 \xymatrix{\Omega_u(\CCC) \circ_\iota \CCC \circ_\iota \Omega_u(\CCC) \ar[r] \ar[d] & \Omega_u(\CCC) \ar[d]\\
  \PPP \circ_\alpha \CCC \circ_\alpha \PPP \ar[r] & \PPP}
 $$
Since the left vertical map and the horizontal maps are quasi-isomorphisms, then the right vertical map is also a quasi-isomorphism.
\end{itemize}
\end{proof}

\begin{defin}[Koszul morphisms]\label{defin:koszul}
 An operadic twisting morphism $\alpha: \CCC \ra \PPP$ satisfying the properties of Theorem \ref{prop:quillengeneral} is called a \textit{Koszul morphism}.
\end{defin}

In the next section, we will explore Koszul duality which is a method to produce Koszul morphisms from a presentation of an operad.

\section{Examples}

The purpose of this section is to apply the general framework described in the previous sections to the case of common nonaugmented operads like the operads $\uAs$ and $\uCom$ whose algebras are respectively the unital associative algebras and the unital commutative algebras. So, for any of these operads $\PPP$, one looks after a curved conilpotent cooperad $\CCC$ together with an operadic twisting morphism $\alpha$ from $\CCC$ to $\PPP$ such that the induced morphism of operads from $\Omega_u\CCC$ to $\PPP$ is a quasi-isomorphism; that is, $\alpha$ is a Koszul morphism. One can use the universal twisting morphism $B_c \PPP \ra \PPP$. However, the bar construction is always very big. Instead, one usually tries to produce a sub-cooperad of $B_c\PPP$ whose cobar construction will be a resolution of $\PPP$. The Koszul duality theory is a way to produce such a sub-cooperad when the operad $\PPP$ has a quadratic presentation or a quadratic-linear presentation. This construction has been extended to quadratic-linear-constant presentations by Hirsh and Mill\`es in \cite{HirshMilles12}, generalizing to operads the curved Koszul duality of algebras developed by Polishchuk and Positselski in \cite{PolischukPositselski05}.

\subsection{Koszul duality}
Koszul duality is a way to build a cooperad $\PPP^\ac$ together with a canonical operadic twisting morphism from $\PPP^\ac$ to $\PPP$,  out of an operad $\PPP$ which has a "nice enough" presentation $\PPP= \Tfree(\VV)/(\RR)$. Here, we present the construction of Hirsh and Mill\`es in \cite{HirshMilles12}.\\

Let $\PPP$ be a graded operad equipped with a presentation $\PPP= \Tfree (\VV)/ (\RR)$, where $\VV$ is a graded $\mbs$-module and where $(\RR)$ is the operadic ideal generated by a sub-graded-$\mbs$-module $\RR$ of $  \Tfree^{\leq 2} (\VV) $ such that 
$$
\begin{cases}
\RR \cap (\II \oplus \VV) = \{0\}\ ,\\
(\RR) \cap  \Tfree^{\leq 2} (\VV) = \RR\ .
\end{cases}
$$
We denote by $q\RR$ the projection of $\RR \subset \Tfree^{\leq 2} (\VV)$ onto $\Tfree^{2} (\VV)$ along $\II \oplus \VV$. Moreover, let  $q\PPP$ be the following operad:
$$
q\PPP:= \Tfree (\VV)/(q\RR)\ .
$$
This is a quadratic operad. The condition $\RR \cap (\II \oplus \VV) = \{0\}$ induces a function $\phi=(\phi_0,\phi_1) : q\RR \to \II \oplus \VV$.

\begin{defin}[Curved cooperad Koszul dual of an operad]\label{defin:curvedkoszulduality}\cite[\S 4.1]{HirshMilles12}
The Koszul dual cooperad $\PPP^\ac$ of $\PPP$, associated to the presentation $\PPP=\Tfree(\VV)/(\RR)$, is the following curved conilpotent cooperad. The underlying graded cooperad is the final graded sub-cooperad of $\Tfree^c (s\VV )$ such that the composition
$$
\PPP^\ac \ra \Tfree^c(s\VV) \ra \Tfree^2(s\VV) / s^2q\RR
$$
is zero. It is equipped with the unique coderivation which extends the following map
\begin{align*}
 \PPP^\ac \twoheadrightarrow s^2q\RR &\ra s\VV \\
 sx \otimes sy & \mapsto (-1)^{|x|}s\phi_1 (x \otimes y)\ .
\end{align*}
Its curvature is the following degree $-2$ map:
\begin{align*}
\theta : \PPP^\ac \twoheadrightarrow s^2q\RR &\ra \mbk \\
 sx \otimes sy & \mapsto (-1)^{|x|}s\phi_0 (x \otimes y)\ .
\end{align*}
Moreover, the map
$$
\kappa: \PPP^\ac \twoheadrightarrow s\VV \to \VV \hookrightarrow \PPP\ ,
$$
is an operadic twisting morphism which induces both a morphism of operads $\Omega_u \PPP^\ac\to \PPP$ and a morphism of curved conilpotent cooperads $\PPP^\ac \to B_c \PPP$.
\end{defin}

\begin{rmk}
 The coherence of the above definition is proven in \cite[\S 4.1]{HirshMilles12}.
\end{rmk}


\begin{defin}[Koszul operad]
 The operad $\PPP$ (together with the presentation $\PPP= \Tfree (\VV)/(\RR)$) is said to be Koszul if the twisting morphism $\kappa: \PPP^\ac \ra \PPP$ is Koszul, that is if the map $\Omega_c \PPP^\ac \ra \PPP$ is a quasi-isomorphism.
\end{defin}

The following theorem is a powerful tool to show that an operad is Koszul.

\begin{thm}\cite[Theorem 4.3.1]{HirshMilles12}\label{thm:hm}
Suppose that the canonical morphism
 $$
 q\PPP \circ_\kappa q\PPP^\ac \ra \II
 $$
 is a quasi-isomorphism. Then $\PPP$ is Koszul.
\end{thm}

\subsection{Coalgebras over a Koszul dual}

In this subsection, we describe the category of $\PPP^\ac$-coalgebras, where $\PPP^\ac$ is the Koszul dual of the "quadratic-linear-homogeneous operad" $\PPP$ defined above. We will need the following definition.

\begin{defin}[Pre-coradical filtration]
 Let $\WW$ be a graded $\mbs$-module and let $\CC$ be a graded $\mbk$-module equipped with a map $\Delta^{(1)}: \CC \ra \WW \circ \CC$. We define $(F_n^{prad}\CC)_{n \in \mbn}$ to be the following (non-necessarily exhaustive) filtration on $\CC$ called the pre-coradical filtration.
 $$
\begin{cases}
 F_0^{prad}(\CC) :=ker(\Delta^{(1)})\ ,\\
 F_n^{prad}(\CC) :=(\Delta^{(1)})^{-1} \big( \WW(0) \oplus \sum_{i_1 + \cdots + i_k = n-1\atop k \geq 1} \WW(k) \otimes_{\mbs_k} (F_{i_1}^{prad}\CC \otimes \cdots \otimes F_{i_k}^{prad}\CC )\big)\text{ , if }n \geq 1\ .
\end{cases}
 $$
\end{defin}

\begin{lemma}\label{lemma:precoalg}
Consider a cofree graded conilpotent cooperad $ \Tfree^c (\WW)$. The category of graded coalgebras over $\Tfree^c (\WW)$ is equivalent to the category of graded-$\mbk$-modules $\CC$ equipped with a map $\Delta^{(1)}:\CC \ra \WW \circ \CC $ such that the pre-coradical filtration $(F_n^{prad}\CC)_{n \in \mbn}$ is exhaustive. Moreover, under this equivalence, the coradical filtration coincides with the pre-coradical filtration.
\end{lemma}

\begin{proof}[Proof]
 Let $\CC$ be a graded-$\mbk$-module with a map $\Delta^{(1)}:\CC \ra \WW \circ \CC $ such that the pre-coradical filtration $(F_n^{prad}\CC)_{n \in \mbn}$ is exhaustive. Then, let us define $\Delta_\CC: \CC \ra \Tfree (\WW) \circ \CC$ by induction as follows.
 $$
\begin{cases}
 \Delta_\CC (x) := 1 \otimes  x\text{ if }x \in F^{prad}_0\CC\ ,\\
 \Delta_\CC (x):= 1 \otimes x + (Id \circ \Delta_\CC)\Delta^{(1)}(x)\text{ if } x \in F^{prad}_n\CC\ .
\end{cases}
$$
 This defines a structure of $\Tfree^c (\WW)$-coalgebra on $\CC$. Conversely, let $(\CC,\Delta)$ be a graded $\Tfree^c(\WW)$-coalgebra. We obtain a map $\Delta^{(1)}: \CC \to \WW \circ \CC$ by composing $\Delta$ with the projection of $\Tfree^c(\WW)$ onto $\WW$. Then, the construction we just described recovers $\Delta$ from $\Delta^{(1)}$.
\end{proof}

\begin{thm}\label{thm:coalgkoszul}
 Suppose that the characteristic of the field $\mbk$ is zero (this assumption is not necessary in the nonsymmetric context). The category of $\PPP^\ac$-coalgebras is equivalent to the category of graded $\Tfree^c(s\VV)$-coalgebras (that is graded-$\mbk$-modules $\CC$ equipped with a map $\Delta^{(1)}\CC:\ \ra s\VV \circ \CC $  such that the pre-coradical filtration $(F_n^{prad}\CC)_{n \in \mbn}$ is exhaustive) such that the composite map
 $$
 \xymatrix{\CC \ar[rrr]^(0.35){\Delta_2=(Id \circ' \Delta^{(1)})\Delta^{(1)} } &&&  \Tfree^2 (s\VV)  \circ \CC \ar@{->>}[r] & \left(\Tfree^2 (s\VV) /s^2q\RR \right) \circ \CC}
 $$
 is zero, together with a degree $-1$ map $d_\CC: \CC\ra \CC$ such that
 $$
\begin{cases}
 d^2_\CC  = (\theta \circ Id)\Delta_2\ ,\\
 \Delta^{(1)} d_\CC =(d_{\PPP^\ac} \circ Id)\Delta_2 + (Id \circ' d_\CC) \Delta^{(1)}\ .
\end{cases}
 $$

\end{thm}

\begin{proof}
 Let $\CC$ be a graded $\Tfree^c(s\VV)$-coalgebra together with a degree $-1$ map $d_\CC: \CC \to\CC$ satisfying the conditions of Theorem \ref{thm:coalgkoszul}. For any $x \in \CC$, let $\CC (x)$ be a finite dimensional sub-$\Tfree^c (s\VV )$-coalgebra of $\CC$ which contains $x$. By Lemma \ref{lemma:keypanitshriakcoalg}, the map $\Delta_{\CC(x)}: \CC(x) \ra \Tfree^c (s\VV) \circ \CC(x)$ factorizes through a unique map $\CC(x) \ra \PPP^\ac \circ \CC(x)$. Hence, $\CC$ has a structure of graded $(\PPP^\ac)^{grad}$-coalgebra. Moreover, we can prove by induction on the coradical filtration of $\CC$ that $d_\CC$ is a coderivation.
 \end{proof}

\begin{lemma}\label{lemma:keypanitshriakcoalg}
 Let $\CC(x)$ be the graded $ \Tfree^c (s\VV)$-coalgebra defined in the proof of Theorem \ref{thm:coalgkoszul}. Then, $\CC(x)$ is a graded $ \PPP^\ac$-coalgebra.
\end{lemma}

\begin{proof}[Proof]
 Remember that $\CC(x)$ is a finite dimensional sub-graded-$\Tfree^c (s\VV)$-coalgebra of $\CC$. Let $(e_i)_{i=1}^m$ be a basis of $\CC(x)$. Then, for any $i\in \{1,\ldots,m\}$, let $p_{i,0}\in \ov \Tfree (s\VV) (0)$, and for any integer $k \geq 1$ and for any nondecreasing function $s$ from $\{1,\ldots,k\}$ to $\{1,\ldots,m\}$, let $p_{i,k,s} \in \ov \Tfree (s\VV) (k)$ such that
 $$
 \Delta (e_i)= 1 \otimes e_i + p_{i,0} + \sum_{k=0}^\infty \sum_s p_{i,k,s} \otimes_{\mbs_k} (e_{s(1)} \otimes \cdots \otimes e_{s(k)})
 $$
 For any nondecreasing function $s$ from $\{1,\ldots,k\}$ to $\{1,\ldots,m\}$ and for any $\sigma \in \mbs_k$, let $\epsilon (s, \sigma )$ be the element of $\mathbb{Z}/2\mathbb{Z}$ such the structural action of $\sigma$ on $\CC^{\otimes k}$ sends $e_{s(1)} \otimes \cdots \otimes e_{s(k)}$ to $(-1)^{\epsilon (s,\sigma)} e_{s\sigma^{-1}(1)} \otimes \cdots \otimes e_{s\sigma^{-1}(k)}$. Besides, let $Inv(s)$ be the subgroup of $\mbs_k$ of permutation $\sigma$ such that $s= s\sigma^{-1}$. Then, we can choose $p_{i,k,s}$ such that $p_{i,k,s}^\sigma = (-1)^{\epsilon (s, \sigma)} p_{i,k,s}$ for any $\sigma \in Inv(s)$. Indeed, if it is not the case, we can replace $p_{i,k,s}$ by
 $$
\frac{1}{\# Inv(s)} \sum_{\sigma \in Inv (s)} (-1)^{\epsilon (s, \sigma)} p_{i,k,s}^\sigma\ .
 $$
Let $\DDD$ be the sub-graded $\mbs$-module of $\Tfree (s\VV)$ generated by $1$ and the elements $p_{i,k,s}$. Since $(\Delta \circ Id)\Delta(e_i) = (Id \circ \Delta)\Delta(e_i)$ for any $i$, then there exists an element of $q_{i,k,s}\in (\DDD \circ \DDD)(k)$ such that
 $$
 \Delta(p_{i,k,s}) \otimes_{\mbs_k} (e_{s(1)} \otimes \cdots \otimes e_{s(k)}) = q_{i,k,s} \otimes_{\mbs_k} (e_{s(1)} \otimes \cdots \otimes e_{s(k)})\ .
 $$
 Since $p_{i,k,s}^\sigma = (-1)^{\epsilon (s, \sigma)} p_{i,k,s}$ for any $\sigma \in Inv(s)$, then
 $$
 \Delta(p_{i,k,s}) = \frac{1}{\# Inv(s)} \sum_{\sigma \in Inv (s)} (-1)^{\epsilon (s, \sigma)} q_{i,k,s}^\sigma\ .
 $$
 So, $ \Delta(p_{i,k,s}) \in \DDD \circ \DDD$. Hence $\DDD$ is a sub-graded cooperad of $\Tfree^c (s\VV)$. Moreover, for any $i$, $(\pi \circ Id)\Delta(e_i)=0$ where $\pi$ is the projection of $\Tfree (s\VV)$ onto $ \Tfree^2 (s\VV ) / s^2q\RR$. So, $\pi(p_{i,k,s})=0$ for any 3-tuple $(i,k,s)$ and $\pi(p_{i,0})=0$ for any $i$ ;  so $\pi_{|\DDD}=0$. Hence $\DDD \subset \PPP^\ac$.
\end{proof}

\subsection{Unital associative algebras up to homotopy}\label{barcobaralg}

\subsubsection{A presentation of the operad $\uAs$}
Let $\uAs$ be the nonsymmetric operad defined by the presentation $\uAs:=\Tfree (\mbk \cdot \mu \oplus \mbk \cdot \xi)/(\RR)$ where $\mu$ is an arity two element and $\xi$ is an arity zero element. The nonsymmetric module $\RR \subset \II \oplus \Tfree^2 (\mbk \cdot \mu \oplus \mbk \cdot \xi)$ is made up of the following relations
$$
\begin{cases}
 \mu \otimes_{ns} (\xi \otimes 1) -1\ ,\\ 
 \mu \otimes_{ns} (1 \otimes \xi) - 1\ ,\\
  \mu \otimes_{ns} (\mu \otimes 1) - \mu \otimes_{ns} (1 \otimes \mu)\ .
\end{cases}
$$

\begin{rmk}
 Here, the symbol $ns$ stands for the composition product of nonsymmetric modules.
\end{rmk}

Given this presentation, the Koszul dual $\uAs^\ac$ is a nonsymmetric curved conilpotent cooperad whose underlying graded cooperad is the final subcooperad of $\Tfree^c (\mbk \cdot s\mu \oplus \mbk \cdot \xi)$ such that 
$$
\uAs^\ac \cap \Tfree^2 (\mbk \cdot s\mu )= \mbk \cdot \left( s\mu \otimes_{ns} (s\mu \otimes 1) -  s\mu \otimes_{ns} (1 \otimes s\mu) \right)\ .
$$
The coderivation of $\uAs^\ac$ is zero and the curvature is given by
$$
\theta \left( s\mu \otimes_{ns} (s\xi \otimes 1) \right)= \theta\left(  s\mu \otimes_{ns} (1 \otimes s\xi)\right)=-1\ .
$$

\begin{rmk}
 The Koszul dual curved cooperad $\uAs^\ac$ of the operad $\uAs$ is described in details in \cite{HirshMilles12}.
\end{rmk}

\subsubsection{Coalgebras over $\uAs^\ac$}

\begin{prop}
The endofunctor of the category of graded $\mbk$-modules $\VV \mapsto s \VV$ induces an equivalence between the category of $\uAs^\ac$-coalgebras and the category of non-counital curved conilpotent coassociative coalgebras.
\end{prop}

\begin{proof}[Proof]
The proof relies on the same arguments as the proof of Proposition \ref{propkoszulduallie} that will be detailed.
\end{proof}

\begin{rmk}
The map $\VV \to s\VV$ also induces an equivalence between graded $(\uAs^\ac)^{grad}$-coalgebras and graded non-counital conilpotent coassociative coalgebras $\CCC$ equipped with a degree $-2$ map $\CCC \to \mbk$. Moreover, this equivalence sends a cofree graded $(\uAs^\ac)^{grad}$-coalgebra $\uAs^\ac \circ \VV$ to the cofree conilpotent coalgebra $\ov \Tfree (\VV \oplus \mbk \cdot v)$ where $|v|=2$ with the degree $-2$ map:
\begin{align*}
 \ov \Tfree (\VV \oplus \mbk \cdot v) \twoheadrightarrow \mbk \cdot v &\to \mbk \\
 v & \mapsto 1\ . 
\end{align*}
\end{rmk}

\begin{nota}
We denote the category of curved conilpotent coassociative coalgebras by $\cCog$. Moreover, we denote the operad $\Omega_u \uAs^\ac$ by $u\mathscr{A}_\infty$.
\end{nota}

\subsubsection{The bar-cobar adjunction and $u\mathscr{A}_\infty$-algebras}

On the one hand, there exists an adjunction relating $\uAs$-algebras to $\uAs^\ac$-coalgebras which is induced by the operadic twisting morphism $\alpha: \uAs^\ac \ra \uAs$. On the other hand the category of $\uAs^\ac$-coalgebras is equivalent to the category $\cCog$ of curved conilpotent coalgebras. Thus, we obtain a bar-cobar adjunction between unital associative algebras and curved conilpotent coalgebras which is the restriction to arity one of the operadic bar-cobar adjunction described in Section \ref{subsection:operadicbarcobar} (with the exception that we can consider noncounital coalgebras instead of coaugmented counital coalgebras). For this reason, we denote this adjunction using the same symbols as in the operadic context, that is $\Omega_u \dashv B_c$. So we have:
$$
\Omega_u \CCC := \ov \Tfree (s^{-1} \CCC)
$$
$$
B_c\AAA :=\ov \Tfree (s\AAA \oplus \mbk \cdot v)
$$
for any curved conilpotent coalgebra $\CCC$ and for any unital algebra $\AAA$. The derivation of $\Omega_u (\CCC)$ and the coderivation of $B_c (\AAA)$ are defined as in Section \ref{subsection:operadicbarcobar}.\\

The adjunction $\Omega_u \dashv B_c$ is part of a larger picture:
$$
\xymatrix{\cCog \ar@<1ex>[r]^(0.42){\Omega_\iota} & u\mathscr{A}_\infty-\mathsf{alg} \ar@<1ex>[l]^(0.58){B_\iota} \ar@<1ex>[r]^(0.5){\phi_!} & \uAs-\mathsf{alg}\ , \ar@<1ex>[l]^(0.5){\phi^*}}
$$
where the adjunction $\phi_! \dashv \phi^*$ is induced by the morphism of operads $\phi: u\mathscr{A}_\infty \to \uAs$ and where $\Omega_u = \phi_!\Omega_\iota$ and $B_c= B_\iota \phi^*$. We know that a $u\mathscr{A}_\infty$-algebra $\AAA=(\Aa,\gamma)$ is the data of a chain complex $\Aa$ together with a coderivation on the cofree graded $(\uAs^\ac)^{grad}$-coalgebra $\uAs^\ac \circ \Aa$ so that it becomes a $\uAs$-coalgebra. Equivalently, it is the data of a chain complex together with a coderivation on the cofree conilpotent coassociative coalgebra $\ov\Tfree (s\Aa \oplus \mbk \cdot v)$ so that it becomes a curved conilpotent coalgebra whose curvature $\theta$ is given by
\begin{align*}
 \ov \Tfree (s\Aa \oplus \mbk \cdot v) \twoheadrightarrow \mbk \cdot v &\to \mbk \\
 v & \mapsto 1\ . 
\end{align*}
By Lemma \ref{lemma:curvcofree}, this is equivalent to a degree $-1$ map
$$
\gamma: \ov \Tfree (s\Aa \oplus \mbk \cdot v) \to \Aa\ ,
$$
such that for any $x_1,\cdots,x_n \in (s\Aa \oplus \mbk \cdot v)$
$$
\sum_{0 \leq i \leq j \leq n} (-1)^{|x_1|+ \cdots +|x_{i-1}|} \gamma (x_1 \otimes \cdots  \otimes \gamma (x_i \otimes \cdots \otimes x_j) \otimes \cdots x_n)=
\begin{cases}
 0 \text{ if }n\neq 2\ ,\\
 \theta(x_1)x_2 -  \theta(x_2)x_1 \text{ if }n= 2\ .
\end{cases}
$$
In particular, we have the following.
\begin{itemize}
 \itemt A degree zero product
 $$
 \gamma_2: \Aa \otimes \Aa \to \Aa\ .
 $$
 \itemt A degree $1$ map
 $$
 \gamma_3: \Aa\otimes \Aa \otimes \Aa \to \Aa\ ,
 $$ 
 whose boundary is the associator of $\gamma_2$, that is 
 $$
 \partial (\gamma_3)= \gamma_2 (Id \otimes \gamma_2) - \gamma_2 (\gamma_2 \otimes Id)\ .
 $$
 \itemt An element $1_\AAA$ defined by $\gamma (v) = s 1_\AAA$.
 \itemt Maps $\gamma_{1,l}: \Aa \to \Aa$ and $\gamma_{1,r}:  \Aa \to \Aa$ of degree $1$ which make $1_\AAA$ a unit up to homotopy, that is
 $$
 \partial (\gamma_{1,l}) = \gamma_2 (1_\AAA \otimes Id) -Id\ ,
 $$
 $$
 \partial (\gamma_{1,l}) = \gamma_2 (Id \otimes 1_\AAA) -Id\ .
 $$
\end{itemize}

\subsubsection{The Koszul property and the infinity category of $u\mathscr{A}_\infty$-algebras}

\begin{prop}\cite[6.1.8]{HirshMilles12}\label{prop:uaskoszul}
The operad $\uAs$ is Koszul. 
\end{prop}

\begin{rmk}
 The model structure on curved conilpotent coalgebras that we get by transfer along the adjunction $\Omega_u \dashv B_c$ is the model structure that Positselski described in \cite{Positselski11}. 
\end{rmk}

There are several ways to describe the infinity-category of $\uAs$-algebras.
\begin{itemize}
 \itemt One can take the Dwyer--Kan simplicial localization of the category of $\uAs$-algebras with respect to quasi-isomorphisms as described in \cite{DwyerKan80a} and \cite{DwyerKan80b}.
 \itemt One can take the simplicial category whose objects are cofibrant-fibrant $\uAs$-algebras and whose spaces of morphisms are 
 $$
\Map(\AAA,\BBB)_n := \HOM^{ns}_{\uAs-\mathsf{alg}}(\AAA,\BBB )\ .
 $$
 \itemt One can also take the simplicial category whose objects are all $\uAs$-algebras and whose spaces of morphisms are
 $$
 \Map(\AAA,\BBB):= \HOM^{ns}_{\uAs-\mathsf{alg}}(\Omega_u B_c \AAA,\BBB )\simeq \HOM^{ns}_{\cCog}(B_c \AAA,B_c \BBB )\ .
 $$
\end{itemize}

\begin{prop}\label{uasinfinitycat}
 The three simplicial categories described above are equivalent. Moreover, the two last ones are fibrant in the sense that for any objects $\AAA$ and $\BBB$, the simplicial set $\Map(\AAA,\BBB)$ is a Kan complex.
\end{prop}

\begin{proof}
It follows from \cite{DwyerKan80c} that the two first simplicial categories are equivalent and it follows from Proposition \ref{prop:uaskoszul} that the two last simplicial categories are equivalent. Moreover, the two last simplicial categories are fibrant by Proposition \ref{prop:enrichhomotalg}.
\end{proof}

\subsection{Unital commutative algebras up to homotopy}

In this section, we assume that the characteristic of the base field $\mbk$ is zero.

\subsubsection{A presentation of the operad $\uCom$}

Let $\uCom$ be the operad defined by the presentation $\uCom:=\Tfree (\mbk \cdot \mu \oplus \mbk \cdot \xi)/(\RR)$ where $\mu$ is an arity two element such that $\mu^{(1,2)}=\mu$ and $\xi$ is an arity zero element. The $\mbs$-module $\RR \subset \II \oplus \Tfree^2 (\mbk \cdot \mu \oplus \mbk \cdot \xi)$ is generated by the elements:
$$
\begin{cases}
 \mu \otimes_{\mbs_2} (\mu \otimes 1) - \mu \otimes_{\mbs_2} (1 \otimes \mu)\ ,\\
 \mu \otimes_{\mbs_2} (\xi \otimes 1) -1\ . 
\end{cases}
$$

\begin{rmk}\leavevmode
\begin{itemize}
 \itemt Since the action of $\mbs_2$ on $\mu$ is trivial, then we have
 $$
 \mu \otimes_{\mbs_2} (1 \otimes \mu) = \left( \mu \otimes_{\mbs_2} (\mu \otimes 1) \right)^{(132)}\ .
 $$
\itemt The element $\mu \otimes_{\mbs_2} (\mu \otimes 1) - \mu \otimes_{\mbs_2} (1 \otimes \mu)$ is a generator of the $\mbs_3$-module $\RR(3)$. However, it is not a generator of $\RR(3)$ as a $\mbk$-module ; one needs to add the element $ \mu \otimes_{\mbs_2} (\mu \otimes 1) -\big( \mu \otimes_{\mbs_2} (\mu \otimes 1)\big)^{(2,3)}$.
\end{itemize}
\end{rmk}

Given this presentation, the Koszul dual $\uCom^\ac$ is a curved conilpotent cooperad whose underlying graded cooperad is the final subcooperad of $\Tfree^c (\mbk \cdot s\mu \oplus \mbk \cdot \xi)$ such that 
$$
\uAs^\ac (3) \cap \Tfree^2 (\mbk \cdot s\mu ) (3)= \mbk [\mbs_3] \cdot \left( s\mu \otimes_{\mbs_2} (s\mu \otimes 1) -  s\mu \otimes_{\mbs_2} (1 \otimes s\mu) \right)\ .
$$
The coderivation of $\uCom^\ac$ is zero and the curvature is given by
$$
\theta \left( s\mu \otimes_{\mbs_2} (s\xi \otimes 1) \right)=-1\ .
$$

\begin{nota}
We denote by $\uCom_\infty$ the operad $\Omega_u \uCom^\ac$.
\end{nota}

\subsubsection{Coalgebras over $\uCom^\ac$}

We will show that the category of $\uCom^\ac$-coalgebras is equivalent to the category of curved conilpotent Lie coalgebras.

\begin{defin}[Curved Lie coalgebra]
 A \textit{curved Lie coalgebra} $\CCC=(\CC,\delta, d, \theta)$ is a graded $\mbk$-module $\CC$ equipped with an antisymmetric map $\delta : \CC \ra \CC \otimes \CC$ such that 
 $$
(\delta \otimes Id) \delta  = (Id \otimes \delta) \delta + (Id\otimes \tau )(\delta \otimes Id) \delta\ ,
 $$
 where $\tau$ is the exchange map $\tau(x \otimes y)= (-1)^{|x||y|}y \otimes x$. It is also equipped with a degree $-1$ map $d: \CC \ra \CC$ which is a coderivation, that is
 $$
 \delta d = (d \otimes Id + Id \otimes d) \delta\ ,
 $$
 and with a degree $-2$ map $\theta: \CC \ra \mbk$ which is a curvature, that is
 $$
 d^2 = (\theta \otimes Id- Id \otimes \theta) \delta\ .
 $$
 A curved Lie coalgebra $\CC$ is said to be \textit{conilpotent} if for any $x\in \CC$, there exists an integer $n$ such that the element

 $$
 (Id \otimes \cdots \otimes \delta \otimes \cdots \otimes Id ) \cdots \delta (x)
 $$
 is zero whenever $\delta$ appears $n$ times. We denote by $\cLieCog$ the category of curved conilpotent Lie coalgebras.
\end{defin}

\begin{prop}\label{propkoszulduallie}
  The endofunctor of the category of graded $\mbk$-modules $\VV \mapsto s \VV$ induces an equivalence between the category of $\uCom^\ac$-coalgebras and the category $\cLieCog$ of curved conilpotent Lie coalgebras.
\end{prop}

\begin{lemma}\label{lemma:Liecog}
The category of $\uCom^\ac$-coalgebras is equivalent to the category whose objects are graded $\mbk$-modules $\CC$ equipped with three maps.
\begin{itemize}
 \itemt A degree $-1$ map $\delta: \CC \to \CC \otimes \CC$ which is symmetric is the sense that $\tau \delta = \delta'$, which satisfy the following equation
 $$
 (\delta \otimes Id) \delta (x) + \left((\delta \otimes Id) \delta (x)\right)^{(2,3)} + \left((\delta \otimes Id) \delta (x)\right)^{(1,3)}=0\ ,
 $$
 and such that for any $x \in \CC$, there exists an integer $n$ such that the element
 $$
 (Id \otimes \cdots \otimes \delta \otimes \cdots \otimes Id) \cdots (Id \otimes \delta) \delta(x)
 $$
 is zero whenever $\delta$ appears at least $n$ times.
 \itemt A degree $-1$ map $\theta: \CC \to \mbk$.
 \itemt A degree $-1$ map $d: \CC \to \CC$ such that $\theta d=0$, such that $\delta d = -(d \otimes Id + Id \otimes d) \delta$ and such that $d^2=-(\theta \otimes Id + Id \otimes \theta) \delta= -2(\theta \otimes Id) \delta$. 
\end{itemize}
The morphisms of this category are the morphisms of graded $\mbk$-modules which commute with these structure maps.
\end{lemma}

\begin{proof}[Proof]
We apply Theorem \ref{thm:coalgkoszul}. A graded $\uCom^\ac$-coalgebra is a graded $\mbk$-module $\CC$ equipped with maps
$$
\begin{cases}
\delta': \CC \to (\mbk \cdot s\mu) \otimes_{\mbs_2} (\CC \otimes \CC)\ ,\\
\theta': \CC \to \mbk \cdot s\xi\ , \\ 
d': \CC \to \CC\ ,
\end{cases}
$$
such that the corresponding pre-coradical filtration is exhaustive, such that
\begin{equation}\label{equationkoszullie}
(Id_{s\mu} \circ' \delta') \delta'(\CC) \subset s^2\RR (3) \otimes_{\mbs_3} \CC^{\otimes 3}\ . 
\end{equation} 
and such that 
$$
\begin{cases}
\delta'd'= (Id \circ' d' ) \delta'\ ,\\
\theta'd'=0\ , \\ 
d'^2= (\theta_{\uCom^\ac} \circ Id)(Id \circ' \delta')  \delta'\ .
\end{cases}
$$
These maps induce new maps
$$
\begin{cases}
\delta: \CC \xrightarrow{\delta'} (\mbk \cdot s\mu) \otimes_{\mbs_2} (\CC \otimes \CC) \to \CC \otimes \CC\ ,\\
\theta: \CC \xrightarrow{\theta'} \mbk \cdot s\xi \to \mbk\ , \\ 
d=d': \CC \to \CC\ ,
\end{cases}
$$
where the degree $-1$ map $(\mbk \cdot s\mu) \otimes_{\mbs_2} (\CC \otimes \CC) \to \CC \otimes \CC$ sends $s\mu\otimes_{\mbs_2} (x \otimes y)$ to $\frac{1}{2}(x \otimes y + (-1)^{|x||y|} y \otimes x)$. Then, for any $x \in \CC$
$$
\delta'(x)= s\mu \otimes_{\mbs_2} \delta(x)\ .
$$
We know from \cite[\S 7.6.3]{LodayVallette12} that the $\mbk$-module $\Tfree(s\mu)(3)$ has three generators $\nu_I$, $\nu_{II}$ and $\nu_{III}$ which are obtained from the composite $s\mu \otimes_{\mbs_2} (s\mu \otimes 1)$ by applying respectively the permutations $Id \in \mbs_3$, $(2,3)$ and $(1,3)$. Moreover, $s^2\RR (3)$ is spanned by $\nu_I- \nu_{II}$ and $\nu_I - \nu_{III}$. Besides, $\mbk \cdot (\nu_{I} + \nu_{II} + \nu_{III})$ is a complementary sub-$\mbk[\mbs_3]$-module of $s^2\RR (3)$ in $\Tfree(s\mu)(3)$. Let us denote by $\pi$ the projection of $\Tfree(s\mu)(3)$ onto $\mbk \cdot (\nu_{I} + \nu_{II} + \nu_{III})$ along $s^2\RR (3)$. Since the action of the group $\mbs_2$ on $s\mu$ is trivial, we have for any $x \in \CC$
$$
(Id_{s\mu} \circ' \delta') \delta' (x) = 2\nu_I \otimes_{\mbs_3} \left((\delta \otimes Id) \delta (x)\right)\ .
$$
Then
$$
(\pi \circ Id)(Id_{s\mu} \circ' \delta') \delta' (x) = \frac{2}{3} (\nu_I +  \nu_{II} + \nu_{III})\otimes_{\mbs_3} \left((\delta \otimes Id) \delta (x)\right)\ .
$$
The above condition (\ref{equationkoszullie}) is equivalent to the fact that $(\nu_I +  \nu_{II} + \nu_{III})\otimes_{\mbs_3} \left((\delta \otimes Id) \delta (x)\right)$ is zero which is equivalent to 
$$
(\delta \otimes Id) \delta (x) + \left((\delta \otimes Id) \delta (x)\right)^{(2,3)} + \left((\delta \otimes Id) \delta (x)\right)^{(1,3)}=0\ .
$$
The other conditions are equivalent to the following
$$
\begin{cases}
 \delta d= -(d \otimes Id + Id \otimes d ) \delta\ ,\\
\theta d=0\ , \\ 
d^2= -(\theta\otimes  Id + Id \otimes \theta)\delta\ .
\end{cases}
$$
Conversely, from the maps $\delta$, $\theta$ and $d$, one can reconstruct $\delta'$, $\theta'$ and $d'$ in the obvious way.
\end{proof}

\begin{proof}[Proof of Proposition \ref{propkoszulduallie}]
We show that the category described in Lemma \ref{lemma:Liecog} is equivalent to the category of curved conilpotent Lie coalgebras. Let $\CCC=(\CC, \delta, \theta,d)$ be a curved conilpotent Lie coalgebra. Then, we can define the maps $(\delta',\theta',d')$ on $s^{-1} \CC$ where $\delta'$ is the composite
\begin{align*}
 s^{-1} \CC \simeq \mbk \cdot s^{-1} \otimes \CC &\to  \mbk \cdot s^{-1} \otimes \mbk \cdot s^{-1} \otimes \CC \otimes \CC \simeq \mbk \cdot s^{-1} \otimes \CC \otimes \mbk \cdot s^{-1} \otimes \CC \\
 s^{-1} \otimes x &\mapsto s^{-1} \otimes s^{-1} \otimes \delta (x)\ ,
\end{align*}
and where $\theta' (s^{-1}x)=\theta( x)$ and $d'(s^{-1}x)=-s^{-1}dx$ for any $x \in \CC$. It is straightforward to prove that these maps satisfy the conditions of Lemma \ref{lemma:Liecog}. Conversely, from a graded $\mbk$-module $\DD$ and maps $(\delta,\theta, d)$ as in Lemma \ref{lemma:Liecog}, one can build a structure of curved conilpotent Lie coalgebra $(\delta',\theta',d')$ on $s\DD$, where $\delta'$ is the composite
\begin{align*}
 s \DD \simeq \mbk \cdot s \otimes \DD &\to  \mbk \cdot s \otimes \mbk \cdot s \otimes \DD \otimes \DD \simeq \mbk \cdot s \otimes \DD \otimes \mbk \cdot s \otimes \DD \\
 s \otimes x &\mapsto - s \otimes s \otimes \delta (x)\ ,
\end{align*}
and where $\theta' (sx)= \theta (x)$ and $d'(sx)=-s dx$ for any $x \in \DD$. It is again straightforward to prove that these maps define actually a structure of curved conilpotent Lie coalgebra. Moreover, these two constructions are inverse one to another.
\end{proof}

\subsubsection{The bar-cobar adjunction}
If we compose the bar-cobar adjunction between $\uCom$-algebras and $\uCom^\ac$-coalgebras with the equivalence between $\uCom^\ac$-coalgebras and curved conilpotent Lie coalgebras, then we obtain an adjunction $\Omega_C \dashv B_L$ between unital commutative algebras and curved conilpotent Lie coalgebras which is as follows.\\

\begin{defin}[Curved Lie bar construction]
Let $\AAA=(\Aa,\gamma_\AAA,1)$ be a unital commutative algebra. Its \textit{curved Lie bar construction} $B_L  (\AAA)$ is the following curved conilpotent Lie coalgebra. The underlying graded Lie coalgebra of $B_L (\AAA)$ is 
$$
B_L  (\Aa) := \Lie^c \circ (s \Aa \oplus \mbk \cdot v)
$$
where $\Lie^c$ denotes the Lie cooperad which is the linear dual of the Lie operad and where $|v|=2$. The coderivation of $B_L(\AAA)$ extends the map
\begin{align*}
 \Lie^c (s\Aa \oplus \mbk v) \twoheadrightarrow s\Aa \wedge s\Aa \oplus s\Aa \oplus \mbk v &\ra s\Aa\\
 sx \wedge  sy &\mapsto (-1)^{|x|}s \gamma_\Aa (x \otimes y) \\
 v & \mapsto s1\\
 sx & \mapsto -sdx\ .
\end{align*}
The curvature is the map
\begin{align*}
 \Lie^c (s \Aa \oplus \mbk \cdot v) \twoheadrightarrow \mbk \cdot v &\ra \mbk\\
 v & \mapsto 1\ .
\end{align*} 
\end{defin}

\begin{defin}[Unital commutative cobar construction]
Let $\CCC=(\CC,\delta, d_\CC, \theta)$ be a curved Lie coalgebra. Its \textit{unital commutative cobar construction} $\Omega_C (\CCC)$ is the free unital commutative algebra
$$
\Omega_C \CC := S(s^{-1} \CC):= \bigoplus_{n \in \mbn} (s^{-1}\CC)^{\otimes n}/\mbs_n\ ,
$$
whose coderivation extends the map
\begin{align*}
 s^{-1} \CC &\ra S(s^{-1} \CC)\\
 s^{-1}x &\mapsto \theta(x) 1 -s^{-1}d_\CC x - \sum (-1)^{|x_1|} s^{-1} x_1 \otimes_{\mbs_2} s^{-1} x_2\ ,
\end{align*}
where $\sum x_1 \wedge x_2 = \delta (x)$. 
\end{defin}

\begin{defin}[Twisting morphisms]
A twisting morphism from a curved conilpotent Lie coalgebra $\CCC$ to a unital commutative algebra $\AAA$ is a degree $-1$ map $\alpha: \CC \ra \Aa$ such that
$$
\partial \alpha + \gamma_\AAA (\alpha \otimes \alpha) \delta_\CCC = \theta (-)1_\AAA\ .
$$
We denote by $\Tw_L (\CCC,\AAA)$ the set of twisting morphisms from $\CCC$ to $\AAA$.
\end{defin}

\begin{prop}
We have functorial isomorphisms
$$
\hom_{\uCom-\mathsf{alg}}(\Omega_C \CCC, \AAA) \simeq \Tw_L (\CCC,\AAA) \simeq \hom_{\cLieCog}(\CCC, B_L\AAA)
$$
for any unital commutative algebra $\AAA$ and any curved conilpotent Lie coalgebra $\CCC$.
\end{prop}

\begin{proof}
 The proof uses the same arguments as the proof of Proposition \ref{prop:barcobarcurvedaslv}.
\end{proof}

The adjunction $\Omega_C \dashv B_L$ is part of a larger picture
$$
\xymatrix{\cLieCog \ar@<1ex>[r]^(0.42){\Omega_\iota} & \uCom_\infty-\mathsf{alg} \ar@<1ex>[l]^(0.58){B_\iota} \ar@<1ex>[r]^(0.5){\psi_!} & \uCom-\mathsf{alg}\ , \ar@<1ex>[l]^(0.5){\psi^*}}
$$
where the adjunction $\psi_! \dashv \psi^*$ is induced by the morphism of operads $\psi: \uCom_\infty \to \uCom$ and where $\Omega_C = \psi_!\Omega_\iota$ and $B_L= B_\iota \psi^*$. We know that a $\uCom_\infty$-algebra $\AAA=(\Aa,\gamma)$ is the data of a chain complex $\Aa$ together with a degree $-1$ map
$$
\gamma: \Lie^c (s\Aa \oplus \mbk \cdot v) \to s\Aa\ .
$$
such that the coderivation of the curved Lie coalgebra $\Lie^c (s\Aa \oplus \mbk \cdot v)$ which extends $\gamma$ squares to $(\theta \otimes Id) \delta$, where $\theta$ is given by
\begin{align*}
 \Lie^c (s \Aa \oplus \mbk \cdot v) \twoheadrightarrow \mbk \cdot v &\ra \mbk\\
 v & \mapsto 1\ .
\end{align*}
 In particular, we have the following.
\begin{itemize}
 \itemt A degree zero symmetric product
 $$
 \gamma_2: \Aa \otimes \Aa \to \Aa\ .
 $$
 \itemt A degree $1$ map
 $$
 \gamma_{I,II}: \Aa\otimes \Aa \otimes \Aa \to \Aa
 $$ 
 whose boundary is the associator of $\gamma_2$, that is 
 $$
 \partial (\gamma_{I,II})= \gamma_2 (Id \otimes \gamma_2) - \gamma_2 (\gamma_2 \otimes Id)\ .
 $$
 \itemt A degree $1$ map
 $$
 \gamma_{I,III}: \Aa\otimes \Aa \otimes \Aa \to \Aa
 $$ 
 whose boundary is 
 $$
 \partial (\gamma_{I,III})= \gamma_2 (Id \otimes \gamma_2) -\gamma_2 (Id \otimes \gamma_2)(\tau \otimes Id)\ .
 $$
 \itemt An element $1_\AAA$ defined by $\gamma (v) = s 1_\AAA$.
 \itemt A degree $1$ map $\gamma_{u}: \Aa \to \Aa$ which makes $1_\AAA$ a unit up to homotopy:
 $$
 \partial (\gamma_{u}) = \gamma_2 (1_\AAA \otimes Id) -Id\ .
 $$
\end{itemize}

\subsubsection{The Koszul property and the infinity category of $u\mathscr{C}om_\infty$-algebras}

\begin{thm}
 The operad $\uCom$ is Koszul.
\end{thm}

\begin{proof}[Proof]
We know from \cite{HirshMilles12} that $q\uCom^\ac  \simeq \Com^\ac \circ (\II \oplus \mbk \cdot s\xi)$. So, we have:
$$
q\uCom \circ q\uCom^\ac  \simeq \mbk \cdot \xi \oplus  \Com \circ \Com^\ac \circ (\II \oplus \mbk \cdot s\xi)\ .
$$ 
We can filter $q\uCom \circ_\kappa q\uCom^\ac$ by the number of $\xi$ and $s\xi$ appearing in the trees. Then, the induced graded complex have the following form:
$$
G(q\uCom \circ_\kappa q\uCom^\ac)\simeq \mbk \cdot \xi \oplus  (\Com \circ_\kappa \Com^\ac) \circ (\II \oplus \mbk \cdot s\xi)\ .
$$
We already know by \cite[Theorems 7.4.6 and 13.1.7]{LodayVallette12} that the the canonical morphism $\Com \circ_\kappa \Com^\ac \to \II$ is a weak equivalence. Then, the map $G(q\uCom \circ_\kappa q\uCom^\ac) \to \II$ may be decomposed as follows.
$$
G(q\uCom \circ_\kappa q\uCom^\ac) \simeq \mbk \cdot \xi \oplus  (\Com \circ_\kappa \Com^\ac) \circ (\II \oplus \mbk \cdot s\xi) \to \II \oplus \mbk \cdot \xi \oplus \mbk \cdot s\xi \to \II\ .
$$
All the maps of this composition are quasi-isomorphisms. So, by Theorem \ref{maclane-homology}, the canonical map $q\uCom \circ_\kappa q\uCom^\ac \to \II$ is a quasi-isomorphism. We conclude by Theorem \ref{thm:hm}.
\end{proof}

There are several ways to describe the infinity category of $\uCom$-algebras.
\begin{itemize}
 \itemt One can take the Dwyer--Kan simplicial localization of the category of $\uCom$-algebras with respect to quasi-isomorphisms as described in \cite{DwyerKan80a} and \cite{DwyerKan80b}.
 \itemt One can take the simplicial category whose objects are cofibrant-fibrant $\uCom$-algebras and whose spaces of morphisms are 
 $$
\Map(\AAA,\BBB)_n := \HOM_{\uCom-\mathsf{alg}}(\AAA,\BBB )\ .
 $$
 \itemt One can also take the simplicial category whose objects are all $\uCom$-algebras and whose spaces of morphisms are
 $$
 \Map(\AAA,\BBB):= \HOM_{\uCom-\mathsf{alg}}(\Omega_C B_L \AAA,\BBB )\simeq \HOM_{\cLieCog}(B_L \AAA,B_L \BBB )\ .
 $$
\end{itemize}

\begin{prop}
 The three simplicial categories described above are equivalent. Moreover, the two last ones are fibrant in the sense that for any two objects $\AAA$ and $\BBB$, the simplicial set $\Map(\AAA,\BBB)$ is a Kan complex.
\end{prop}

\begin{proof}
The proof uses the same arguments as the proof of Proposition \ref{uasinfinitycat}. 
\end{proof}

\vspace{1cm}
\appendix
\section*{Appendix}

The purpose of this appendix is to describe the category of dg counital cocommutative coalgebras over an algebraically closed field of characteristic zero in the vein of the article \cite{ChuangLazarevMannan14}. In the sequel, dg counital cocommutative coalgebras are simply called cocommutative coalgebras. We suppose that the base field $\mbk$ is algebraically closed field and of characteristic zero.

\begin{rmk}
The characteristic zero assumption is needed in the theorem 2.9 of the article \cite{ChuangLazarevMannan14}.
\end{rmk}

We know that the linear dual of a cocommutative coalgebra is a commutative algebra. Moreover, for any cocommutative coalgebra $\CCC$, the sub-coalgebras of $\CCC$ are in correspondence with the ideals of $\CCC^*$.

\begin{defin}[Orthogonal ideals and sub-coalgebras]
 Let $\DDD=(\DD,\Delta,\epsilon)$ be a sub-coalgebra of $\CCC$. The orthogonal of $\DDD$ is the sub-chain complex $
\DDD^\perp:= \{f \in \CCC^*|\ \forall x\in \DD,\ f(x)=0\} \subset \CCC^*
$ which is an ideal of $\CCC^*$. Let $I$ be an ideal of the commutative algebra $\CCC^*$. The orthogonal of $I$ is the sub-chain complex $
I^\perp:= \{x \in \CCC|\ \forall f\in I,\ f(x)=0\}\ \subset \CCC
$ which is a sub-coalgebra of $\CCC$.
\end{defin}

\begin{defin}[Pseudo-compact algebras]
A pseudo-compact algebra is a dg unital commutative algebra $\AAA$ together with a set $\{I_u\}_{u \in U}$ of ideals of finite codimension, which is stable finite intersections and such that
$$
\AAA \simeq \lim \AAA/I_u\ .
$$
A morphism of pseudo-compacts algebras from $(\AAA, \{I_u\}_{u\in U})$ to $(\BBB, \{J_v\}_{v\in V})$ is a morphism of algebras $f: \AAA \to \BBB$ which is continuous with respect to the induced topologies, that is such that for any $v \in V$, there exists an $u \in U$ such that the composite morphism $\AAA \to \BBB \to \BBB/ J_v$ factors through $\AAA \to \AAA/I_u$. A pseudo-compact algebra $\AAA$ is called local if its underlying graded algebra is local.
\end{defin}

\begin{prop}
 The linear dual of a cocommutative coalgebra is a pseudo-compact algebra. Moreover, the linear dual functor is an antiequivalence between the category of cocommutative coalgebras and the category of pseudo-compact algebras.
\end{prop}

\begin{proof}
 It is clear that linear duality induces an antiequivalence between finite dimensional cocommutative coalgebras and finite dimensional commutative algebras. The rest is a consequence of the following proposition \ref{prop:indfinitecoalg}.
\end{proof}

\begin{prop}\cite{GetzlerGoerss99}\label{prop:indfinitecoalg}
 Let $\CCC$ be a cocommutative coalgebra and let $x$ be an element of $\CCC$. There exists a finite dimensional sub-coalgebra of $\CCC$ which contains $x$. Then, $\CCC$ is the colimit of the filtered diagram of its finite dimensional sub-coalgebras.
\end{prop}

Chuang, Lazarev and Mannan showed that any pseudo-compact algebra can be decomposed into a product of local pseudo-compact algebras.

\begin{thm}\cite[2.9]{ChuangLazarevMannan14}
Any pseudo-compact algebra $\AAA$ is isomorphic to the product of local pseudo-compact algebras $\AAA \simeq \prod_{i \in I} \AAA_i$. Moreover, a morphism of products of local pseudo-compact algebras $f: \prod_{i \in I} \AAA_i \to \prod_{j \in J} \BBB_j$ is the data of a function $\phi:J \to I$ and a morphism $f_j:\AAA_{\phi(j)} \to \BBB_j$ for any $j \in J$, where $\pi_j f = f_j\pi_{\phi(j)}$ (here $\pi_j$ and $\pi_{\phi(j)}$ denote respectively the projection of $\prod_{j \in J} \BBB_j$ onto $\BBB_j$ and the projection of $\prod_{i \in I} \AAA_i$ onto $\AAA_{\phi(j)}$).
\end{thm}

We show that local pseudo-compact algebras are linear duals of conilpotent cocommutative coalgebras.

\begin{defin}[Irreducible coalgebras]
 A nonzero graded cocommutative coalgebra is said to be \textit{irreducible} if any two nonzero sub-coalgebras have a nonzero intersection.
\end{defin}

\begin{prop}\label{prop:irrlocal}
 A graded cocommutative coalgebra is irreducible if and only if its dual algebra is local.
\end{prop}

\begin{proof}
 Let $\CCC=(\CC,\Delta, \epsilon)$ be a graded cocommutative coalgebra. We first suppose that it is irreducible. Let $M_1$ and $M_2$ be two maximal ideals of the commutative algebra $\CCC^*$. Since, $\CCC$ is irreducible, then the sub-coalgebras $M_1^\perp$ and $M_2^\perp$ have a nonzero intersection. So $M_1 +M_2 \subset (M_1^\perp \cap M_2^\perp)^\perp$ is a proper ideal. Since $M_1$ and $M_2$ are maximal ideals, then $M_1=M_1+M_2=M_2$. So $\CCC^*$ is local. Conversely, suppose that $\CCC^*$ is local. We denote by $M$ its maximal ideal. By Lemma \ref{lemma:maxideal}, $M$ is the kernel of an augmentation $\CCC^* \to \mbk$. By the antiequivalence between pseudo-compact algebras and cocommutative coalgebras, we obtain a morphism of coalgebras $\mbk \to \CCC$, that is an atom $a$ of $\CCC$. For any nonzero sub-coalgebra $\DDD$ of $\CCC$, the orthogonal $\DDD^\perp$ is contained in $M$. Thus $\mbk \cdot a = M^\perp \subset (\DDD^\perp)^\perp=\DDD$. So any nonzero sub-coalgebra of $\CCC$ contains $a$. Subsequently, $\CCC$ is irreducible.
\end{proof}

\begin{lemma}\label{lemma:maxideal}
 Let $\AAA$ be a graded local pseudo-compact algebra. Then, the maximal ideal $M$ of $\AAA$ is the kernel of an augmentation $\AAA \to \mbk$.
\end{lemma}

\begin{proof}
 Since $\AAA= (\Aa,\gamma_\AAA,1)$ is the inverse limit of finite dimensional algebras and since $M$ is maximal, then $M$ is the kernel of a surjection $\AAA \to \BBB$ where $\BBB=(\BB,\gamma_\BB,1)$ is a finite dimensional commutative algebra. Since $M$ is maximal, then any nonzero element of $\BB$ is invertible. Since, the elements in nonzero degrees are nilpotent, then $\BB$ is concentrated in degree zero. So $\BB$ is a finite dimensional field extension of $\mbk$. Finally, $\BB \simeq \mbk$ because $\mbk$ is an algebraically closed field. 
\end{proof}

\begin{cor}
A graded cocommutative coalgebra is irreducible if and only if it contains a single atom.
\end{cor}

\begin{proof}
 It is a direct consequence of Proposition \ref{prop:irrlocal}.
\end{proof}

\begin{prop}\label{prop:irrconil}
 Irreducible graded cocommutative coalgebras are conilpotent graded cocommutative coalgebras.
\end{prop}

\begin{proof}
Let $\CCC=(\CC,\Delta,\epsilon)$ be an irreducible graded cocommutative coalgebra. Let $x$ be an element of $\CC$ and let $\DDD=(\DD,\Delta,\epsilon)$ be a finite dimensional sub-coalgebra of $\CCC$ which contains $x$. The commutative algebra $\DDD^*$ is local; its maximal ideal is $M := {\ov \DD}^*$. Then, $\DDD^*_0$ is also local with maximal ideal $M_0$. By Nakayama's lemma, $M_0$ is nilpotent. So, $M$ is nilpotent and so $\DDD$ is a conilpotent  cocommutative coalgebra.
\end{proof}

\begin{cor}
 The antiequivalence between the category of pseudo-compact algebras and the category $\uCocom$ of cocommutative coalgebras restricts to an antiequivalence between the category of local pseudo-compact algebras and the category $\uNilCocom$ of conilpotent cocommutative coalgebras. 
\end{cor}

\begin{proof}
 It is a direct consequence of Proposition \ref{prop:irrlocal} and Proposition \ref{prop:irrconil}.
\end{proof}

\begin{thmappendix*}\label{thm:decompcocom}
Let $\CCC= (\CC,\Delta,\epsilon)$ be a dg cocommutative coalgebra over an algebraically closed field of characteristic zero and let $A$ be its set of graded atoms. There exists a unique decomposition $\CCC \simeq \bigoplus_{a \in A} \CCC_a$ where $\CCC_a$ is a sub-coalgebra of $\CCC$ which contains $a$ and which belongs to the category $\uNilCocom$. Moreover, a morphism of dg cocommutative coalgebras $f: \bigoplus_{a \in A} \CCC_a \to \bigoplus_{b \in B} \DDD_b$ is the data of a function $\phi: A \to B$ and of a morphism $f_a: \CCC_a \to \DDD_{\phi(a)}$ for any $a \in A$. 
\end{thmappendix*}

\begin{proof}
The only point that needs to be cleared up is that, in the decomposition $\CCC= \bigoplus_{i \in I} \CCC_i$, the set $I$ is isomorphic to the set of graded atoms of $\CCC$. A graded atom of $\CCC$ is a morphism of graded cocommutative coalgebras from $\mbk$ to $\CCC$, that is a morphism of graded pseudo-compact algebras from $\prod_{i \in I} \CCC_i^*$ to $\mbk$. So it is the choice of an element of $I$.
\end{proof}


\bibliographystyle{amsalpha}
\bibliography{biblg}

\end{document}